\documentclass[oneside,english]{amsart}
\usepackage[T1]{fontenc}
\usepackage[latin9]{inputenc}
\usepackage{textcomp}
\usepackage{amsthm}
\usepackage{amstext}
\usepackage{esint}

\makeatletter

\newcommand{\noun}[1]{\textsc{#1}}

\numberwithin{equation}{section}
\numberwithin{figure}{section}
\theoremstyle{plain}
\newtheorem{thm}{Theorem}
  \theoremstyle{plain}
  \newtheorem{lem}[thm]{Lemma}
  \theoremstyle{plain}
  \newtheorem{prop}[thm]{Proposition}
  \theoremstyle{remark}
  \newtheorem{rem}[thm]{Remark}
  \theoremstyle{plain}
  \newtheorem{cor}[thm]{Corollary}
 \theoremstyle{definition}
 \newtheorem*{defn*}{Definition}

\def\makebbb#1{
    \expandafter\gdef\csname#1\endcsname{
        \ensuremath{\Bbb{#1}}}
}
\makebbb{R}
\makebbb{N}
\makebbb{Z}
\makebbb{C}
\makebbb{G}
\makebbb{Q}
\makebbb{H}
\makebbb{P}
\makebbb{B}
\makebbb{K}
\makebbb{E}

\makeatother

\usepackage{babel}

\begin{document}

\title{Relative Kähler-Ricci flows and their quantization}

\author{Robert J.Berman }
\begin{abstract}
Let $\pi:\mathcal{\, X}\rightarrow S$ be a holomorphic fibration
and let $\mathcal{L}$ be a relatively ample line bundle over $\mathcal{X}.$
We define relative Kähler-Ricci flows on the space of all Hermitian
metrics on $\mathcal{L}$ with relatively positive curvature and study
their convergence properties. Mainly three different settings are
investigated: the case when the fibers are Calabi-Yau manifolds and
the case when $\mathcal{L=}\pm K_{\mathcal{X}/S}$ is the relative
(anti-) canonical line bundle. The main theme studied is whether {}``positivity
in families'' is preserved under the flows and its relation to the
variation of the moduli of the complex structures of the fibers. The
{}``quantization'' of this setting is also studied, where the role
of the Kähler-Ricci flow is played by Donaldson's iteration on the
space of all Hermitian metrics on the finite rank vector bundle $\pi_{*}\mathcal{L}\rightarrow S.$
Applications to the construction of canonical metrics on the relative
canonical bundles of canonically polarized families and Weil-Petersson
geometry are given. Some of the main results are a parabolic analogue
of a recent elliptic equation of Schumacher and the convergence towards
the Kähler-Ricci flow of Donaldson's iteration in a certain double
scaling limit.
\end{abstract}

\address{Mathematical Sciences, Chalmers University of Technology and the
University of Gothenburg, SE-412 96 Göteborg, Sweden }

\email{robertb@chalmers.se}


\maketitle

\section{Introduction}

\subsection{Background}

On an $n-$dimensional Kähler manifold $(X,\omega_{0})$ Hamilton's
Ricci flow \cite{ha} on the space of Riemannian metrics on $X$ preserves
the Kähler condition of the initial metric and may be written as the\emph{
Kähler-Ricci flow} 

\begin{equation}
\frac{\partial\omega_{t}}{\partial t}=-\mbox{Ric}\omega_{t},\label{eq:ke eq for ric intr}\end{equation}
When $X$ is a Calabi-Yau manifold (which here will mean that the
canonical line bundle $K_{X}$ is holomorphically trivial) it was
shown by Cao \cite{ca} that the corresponding flow in the space of
Kähler metrics in $[\omega_{0}]\in H^{2}(X,\R)$ has a large time
limit. The limit is thus a fixed point of the flow which coincides
with the unique Ricci flat Kähler metric in $[\omega_{0}],$ whose
existence was first established by Yau \cite{yau} in his celebrated
proof of the Calabi conjecture. The non-Calabi-Yau cases when $[\omega_{0}]$
is the first Chern class $c_{1}(L)$ of $L=rK_{X},$ where $r=\pm1$
have also been studied extensively (where $-r\omega$ is added to
the right hand side in equation \ref{eq:ke eq for ric intr}). In
general the fixed points of the corresponding Kähler-Ricci flows are
hence Kähler-Einstein metrics of negative ($r=1)$ and positive $(r=-1)$
scalar curvature. The convergence towards a fixed point - when it
exists - in the latter positive case (i.e. $X$ is a Fano manifold
) was only established very recently by Perelman and Tian-Zhu \cite{t-z}. 

A distinctive feature of Kähler geometry is that a Kähler metric $\omega$
may be locally described in terms of a local function $\phi,$ such
that $\omega=dd^{c}\phi,$ where $\phi$ is determined up to an additive
constant. In the integral case, i.e. when $[\omega_{0}]=c_{1}(L)$
is the first Chern class of an ample line bundle $L\rightarrow X$
this just amounts to the global fact that the space of Kähler metrics
$\omega$ in $c_{1}(L)$ may be identified with the space \emph{$\mathcal{H}_{L}$}
of smooth metrics $h$ on the line bundle $L$ with positive curvature
form $\omega,$ modulo the action of $\R$ on \emph{$\mathcal{H}_{L}$}
by scalings. Locally, $h=e^{-\phi}$ and we will refer to the additive
object $\phi$ as a\emph{ weight }on $L$ (see section \ref{sub:The-weight-notation}).
In this notation the Kähler-Einstein equations may be expressed as
Monge-Ampère equations on $\mathcal{H}_{L}.$ For example, on a Calabi-Yau
manifold $\omega_{\phi}:=dd^{c}\phi$ is Ricci flat precisely when
\begin{equation}
(dd^{c}\phi)^{n}/n!=\mu,\label{eq:inhome ms intr}\end{equation}
 where $\mu$ is the canonical probability measure on $X$ such that
$\mu=i^{n^{2}}\Omega\wedge\bar{\Omega},$ for $\Omega$ a suitable
global holomorphic $n-$form trivializing $K_{X}$ (to simplify the
notion we will in the following always assume that the volume of the
given class $[\omega_{0}]$ is equal to one, so that $\omega_{0}^{n}/n!$
defines a probability measure on $X$ for any $\omega\in[\omega_{0}]).$
By letting $\mu$ depend on $\phi$ in a suitable way general Kähler-Einstein
are obtained. 

As emphasized by Yau \cite{yau2} one can expect to obtain approximations
to Kähler-Einstein metrics by using holomorphic sections of high powers
of a line bundle. In this direction Donaldson recently introduced
certain iterations on the {}``quantization'' (at level $k)$ of
the space $\mathcal{H}_{L}$ of Kähler metrics in $c_{1}(L)$ \cite{do3}.
Geometrically, this quantized space, denoted by $\mathcal{H}^{(k)},$
is the space of all Hermitian metrics on the finite dimensional vector
space $H^{0}(X,kL)$ of global holomorphic sections of $kL,$ where
$kL$ denotes the $k$th tensor power of $L,$ in our additive notation
(for the definition see section \ref{sec:Quantization:-The-Bergman general}).
In other words $\mathcal{H}^{(k)}$ can be identified with the symmetric
space $GL(N_{k},\C)/U(N_{k})$ of $N_{k}\times N_{k}$ Hermitian matrices
which in turn, using projective embeddings, corresponds to the space
of level $k$ Bergman metrics on $L.$ The fixed points of Donaldson's
iteration are called \emph{balanced} metrics at level $k$ (with respect
to $\mu)$ and they first appeared in the previous work of Bourguignon-Li-Yau
\cite{b-l-y}. Again, in the $\pm K_{X}-$setting one lets $\mu$
depend on $\phi$ in a suitable way leading to different settings
(see below). In the limit when $L$ is replaced by a large tensor
power it has very recently been shown that balanced metrics in the
different settings indeed converge to Kähler-Einstein metrics \cite{wa,ke,bbgz}.
It was pointed out by Donaldson in \cite{do3} that it seems likely
that these iterations can be viewed as discrete approximations of
the Ricci flow. This will be made precise and confirmed in the present
paper (Theorem \ref{thm:conv of bergman iter to ricci cy} and Theorem
\ref{thm:conv of bergman to ricci}).

\subsection{Outline of the present setting and the main results}

The aim of the present paper is to study\emph{ relative }versions
of the Kähler-Ricci flow and Donaldson's iteration (in the various
settings) and investigate whether {}``positivity in families'' is
preserved under the flows. In other words, the given geometric setting
is that of a holomorphic fibration $\pi:\,\mathcal{X}\rightarrow S$
of relative dimension $n$ and a relatively ample line bundle $\mathcal{L}\rightarrow\mathcal{X}.$
The fibration will mainly be assumed to be a proper submersion over
a connected base, so that all fibers are diffeomorphic (for general
quasi-projective morphisms see section \ref{sub:fam gen type}). Denote
by $\mathcal{H}_{\mathcal{L}/S}$ the space of all metrics on $\mathcal{L}$
which are fiber-wise of positive curvature. In other words, $\mathcal{H}_{\mathcal{L}/S}$
is an infinite dimensional fiber bundle over $S$ whose fibers are
of the form $\mathcal{H}_{L},$ as in the previous section. The \emph{relative
Kähler-Ricci flows} are now defined as suitable flows on $\mathcal{H}_{\mathcal{L}/S}$
such that the induced flow of curvature forms restricts to the usual
Kähler-Ricci flow fiber-wise: We will say that {}``positivity is
preserved under the flow'' if, for any initial metric with positive
curvature (in\emph{ all} directions on $\mathcal{X}),$ the evolved
metric also has positive curvature for all times, i.e. the flow induces
a flow of Kähler forms on the total space $\mathcal{X}$ of the fibration
(and not only along the fibers).

\subsubsection*{The Calabi-Yau setting}

Let us first summarize the main results in the setting when the fibers
are Calabi Yau. In this case the flow $\phi_{t}$ in $\mathcal{H}_{\mathcal{L}/S}$
is defined fiber-wise by \begin{equation}
\frac{\partial\phi_{t}}{\partial t}=\log(\frac{(dd^{c}\phi_{t})^{n}/n!}{\mu}),\label{eq:k-r flow of weights in intro}\end{equation}
 with $\mu$ a measure as in equation \ref{eq:inhome ms intr}. Of
course, adding the pull-back of a time-dependent function on the base
$S$ to the right hand side of the previous equation does not alter
the induced flows of the\emph{ fiber-wis}e\emph{ restricted }Kähler
forms $d_{X}d_{X}^{c}\phi_{t}$ , but it certainly effects the flow
of $dd^{c}\phi_{t}$ on $\mathcal{X}$ which will typically \emph{not}
preserve the initial Kähler property. 

One of the main results of the present paper is a parabolic evolution
equation along the flow \ref{eq:k-r flow of weights in intro} for
the function \[
c(\phi):=c(\phi):=\frac{1}{n}(dd^{c}\phi)^{n+1}/(d_{X}d_{X}^{c}\phi)^{n}\wedge ids\wedge d\bar{s}\]
 on $\mathcal{X}$ which is well-defined when $S$ is embedded in
$\C.$ The point is that $c(\phi)>0$ precisely when $dd^{c}\phi>0$
on $\mathcal{X}.$ The evolution equation for $c(\phi_{t})$ reads
(Theorem \ref{thm:heat eq for c in c-y}) \begin{equation}
(\frac{\partial}{\partial t}-\Delta_{\omega_{t}^{X}})c(\phi_{t})=|A_{\omega_{t}}|_{\omega_{t}^{X}}^{2}-\omega_{WP},\label{eq:parab eq for c intro}\end{equation}
 where $\omega_{t}^{X}$ denotes the flow of the fiber-wise restricted
curvature forms, $A_{\omega_{t}}$ is a certain representative of
the Kodaira-Spencer class of the fiber $\mathcal{X}_{s}$ and $\omega_{WP}$
is the pull-back to $\mathcal{X}$ of the (generalized) Weil-Petersson
 form on the base $S;$ by a result of Tian \cite{ti0} and Todorov
\cite{to} that we will reprove $\omega_{WP}$ can be represented
by the global squared $L^{2}-$norm of $A_{\omega_{KE}}$ for $\omega_{KE}$
the unique Ricci flat metric in $c_{1}(L).$ Applying the maximum
principle then gives (Corollary \ref{cor:conserv of pos along k-r})
that the initial condition $dd^{c}\phi_{0}>0$ implies that \begin{equation}
dd^{c}\phi_{t}>-t\omega_{WP}\label{eq:lower bd on pos in intro}\end{equation}
 (and similarly when the initial curvature is \emph{semi}-positive).
By its very definition $\omega_{WP}$ vanishes at $s$ precisely when
the infinitesimal deformation of the complex structure on the fibers
$\mathcal{X}_{s}$ (i.e. the Kodaira-Spencer class) vanishes at $s.$
Hence, if the fibration $\pi:\,\mathcal{X}\rightarrow S$ is holomorphically
trivial, then, by inequality \ref{eq:lower bd on pos in intro}, positivity
is indeed preserved along the flow. This latter situation appears
naturally in Kähler geometry. Indeed, if the base $S$ is an annulus
in $\C$ and $\phi_{s}$ is rotationally invariant, then $\phi_{s}$
corresponds to a curve in $\mathcal{H}_{L}$ and $c(\phi_{s})$ is
then the geodesic curvature of the curve $\phi_{s}$ when $\mathcal{H}_{L}$
is equipped with its symmetric space Riemannian metric (see \cite{ch}
and references therein). In the non-normalized $K_{X}-$setting (see
section \ref{sec:The-(anti-)-canonical}) the equation \ref{eq:parab eq for c intro}
can be seen as a parabolic generalization of a very recent elliptic
equation of Schumacher \cite{sc}.

Similarly, the {}``quantized'' version of the previous setting is
studied, i.e. the relative version of Donaldson's iteration. It gives
an iteration on the space of all Hermitian metrics $H$ on the finite
rank vector bundle $\pi_{*}k\mathcal{L}\rightarrow S$ for any positive
integer $k$ (recall that the fiber of $\pi_{*}\mathcal{L}$ over
$s$ is, by definition, the space $H^{0}(\mathcal{X}_{s},\mathcal{L}_{s})$
of all global holomorphic sections on the fiber $\mathcal{X}_{s}$
with values in $\mathcal{L}_{|\mathcal{X}_{s}}).$ More precisely,
we will study the equivalent fiber-wise iteration $\phi_{m}^{(k)}$
in $\mathcal{H}_{\mathcal{L}/S}$ obtained by applying the (scaled)
Fubini-Study map to Donaldson's iteration. It will be called the \emph{relative
Bergman iteration at level $k.$ }When the discrete time $m$ tends
to infinity it is shown (Theorem \ref{thm:conv of bergman iter at level k })
that the iteration converges to a fiber-wise balanced weight: \[
\phi_{m}^{(k)}\rightarrow\phi_{\infty}^{(k)}\]
 in the $\mathcal{C}^{\infty}-$topology on $\mathcal{X}_{s}$, uniformly
with respect to $s.$ It is also observed that an analogue of the
inequality \ref{eq:lower bd on pos in intro} holds, i.e.\emph{ }\begin{equation}
dd^{c}\phi_{m}^{(k)}\geq-\frac{k}{m}\omega_{WP}.\label{eq:lower pos along bergman in intro}\end{equation}
 This turns out to be a simple consequence of a recent theorem of
Berndtsson \cite{bern0} about the curvature of vector bundles of
the form $\pi_{*}(\mathcal{L}+K_{\mathcal{X}/S}).$ We also confirm
Donaldson's expectation about the semi-classical limit when the level
$k$ tends to infinity. More precisely, it is shown that, in the double
scaling limit where $m/k\rightarrow t$ the (relative) Bergman iteration
at level $k$ approaches the (relative) Kähler-Ricci flow \ref{eq:k-r flow of weights in intro}:
\begin{equation}
\phi_{m}^{(k)}\rightarrow\phi_{t}\label{eq:doubl scal limit intro}\end{equation}
 uniformly on $\mathcal{X}.$ In particular, combining this convergence
with \ref{eq:lower pos along bergman in intro} gives an alternative
proof of the semi-positivity in the inequality \ref{eq:lower bd on pos in intro}.
Moreover, by taking $m=m_{k}$ such that $m/k\rightarrow\infty$ this
gives a dynamical construction of solutions to the inhomogeneous Monge-Ampère
equation \ref{eq:inhome ms intr} in the setting where $\mu$ is any
fixed volume form (Corollary \ref{cor:conv of balanced etc cy}).

\subsubsection*{The (anti-) canonical setting}

The previous results are also shown to have analogues in the setting
when the ample line bundle $\mathcal{L}$ is either the relative canonical
line bundle $K_{\mathcal{X}/S}$ over $\mathcal{X}$ or its dual,
which we write as $\mathcal{L=}\pm K_{\mathcal{X}/S}$ in our additive
notation. The starting point is the fact that any metric $h=e^{-\phi}$
on $\pm K_{X}$ induces, by the very definition of $K_{X},$ a volume
form on $X$ which may be written suggestively as $e^{\pm\phi}.$
The previous constructions, i.e. the relative Kähler-Ricci flows and
the Donaldson iteration, can then be repeated word for word for these
$\phi-$dependent measures $\mu=\mu(\phi)$. For example, the relative
Kähler-Ricci flows are are defined by \begin{equation}
\frac{\partial\phi_{t}}{\partial t}=\log(\frac{(dd^{c}\phi_{t})^{n}/n!}{e^{\pm\phi_{t}}}),\label{eq:intro non-norma flow in pm kx}\end{equation}
and we obtain (Theorem \ref{thm: heat eq in kx-setting}) a corresponding
parabolic equation for $c(\phi_{t}):$ \[
\left(\frac{\partial}{\partial t}-(\Delta_{\omega_{t}^{X}}-\pm1)\right)c(\phi_{t})=|A_{\omega_{t}}|_{\omega_{t}^{X}}^{2}\]
and as a consequence \emph{the flows always preserve positivity }(Corollary
\ref{cor:pos along flow in kx-s}) in these settings. In fact, in
the case of infinitesimally non-trivial fibration the flows will even
improve the positivity, i.e. any initial weight which is merely semi-positively
curved instantly becomes positively curved under the flows. In the
$+K_{X}-$setting the unique fixed point of the flow \ref{eq:intro non-norma flow in pm kx}
is the (fiber-wise) \emph{normalized} Kähler-Einstein weight uniquely
determined by 

\[
e^{-\phi_{KE}}=(\omega_{KE})^{n}/n!,\]
 where $\omega_{KE}$ is the unique Kähler-Einstein metric on $X$
(Corollary \ref{cor:conv of non-normal flow}). The corresponding
elliptic equation for $c(\phi_{KE})$ was first obtained by Schumacher
\cite{sc} who used it to deduce the following interesting result:
$\phi_{KE}$ is always semi-positively curved on the total space of
$\mathcal{X}$ and strictly positively curved for an infinitesimally
non-trivial fibration. As a consequence he obtained several applications
to the geometry of moduli spaces. For example, applied to the case
when $\mathcal{X}\rightarrow S$ is the universal curve over the Teichmuller
space of Riemann surfaces of genus $g\geq2$ it gives, when combined
with Berndtsson's Theorem \ref{thm:(Berndtsson).-Let-}, a new proof
of the hyperbolicity result of Liu-Sun-Yau \cite{lsy} saying that
the curvature of the Weil-Petersson metric on the Teichmuller space
is dual Nakano positive.

In the $-K_{X}-$setting the relative Kähler-Ricci flow will diverge
for generic initial data. But using the convergence on the level of
Kähler forms, established by Perelman and Tian-Zhu, will show that
it in case the Fano manifold $X$ admits a unique positively curved
Kähler-Einstein metric $\omega_{KE},$ the flow does convergence to
a weight for $\omega_{KE}$ in the \emph{normalized}\emph{\noun{ $\pm K_{X}-$}}\emph{setting.
}This latter setting is simple\noun{ }obtained by normalizing the
volume forms $e^{\pm\phi}$ used above. 

We will also use the relative Bergman iteration to obtained a {}``quantized''
version of Schumacher's result: the canonical {}``semi-balanced''
metric at level $k$ on $K_{\mathcal{X}/S},$ which by definition
is fiber-wise normalized and balanced, is smooth with semi-positive
curvature on $\mathcal{X}$ (Corollary \ref{cor:pos of semi-bal when kx amp})
and strictly positively curved in the case of an infinitesimally non-trivial
fibration. As a consequence the semi-balanced metric gives an alternative
to the canonical metric on $kK_{\mathcal{X}/S}$ introduced by Narasimhan-Simha
(see \cite{ns,ka,ts2,b-p2} for positivity properties of this latter
metric). In section \ref{sub:fam gen type} some of the results concerning
the setting when $K_{X}$ is ample are generalized to projective fibrations
of varieties of general type (i.e. $K_{X}$ is merely big) and the
corresponding canonical semi-balanced metric is shown to have a positive
curvature current (Theorem \ref{thm:pos of semi-bal on v general type}).
Relations to deformation invariance of plurigenera \cite{si} are
also briefly discussed.

\subsection{Further relations to previous results}

A variant of Donaldson iteration (but with a single parameter $k)$
in the $K_{X}-$setting was introduced by Tsuji \cite{ts}. He proved
convergence in the $L^{1}-$topology towards the normalized Kähler-Einstein
weight $\phi_{KE}$ in the large $k-$limit (see \cite{s-w} for a
proof of uniform convergence) and deduced the \emph{semi}-positivity
result for $\phi_{KE}$ of Schumacher referred to above. These works
of Tsuji and Schumacher provided an important motivation for the present
one. Steve Zelditch has also informed the author of a joint work in
progress with Jian Song, where they show that the linearization of
Tsuji's iteration at the fixed point coincides with the linearization
of the Kähler-Ricci flow. It should also be pointed out that another
discretization of the Kähler-Ricci flow on a Fano manifold was studied
by Rubinstein \cite{ru} and Keller \cite{ke}.

The $C^{0}-$convergence of the Bergman iteration at a fixed level
$k$ in the Calabi-Yau setting (or more generally in the setting of
a fixed measure $\phi)$ was pointed out by Donaldson in \cite{do3}
and the proof was sketched. Sano provided an explicit proof in the
constant scalar curvature setting \cite{sa} (see section \ref{sub:Comparison-with-the}) 

It is also interesting to compare with the very recent work of Fine
\cite{fi} concerning the constant scalar curvature setting. He shows
that a continuous version of Donaldson's iteration in this latter
setting, called balancing flows, converges to the Calabi flow, when
the latter flow exists. Julien Keller and Xiaodong Cao have informed
the author of a joint work in progress where an analogue of Fine's
balancing flows in the Calabi-Yau setting (or more generally in the
setting of a fixed volume form $\mu$) is shown to converge to a flow
on metrics, which however is different than the Kähler-Ricci flow.

There are also, at last tangential, relations to the work of Gross-Wilson
\cite{g-w}, where fibrations with Calabi-Yau fibers are considered.
In particular, they construct certain semi-flat Kähler metrics $\omega$
on the fibration $\mathcal{X},$ i.e. $\omega$ is fiber-wise Ricci
flat. Such metrics first appeared in the string theory literature
\cite{g-s-v-y}. In this terminology the inequality \ref{eq:lower bd on pos in intro}
shows that the relative Kähler-Ricci flow deforms any given Kähler
metric to a semi flat one, when there is no variation of the moduli
of the complex structure of the fibers. More generally, this latter
statement holds in a double scaling limit when the variation of the
complex structure is very small in the sense that $\omega_{FS}(s_{t})t\rightarrow0$
as $t\rightarrow\infty.$ 

A Kähler-Ricci flow on \emph{compact} fibrations $\mathcal{X}$ with
Calabi-Yau fibers was also considered very recently by Song-Tian.
But they consider the usual (i.e. non-relative) Kähler-Ricci flow
(with $r=1)$ when the canonical line bundle is only semi-ample and
relatively trivial (i.e. the base $S$ is the canonical model of $\mathcal{X}).$
They prove that the flow collapses the fibers so that the limit is
the pull-back of metric on the base $S$ solving a {}``twisted''
Kähler-Einstein equation where the twist is described by the (generalized)
 Weil-Petersson  form  $\omega_{FS}.$

\subsubsection*{Acknowledgments}

I would like to thank Bo Berndtsson for interesting conversations
on the topic of the present paper, in particular in connection to
\cite{bern0,bern2}. In the first ArXiv versions (1 and 2) of the
present paper the very recent result about\emph{ strict }positivity
of direct image bundles in \cite{bern2} was not available, but in
the present version it is used to improve the positivity properties
of the Bergman iteration in the case of infinitesimally non-trivial
fibrations. Also thanks to Valentino Tosatti and Yanir Rubinstein
for comments on a draft of the present paper (in particular thanks
to Valentino Tosatti for pointing out the reference \cite{g-w} in
connection to semi-flat metrics and the references in remark \ref{rem:div of flow},
which was was stated inaccurately in the previous version).

\subsection{Organization of the paper}

In section \ref{sec:The-general-setting} a general setting is introduced
and the associated relative Kähler-Ricci flow and its quantization
are defined. General convergence criteria for the flows are given.
In the following two sections the general setting is applied to get
convergence results in particular settings of geometric relevance:
the Calabi-Yau setting (section \ref{sec:The-Calabi-Yau-setting})
and the (anti-) canonical setting \ref{sec:The-(anti-)-canonical},
respectively. The new feature of these convergence results for the
Kähler-Ricci flows is that the convergence takes place on on the level
of weights, i.e. for the\emph{ potentials} of the evolving Kähler
metrics. Furthermore, the main question whether {}``positivity in
families'' is preserved under the flows is studied in these two sections
and relations to Weil-Petersson geometry are also discussed. It is
also shown that the quantized flows converge to Kähler-Ricci flows
in the large tensor power limit. Applications to canonical metrics
on relative canonical bundles are also given.

\tableofcontents{}

\section{\label{sec:The-general-setting}The general setting}

In this section we will consider a general setup that will subsequently
be applied to particular settings in sections \ref{sec:The-Calabi-Yau-setting},
\ref{sec:The-(anti-)-canonical}. 

We assume given a holomorphic submersion $\pi:\,\mathcal{X}\rightarrow S$
of relative dimension $n$ over a connected base and a relatively
ample line bundle $\mathcal{L}\rightarrow\mathcal{X}.$ In the absolute
case when $S$ is a point we will often use the notation $L\rightarrow X$
for the corresponding ample line bundle. In this latter case we will
write $\mathcal{H}_{L}$ for the space of all smooth Hermitian metrics
on $L$ with positive curvature form. In the relative case we will
denote by $\mathcal{H}_{\mathcal{L}/S}$ the space of all metrics
on $\mathcal{L}$ which are fiber-wise of positive curvature. We will
denote by $c_{1}(L)$ the first Chern class of $L,$ normalized so
that it lies in $H^{1,1}(X)\cap H^{2}(X,\Z).$ To simplify the formulas
to be discussed we will also assume that the relative volume of $\mathcal{L}$
is equal to one, i.e. \[
V:=\int_{X}c^{1}(L)^{n}/n!=1\]
for some (and hence any) fiber $X.$ The general formulas may then
be obtained by trivial scalings by $V$ at appropriate places. When
considering tensor powers of $L,$ written as $kL$ in additive notation,
we will always assume that $kL$ is very ample (which is true for
$k$ sufficiently large).

\subsection{\label{sub:The-weight-notation}The weight notation for $\mathcal{H}_{L}$}

It will be convenient to use the {}``weight'' representation of
a metric $h$ on $L:$ locally, any metric $h$ on $L$ may be represented
as $h=e^{-\phi},$ where $h$ is the point-wise norm of a local trivializing
section $s$ of $L.$ We we will call the additive object $\phi$
a {}``weight'' on $L.$ One basic feature of this formalism is that
even though the functions representing $\phi$ are merely locally
defined the normalized curvature form of the metric $h$ may be expressed
as \[
\omega_{\phi}:=dd^{c}\phi:=\frac{i}{2\pi}\partial\bar{\partial}\phi\]
 which is hence globally well-defined (but it does not imply that
$\omega$ is exact!). The normalizations are made so that $[\omega_{\phi}]=c_{1}(L)\in H^{1,1}(X)\cap H^{2}(X,\Z).$
In the absolute setting we will denote by $\mathcal{H}_{L}$ the space
of all weights such that $\omega_{\phi}>0.$ In other words, the map
$\phi\mapsto\omega_{\phi}$ establishes an isomorphism between $\mathcal{H}_{L}/\R$
and the space of Kähler metrics in $c_{1}(L).$ In the relative setting
we will denote by $\mathcal{H}_{\mathcal{L}/S}$ the space of all
smooth weights on $\mathcal{L}$ such that the restriction to each
fiber is of positive curvature. 

After fixing a reference weight $\phi_{0}$ in $\mathcal{H}_{L}$
the map $\phi\mapsto u:=\phi-\phi_{0}$ identifies the affine space
of all smooth weights on $L$ with the vector space $\mathcal{C}^{\infty}(X).$
Moreover, the subspace $\mathcal{H}_{L}$ of all positively curved
smooth weights gets identified with the open convex subspace $\mathcal{H}_{\omega}:=\{u:\,\,\, dd^{c}u+\omega_{0}>0\}$
of $\mathcal{C}^{\infty}(X),$ where $\omega_{0}$ denotes the Kähler
form $dd^{c}\phi_{0}.$ The $L^{1}-$closure of $\mathcal{H}_{\omega}$
is usually called the space of all\emph{ $\omega_{0}-$plurisubharmonic
functions} in the literature \cite{g-z}. In fact, all the results
in the present paper whose formulation does not use that the given
class $[\omega_{0}]$ is integral are valid in the more general setting
when $\mathcal{H}_{L}$ is replaced by $\mathcal{H}_{\omega}$ (with
essentially the same proofs). However, since the quantized setting
(section \ref{sec:Quantization:-The-Bergman general}) only makes
sense for integral classes we will stick to the weight notation in
the following.

\subsection{\label{sub:The-measure-mu and functionals}The measure $\mu_{\phi}$
and associated functionals on $\mathcal{H}_{L}$ }

First consider the absolute case when $S$ is a point. In each particular
setting studied in sections \ref{sec:The-Calabi-Yau-setting}, \ref{sec:The-(anti-)-canonical}
we will assume given a function $\mu$ on $\mathcal{H}_{L},$ $\phi\mapsto\mu(\phi)$
(also denoted by $\mu_{\phi})$ taking values in the space of volume
forms on $X,$ which is \emph{exact }in the following sense. First
observe that we may identify $\mu(\phi)$ with a one-form on the affine
space $\mathcal{H}_{L}$ by letting its action on a tangent vector
$v\in\mathcal{C}^{\infty}(X)$ at the point $\phi\in\mathcal{H}_{L}$
be defined by \[
\left\langle \mu(\phi),v\right\rangle :=\int_{X}v\mu(\phi).\]
The assumption on $\mu(\phi)$ is then simply that this one-form is
closed and hence exact, i.e. there is a functional $I_{\mu}$ on $\mathcal{H}_{L}$
such that $dI_{\mu}=\mu:$ \begin{equation}
\frac{dI_{\mu}(\phi_{t})}{dt}=\int_{X}\frac{\partial\phi_{t}}{\partial t}\mu_{\phi_{t}}\label{eq:i functional as primit}\end{equation}
 for any path $\phi_{t}$ in $\mathcal{H}_{L}$. The functional is
determined up to a constant which will be fixed in each particular
setting to be studied. We will also assume that for any fixed $v\in\mathcal{C}^{\infty}(X)$
the functional $\phi\mapsto\left\langle \mu(\phi),v\right\rangle $
is continuous with respect to the $L^{\infty}-$topology on $\mathcal{H}_{L}.$

Two particular examples of such exact one-forms and their anti-derivatives
that will be used repeatedly are as follows:
\begin{itemize}
\item The Monge-Ampère measure $\phi\mapsto(dd^{c}\phi)^{n}/n!:=MA(\phi).$
Its anti-derivative \cite{ma} will be denoted by $\mathcal{E}(\phi),$
normalized so that $\mathcal{E}(\phi_{0})=0$ for a fixed reference
weight in $\phi_{0}$ in $\mathcal{H}_{L}.$ Integrating along line
segments in $\mathcal{H}_{L}$ gives an explicit expression for $\mathcal{E},$
but it will not be used here.
\item $\phi\mapsto\mu_{0}$ for a volume form $\mu_{0}$ on $X,$ fixed
once and for all with $I_{\mu_{0}}(\phi):=\int_{X}(\phi-\phi_{0})\mu_{0}.$
Since we have already fixed a reference weight $\phi_{0}$ it will
be convenient to take $\mu_{0}:=(dd^{c}\phi_{0})^{n}/n!$ 
\end{itemize}
Given $\mu=\mu(\phi)$ we define the associated functional \[
\mathcal{F}_{\mu}:=\mathcal{E}-I_{\mu},\]
 By construction its critical points in $\mathcal{H}_{L}$ are precisely
the solutions to the Monge-Ampère equation \begin{equation}
(dd^{c}\phi)^{n}/n!=\mu(\phi)\label{eq:general ma eq}\end{equation}
We will say that $\mu(\phi)$ is \emph{normalized }if it is a probability
measure for all $\phi.$ Equivalently, this means that $I_{\mu}$
is \emph{equivariant} under scalings, i.e. $I_{\mu}(\phi+c)=I_{\mu}(\phi)+c$
which in turn is equivalent to $\mathcal{F}_{\mu}$ being \emph{invariant}
under scalings. 

In the relative setting we assume that $\mu_{s}(\phi)$ is a smooth
family of measures on the fibers $\mathcal{X}_{s}$ as above, parametrized
by $s\in S.$

\subsubsection*{Properness and coercivity}

We first recall the definition of the well-known $J-$functional,
defined with respect to a fixed reference weight $\phi_{0}$ (see
\cite{bbgz} for a general setting and references). It is the natural
higher dimensional generalization of the (squared) Dirichlet norm
on a Riemann surface and it will play the role of an exhaustion function
of $\mathcal{H}_{L}/\R$ (but without specifying any topology!). In
our notation $J$ is simply given by the scale invariant function
\[
J=-\mathcal{F}_{\mu_{0}}\]
 We will then say that a functional $\mathcal{G}$ is \emph{proper}
if \[
J\rightarrow\infty\implies\mathcal{G}\rightarrow\infty.\]
 and \emph{coercive} if for there exists $\delta>0$ and $C_{\delta}$
such that \[
J\rightarrow\infty\implies\mathcal{G\geq}\delta J-C_{\delta}.\]
Note that $\delta$ may be taken arbitrarily small at the expense
of increasing $C_{\delta}.$ In many geometric applications properness
(and coercivity) of suitable functionals can be thought as analytic
versions of algebro-geometric stability (compare remark \ref{rem:stab}).

\subsection{\label{sub:The-relative-K=0000E4hler-Ricci general}The relative
Kähler-Ricci flow with respect to $\mu_{\phi}$}

Given an initial weight $\phi_{0}\in\mathcal{H}_{\mathcal{L}/S}$
the relative Kähler-Ricci flow in $\mathcal{H}_{\mathcal{L}/S}$ is
defined by the fiber-wise parabolic Monge-Ampère equation 

\begin{equation}
\frac{\partial\phi_{t}}{\partial t}=\log\frac{(dd^{c}\phi_{t})^{n}/n!}{\mu(\phi_{t})}\label{eq:general kr flow on weights}\end{equation}
 for $\phi_{t}$ smooth over $\mathcal{X}\times[0,T],$ where $T\geq0.$
We will make the following assumptions on the flow which will all
be satisfied in the particular settings studied in sections \ref{sec:The-Calabi-Yau-setting},
\ref{sec:The-(anti-)-canonical}.

\subsubsection*{Analytical assumptions on the flow: }
\begin{itemize}
\item \emph{Existence: }The flow exists and is smooth over $\mathcal{X}\times[0,\infty[$
\item \emph{Uniqueness: }Any fixed point in $\mathcal{H}_{L}$ of the flow
is unique mod $\R$
\item \emph{Stability:} For any $l>0$ and $M>0$ there is a constant $B_{l,M}$
only depending on the upper bound on the $\mathcal{C}^{l}-$norm of
the initial weight $\phi_{0}$ (with respect to a fixed reference
weight) and a lower bound on the absolute value of $dd^{c}\phi_{0}$
such that \begin{equation}
\left\Vert \phi_{t}-\phi_{0}\right\Vert _{\mathcal{C}^{l}(X\times[0,M]}\leq B_{l,M}\label{eq:stab estim}\end{equation}
 (locally uniformly with respect to $s$ in the relative setting)
\end{itemize}
It follows immediately that $\phi$ is fixed under the flow iff it
solves the Monge-Ampère equation \ref{eq:general ma eq}. Note that
since we have assume that $Vol(L)=1$ a necessary condition to be
stationary is hence that $\int_{X}\mu_{\phi}=1.$ For any solution
$\phi_{t}$ and fixed fiber $X=\mathcal{X}_{s}$ the Kähler metrics
$\omega_{t}$ on $X$ obtained as the restricted curvature forms of
$\phi_{t}$ hence evolve according to \begin{equation}
\frac{\partial\omega_{t}}{\partial t}=-\mbox{Ric}\omega_{t}-\eta_{\mu},\label{eq:general kr on kähler metrics}\end{equation}
 where $\mbox{Ric}\omega_{t}$ in the Ricci curvature of the Kähler
metric $\omega_{t}$ and $\eta_{\mu}=dd^{c}\log\mu(\phi).$ 

Thanks to the following simple lemma the Kähler-Ricci flow is {}``gradient-like''
for the functional $\mathcal{F}_{\mu}.$ For the Fano case see \cite{ct}.
\begin{lem}
The functional $\mathcal{F}_{\mu}$ is increasing along the \emph{Kähler-Ricci
flow on $\mathcal{H}_{L}$ (defined with respect to $\mu_{\phi}).$
Moreover, it is }strictly \emph{increasing at $\phi_{t}$ unless $\phi_{t}$
is stationary.}\end{lem}
\begin{proof}
Differentiating along the flow gives $\frac{d\mathcal{F}(\phi_{t})}{dt}=$
\[
=\int_{X}\log(\frac{MA(\phi_{t})}{\mu(\phi_{t})})(MA(\phi_{t})-\mu(\phi_{t})=\int_{X}\log(\frac{MA(\phi_{t})}{\mu(\phi_{t})})(\frac{MA(\phi_{t})}{\mu(\phi_{t})}-1)\mu(\phi_{t})\geq0\]
 where the last inequality follows since both factors in the last
integrand clearly have the same sign.
\end{proof}
If moreover, $\mu(\phi)$ is normalized then both terms appearing
in the definition of $\mathcal{F}_{\mu}$ are monotone:
\begin{lem}
Assume that $\mu(\phi)$ is normalized.  Then the functionals $-I_{\mu}$
and $\mathcal{E}$ are both increasing along the \emph{Kähler-Ricci
flow on $\mathcal{H}_{L}$ with respect to $\mu(\phi).$ Moreover,
they are }strictly \emph{increasing at $\phi_{t}$ unless $\phi_{t}$
is stationary.}\end{lem}
\begin{proof}
Differentiating along the flow gives \[
-\frac{dI(\phi_{t})}{dt}=-\int_{X}\log(\frac{MA(\phi_{t})}{\mu(\phi_{t})})\mu(\phi_{t})\geq0\]
 using Jensen's inequality applied to the concave function $f(t)=\log t$
on $\R_{+}$ in the last step (recall that $MA(\phi_{t}),\mu(\phi_{t})$
are both probability measures). Similarly,

\[
\frac{d\mathcal{E}(\phi_{t})}{dt}=\int_{X}\log(\frac{MA(\phi_{t})}{\mu(\phi_{t})})MA(\phi_{t})=-\int_{X}\log(\frac{\mu(\phi_{t})}{MA(\phi_{t})})MA(\phi_{t})\geq0,\]
 again using Jensen's inequality, but with the roles of $MA(\phi_{t}),\mu(\phi_{t})$
reversed. The statement about strict monotonicity also follows from
Jensen's inequality since $f(t)=\log t$ is strictly concave.
\end{proof}
From the previous lemma we deduce the following compactness property
of the flow.
\begin{lem}
\label{lem:compactness of kr flow}Assume that $\mu(\phi)$ is normalized
and that the associated functional $-\mathcal{F}_{\mu}$ is coercive.
Then there is a constant $C$ such that $J(\phi_{t})\leq C$ and $\int|\phi_{t}-\phi_{0}|\mu_{0}\leq C$
along the Kähler-Ricci flow for $\phi_{t}$ (with respect to $\mu(\phi)).$\end{lem}
\begin{proof}
Combining the monotonicity of $\mathcal{F}_{\mu}$ and the assumption
that $\mathcal{F}_{\mu}$ be coercive (and in particular proper) immediately
gives the first inequality $J(\phi_{t})\leq C.$ Next, by the definition
of coercivity there are $\delta\in]0,1[$ and $C_{\delta}>0$ such
that $I_{\mu}-\mathcal{E}\geq\delta I_{\mu_{0}}-\delta\mathcal{E}-C_{\delta}$
i.e. \[
\delta I_{\mu_{0}}\mathcal{\leq}(-1+\delta)\mathcal{E}+I_{\mu}+C_{\delta}.\]
along the flow. Since by the previous lemma $-\mathcal{E}$ and $I_{\mu}$
are both bounded from above along the flow it follows that there is
a constant $A$ such that $I_{\mu_{0}}\leq A$ along the flow. Finally,
by basic pluripotential theory the set $\{\phi\in\mathcal{H}_{L}:\, J(\phi)\leq C,\,\,\, I_{\mu_{0}}(\phi)\leq C\}$
is relatively compact in the $L^{1}-$topology \cite{bbgz}. This
proves the last inequality in the statement of the lemma. 
\end{proof}
The next proposition shows that, under suitable assumptions, the Kähler-Ricci
flow with respect to a normalized measure $\mu_{\phi}$ converges
on the level of weights precisely when it converges on the level of
Kähler metrics. In sections \ref{sec:The-Calabi-Yau-setting} and
\ref{sec:The-(anti-)-canonical} the proposition will be applied to
the usual geometric Kähler-Ricci flows, where the convergence is already
known to hold on the level of Kähler metrics. To simplify the notation
we will only state the result in the absolute case, the extension
to the relative case being immediate.
\begin{prop}
\label{pro:conv of curv implies conv of weights}Assume that $\mu(\phi)$
is normalized and that the associated functional $-\mathcal{F}_{\mu}$
is coercive. Let $\phi_{t}$ evolve according to the Kähler-Ricci
flow defined with respect to $\mu_{\phi}$ and write $\omega_{t}=dd^{c}\phi_{t}.$
Then the following is equivalent: 
\begin{itemize}
\item The sequence of Kähler metrics $\omega_{t}$ is relatively compact
in the $\mathcal{C}^{\infty}-$topology on $X,$ i.e. for any positive
integer $l$ the sequence $\omega_{t}$ is uniformly bounded in the
$\mathcal{C}^{l}-$norm on $X.$
\item The weights converge: $\phi_{t}\rightarrow\phi_{\infty}\in\mathcal{H}_{L}$
in the $\mathcal{C}^{\infty}-$topology on $X$ as $t\rightarrow\infty$
\item The Kähler metrics $\omega_{t}\rightarrow\omega_{\infty}$ in the
$\mathcal{C}^{\infty}-$topology on $X$, where $\omega_{\infty}$
is a Kähler form.
\end{itemize}
\end{prop}
\begin{proof}
Assume that the first point of the proposition holds. Then it is a
basic fact that the sequence of normalized weights $\tilde{\phi_{t}}:=\phi_{t}-C_{t}$
where $C_{t}:=I_{\mu_{0}}(\phi_{t})$ is relatively compact in the
$\mathcal{C}^{\infty}-$topology on $X$ to $\tilde{\phi}_{\infty}\in\mathcal{H}_{L}$
(as is seen by inverting the associated Laplacians). By the previous
lemma $|C_{t}|\leq D$ for some positive constant $D$ and hence $\{\phi_{t}\}$
is also relatively compact in the $\mathcal{C}^{\infty}-$topology
on $X.$ 

In the rest of the argument we will use the $\mathcal{C}^{l}-$topology
on \emph{$\mathcal{H}_{L}$ }for $l$ a large fixed integer. Let $\mathcal{K}:=\overline{\{\phi_{t}\}}$
be the closure of $\{\phi_{t}\}$ which is relatively compact in \emph{$\mathcal{H}_{L}$}
by the previous argument. Denote by $\psi_{0}$ an accumulation point
in $\mathcal{K}:$ \[
\lim_{j}\phi_{t_{j}}=\psi_{0}\]
 By continuity of the {}``time $s$ flow map'' (which follows immediately
from the stability assumption on the flow) and the semi-group structure
of the flow we deduce that \[
\lim_{j}\phi_{t_{j}+s}=\psi_{s}\]
 for any fixed $s>0.$ In other words, $\mathcal{K}$ is in fact\emph{
compact} and\emph{ invariant }under the {}``time $s$ flow map''.
Note also that by monotonicity \begin{equation}
\lim_{t}\mathcal{E}(\phi_{t})=\mathcal{E}(\psi_{0})=\sup_{\mathcal{K}}\mathcal{E}\label{eq:pf of unif conv of berg}\end{equation}
 Assume now to get a contradiction that $\psi_{s}\neq\psi_{0}.$ By
the strict monotonicity in lemma \ref{lem:monotone along bergman cy}
we have that $\mathcal{E}(\psi_{s})>\mathcal{E}(\psi_{0}),$ contradicting
\ref{eq:pf of unif conv of berg} (since $\psi_{s}\in\mathcal{K}$
as explained above). Hence, $\psi_{0}$ is a fixed point of the flow
and hence, by the uniqueness assumption on the flow, it is determined
up to an additive constant. This means that for any two limit points
$\psi_{0}$ and $\psi'_{0}$ of the flow there is a constant $C$
such that \[
\psi_{0}-\psi'_{0}=C\]
But as explained above $\mathcal{E}(\psi_{0})=\mathcal{E}(\psi'_{0})$
and hence, by the scaling equivariance of $\mathcal{E}$ it follows
that $C=0.$ All in all this means that we have shown that the flow
$\phi_{t}$ converges, in the $\mathcal{C}^{\infty}-$topology on
$X$ to a limit $\phi_{\infty}$ in $\mathcal{H}_{L},$ i.e. that
the second point of the proposition holds. The rest of the implications
are trivial. \end{proof}
\begin{rem}
The coercivity is used to make sure that the compactness property
of the flow $\phi_{t}$ holds without normalizing $\phi_{t}$ (say,
by subtracting $I_{\mu_{0}}(\phi_{t})).$ If one only assumes properness
then the same proof shows that the statement still holds upon replacing
$\phi_{t}$ by $\phi_{t}-I_{\mu_{0}}(\phi_{t})$ (which, of course,
does not effect the curvature forms). The same remark applies to Proposition
\ref{pro:conv of bergman iter at level k general} below.
\end{rem}

\subsection{\label{sec:Quantization:-The-Bergman general}Quantization: The Bergman
iteration on $\mathcal{H}_{L}$}

Proceeding fiber-wise it will be enough to consider the absolute case
when $S$ is a point and we are given an ample line bundle $L\rightarrow X.$
For any positive integer $k$ such that $kL$ is very ample the quantization
at level $k$ of the space $\mathcal{H}_{L}$ is defined as the space
$\mathcal{H}^{(k)}$ of all Hermitian metrics on the $N_{k}-$dimensional
complex vector space $H^{0}(X,kL).$ Hence, $\mathcal{H}^{(k)}$ may
be identified with the symmetric space $GL(N_{k},\C)/U(N_{k})).$
In the relative setting $\mathcal{H}^{(k)}$ is replaced by the space
of all Hermitian metrics on the rank $N_{k}-$vector bundle $\pi_{*}(k\mathcal{L})$
over the base $S$ (compare section \ref{sub:Conservation-of-positivity}). 

Fix a volume form $\mu_{\phi}$ on $X$ depending on $\phi$ as above.
Then any given $\phi\in\mathcal{H}_{L}$ induces a Hermitian metric
$Hilb^{(k)}(\phi)$ \[
Hilb^{(k)}(\phi)(f,f):=\int_{X}|f|^{2}e^{-k\phi}d\mu_{\phi},\]
 giving a map \[
Hilb^{(k)}:\mathcal{H}_{L}\rightarrow\mathcal{H}^{(k)},\]
 There is also a natural injective map (independent of $\mu_{\phi})$
in the reverse direction, called the (scaled)\emph{ Fubini-Study map}
$FS^{(k)}:$ \[
FS^{(k)}(H):=\log(\frac{1}{N_{k}}\sum_{i=1}^{N_{k}}|f_{i}^{H}|^{2})\]
 where $f_{i}^{H}$ is any bases in $H^{0}(X,kL)$ which is orthonormal
with respect to $H.$ 

\emph{Donaldson's iteration (with respect to $\mu_{\phi})$} on the
space $\mathcal{H}^{(k)}$ is then obtained by iterating the composed
map \[
T^{(k)}:=Hilb^{(k)}\circ FS^{(k)}:\,\,\,\mathcal{H}^{(k)}\rightarrow\mathcal{H}^{(k)}\]
 and its fixed points are called \emph{balanced metrics at level $k$
(with respect to $\mu).$ }

In order to facilitate the comparison with the Kähler-Ricci flow it
will be convenient to consider the (essentially equivalent) iteration
on the space $\mathcal{H}_{L}$ obtained by iterating the map $FS^{(k)}\circ Hilb^{(k)}.$
This latter iteration will be called the \emph{Bergman iteration at
level $k$ (with respect to $\mu_{\phi})$} and we will denote the
$m$ th iterate by $\phi_{m}^{(k)}$ and call the parameter $m$ \emph{discrete
time.} Hence, the iteration immediately enters the \emph{finite dimensional}
submanifold o $FS(\mathcal{H}^{(k)})\subset\mathcal{H}_{L}$ of Bergman
metrics at level $k$ and stays there forever. By the very definition
of the Bergman iteration it may be written as the difference equation
\[
\phi_{m+1}^{(k)}-\phi_{m}^{(k)}=\frac{1}{k}\log\rho^{(k)}(\phi_{m}^{(k)}),\]
where $\rho^{(k)}(\phi)$ is the \emph{Bergman function at level $k$}
associated to $(\mu_{\phi},\phi),$ i.e. 

\[
\rho^{(k)}(\phi)=\frac{1}{N_{k}}\sum_{i=1}|f_{i}|^{2}e^{-k\phi},\]
 where $f_{i}$ is an orthonormal basis with respect to the Hermitian
metric $Hilb^{(k)}(\mu_{\mu},\phi).$ Note that the \emph{Bergman
measure} $\rho^{(k)}(\phi)\mu_{\phi}$ is a probability measure on
$X$ and independent of the choice of orthonormal bases. It plays
the role of the Monge-Ampère measure in the quantized setting.

It will also be convenient to, following Donaldson \cite{do3}, study
functionals defined directly on the space $\mathcal{H}^{(k)}.$ Fixing
the reference metric $H_{0}^{(k)}:=Hilb^{(k)}(\phi_{0})\in$$\mathcal{H}^{(k)}$
we may identify $\mathcal{H}^{(k)}$ with the space of all rank $N_{k}$
Hermitian matrices. We let \[
\mathcal{F}_{\mu}^{(k)}(H):=-\frac{1}{N_{k}k}\log\det(H)-I_{\mu}\circ FS^{(k)}(H)\]
 whose critical points in $\mathcal{H}^{(k)}$ are precisely the balanced
metrics (with respect to $\mu_{\phi})$; this is proved exactly as
in the particular cases considered in \cite{do2,bbgz}. We will also
consider the following functional on $\mathcal{H}_{L}$

\[
\mathcal{L}^{(k)}(\phi):=-\frac{1}{N_{k}k}\log\det(Hilb^{(k)}(\mu_{\phi},\phi),\]
normalized so that $\mathcal{L}^{(k)}(\phi+c)=\mathcal{L}^{(k)}(\phi)+c.$
Equivalently, we could have defined $\mathcal{L}^{(k)}$ as the anti-derivative
of the one-form on $\mathcal{H}_{L}$ defined by integration against
the Bergman measure $\rho^{(k)}(\phi)\mu_{\phi}.$

\subsubsection*{Monotonicity}

The following monotonicity properties were shown by Donaldson in the
particular setting considered in \cite{do3} (where $\mu_{\phi}$
is independent of $\phi).$ See also \cite{do2} for the setting when
$\mu(\phi)=MA(\phi)$ (compare section \ref{sub:Comparison-with-the}).
The main new observation here is that concavity of $I_{\mu}$ implies
monotonicity.
\begin{lem}
\label{lem:monotone along bergman general}Assume that $\mu_{\phi}$
is normalized. Then the following monotonicity with respect to the
discrete time $m$ holds along Bergman iteration \emph{$\phi_{m}^{(k)}$}
on $\mathcal{H}_{L}$ (defined with respect to $\mu_{\phi}):$
\begin{itemize}
\item The functional $\mathcal{L}^{(k)}$ is increasing along the Bergman
iteration and strictly\emph{ }increasing\emph{ }at\emph{ $\phi_{m}^{(k)}$
}unless\emph{ $\phi_{m}^{(k)}$ is }stationary. Equivalently, the
functional $-\log\det$ is strictly increasing along the Donaldson
iteration in $\mathcal{H}^{(k)}$ away from balanced metrics.
\item If $I_{\mu}$ is concave on the space $\mathcal{H}_{L}$ with respect
to the affine structure then it is decreasing along the iteration
and strictly\emph{ }decreasing\emph{ }at\emph{ $\phi_{m}^{(k)}$ }unless\emph{
$\phi_{m}^{(k)}$ is }stationary. Equivalently, the functional $I_{\mu}\circ FS^{(k)}$
is strictly decreasing along the Donaldson iteration in $\mathcal{H}^{(k)}$
away from balanced metrics..
\end{itemize}
\end{lem}
\begin{proof}
The proof of the first point is essentially the same as in Donaldson's
setting in \cite{do3}, but for completeness we repeat it here. By
definition \[
\mathcal{L}^{(k)}(\phi_{m+1})-\mathcal{L}^{(k)}(\phi_{m})=-\frac{1}{N_{k}k}\log\frac{\det(Hilb^{(k)}(\phi_{m+1})}{\det(Hilb^{(k)}(\phi_{m})}\]
By the concavity of $\log$ and Jensen's inequality we hence get \[
\mathcal{L}^{(k)}(\phi_{m+1})-\mathcal{L}^{(k)}(\phi_{m})\geq-\frac{1}{k}\log\frac{1}{N_{k}}\sum_{i=1}^{N_{k}}\left\Vert f_{i}\right\Vert _{T(Hilb^{(k)}(\phi_{m})}^{2},\]
 where $f_{i}$ is an orthonormal basis with respect to the Hermitian
metric $Hilb^{(k)}(\phi_{m})$ and where by definition $T(Hilb^{(k)}(\phi_{m})=Hilb^{(k)}(FS(Hilb^{(k)}(\phi_{m})).$
Writing out the norms explicitly shows that the rhs above may be written
as \[
-\frac{1}{k}\log(\frac{1}{N_{k}}(\sum_{i=1}^{N_{k}}|f_{i}|^{2}/\sum_{i=1}|f_{i}|^{2})\mu_{(FS(Hilb^{(k)}(\phi_{m}))}=-\frac{1}{k}\log(1)=0,\]
 using that $\mu_{\phi}$ is normalized. This hence proves the first
point.

To prove the second point we use that $I_{\mu}$ is assumed concave
and that, by definition, $\mu_{\phi}=dI_{\mu}$ as a differential,
to get \[
I_{\mu}(\phi_{m+1}^{(k)})-I_{\mu}(\phi_{m}^{(k)})\leq\int(\phi_{m+1}^{(k)}-\phi_{m}^{(k)})\mu_{\phi_{m}^{(k)}}=\frac{1}{k}\int\log\rho^{(k)}(\phi_{m}^{(k)})\mu_{\phi_{m}^{(k)}}\leq\]
\[
\leq\frac{1}{k}\log\int\rho^{(k)}(\phi_{m}^{(k)})\mu_{\phi_{m}^{(k)}}=0\]
 using the definition of the iteration and Jensen's inequality in
the last step (and the fact that $\rho^{(k)}(\phi)\mu_{\phi}$ and
$\mu_{\phi}$ are both probability measures). This proves the monotonicity
of $I_{\mu}.$ The statement about strict monotonicity follow immediately
from the fact that $\log t$ is strictly concave.
\end{proof}

\subsubsection*{Properness and coercivity }

Properness and coercivity of functionals on $\mathcal{H}^{(k)}$ are
defined as in section \ref{sub:The-relative-K=0000E4hler-Ricci general},
but with the functional $J$ replaced by its quantized version on
the space $\mathcal{H}^{(k)}:$ \[
J^{(k)}(H):=-\mathcal{F}_{\mu_{0}}^{(k)}:=I_{\mu_{0}}\circ FS^{(k)}+\frac{1}{kN_{k}}\log\det(H)\]
 The content of the  following lemma is essentially contained in the
proof of Proposition 3 in \cite{do3}. We will fix a metric $H_{0}\in\mathcal{H}^{(k)}.$
For any given $H_{0}-$orthonormal base $(f_{i})$ we can then identify
an Hermitian metric $H$ with a matrix and we will denote by $H_{\lambda}$
the diagonal matrix with entries $e^{-\lambda_{i}}$ on the diagonal.
\begin{lem}
\label{lem:jk is exhaust in bergman}The following holds 
\begin{itemize}
\item For $\lambda\in\C^{N_{k}}$ let $\phi_{\lambda}=FS^{(k)}(H_{\lambda}):=\frac{1}{k}\log(\frac{1}{N_{k}}\sum_{i}e^{k\lambda_{i}}|f_{i}|^{2}).$
Then there is a constant $C$ such that \textup{\[
\max_{i}\lambda_{i}\leq I_{\mu_{0}}(\phi_{\lambda})+C.\]
}
\item The functional $J^{(k)}$ is an exhaustion function on $\mathcal{H}^{(k)}/\R^{*}$
with respect to its usual topology
\item In particular, the set of all $H\in\mathcal{H}^{(k)}$ such that \textup{\begin{equation}
-\log\det(H)\geq-C,\,\,\,\,\,\,(I_{\mu_{0}}\circ FS)(H)\leq C\label{eq:pf of bergman it two bounds}\end{equation}
is relatively compact. }
\end{itemize}
\end{lem}
\begin{proof}
For the benefit of the reader we repeat Donaldson's simple proof:
let $i_{max}$ be an index such that $\max_{i}\lambda_{i}=\lambda_{i_{max}}.$
Clearly, \begin{equation}
\max_{i}\lambda_{i}+\frac{1}{k}(\log(\frac{1}{N_{k}}|f_{i_{max}}|^{2})\leq\phi_{\lambda}\leq\max_{i}\lambda_{i}+\frac{1}{k}\log(\frac{1}{N_{k}}\sum_{i}|f_{i}|^{2}),\label{eq:pf of lemma exhaus on bergman}\end{equation}
and hence integrating over $X$ and using the first inequality above
gives \[
\max_{i}\lambda_{i}+\int_{X}(\log(|f_{i_{max}}|^{2})-\phi_{0})d\mu_{0}\leq I_{\mu_{0}}(\phi_{\lambda}),\]
 which proves the lemma since it is well-known that $I_{\mu_{0}}(\psi)>-\infty$
for any psh weight $\psi$ if $\mu_{0}$ is a smooth volume form (as
follows from the local fact that any psh function is in $L^{1})$
and in particular $-C:=I_{\mu}(\log(|f_{i_{max}}|^{2}))>-\infty.$
This proves the first point. As for the second and third one we first
note that any Hermitian metric $H$ can be represented by a diagonal
matrix (which we write in the form $H_{\lambda})$ after perhaps changing
the base ($f_{i}$) above. Moreover, by the compactness of $U(N)$
the constant $C$ in the previous point can be taken to be independent
of the base $(f_{i}))$. 

Next, it will be enough to prove the last point of the lemma (the
second point then follows since we may by scaling invariance assume
that $\det(H_{\lambda})=1).$ We may assume that $\inf_{i}\lambda_{i}=\lambda_{0}$
and since, by assumption, \[
-\log\det(H)=\sum_{i}\lambda_{i}\geq-C\]
we get \[
-\inf_{i}\lambda_{i}\leq C+\sum_{i\neq0}\lambda_{i}\leq C+(N-1)\max_{i}\lambda_{i}\]
By the assumption $(I_{\mu_{0}}\circ FS)(H)\leq C$ and the first
point of the lemma the rhs above is bounded from above and hence we
conclude that so is $-\inf_{i}\lambda_{i}.$ All in all this means
that $\max_{i}|\lambda_{i}|$ is uniformly bounded from above by a
constant, i.e. $H$ stays in a relatively compact subset of $\mathcal{H}^{(k)}.$\end{proof}
\begin{rem}
The proof of the previous lemma shows that the conclusion of the lemma
remains valid for any choice of a fixed reference weight $\phi_{0}$
and probability measure $\mu_{0}$ (which are used in the definition
of $J^{(k)})$ such that $\int_{X}\log(|f|-\phi_{0})\mu_{0}$ is finite
for any section $f\in H^{0}(X,kL).$ 
\end{rem}

\subsection*{Criteria for convergence in the large time limit}
\begin{prop}
\label{pro:conv of bergman iter at level k general}Assume that $\mu_{\phi}$
is normalized, that $I_{\mu}$ is decreasing along the Bergman iteration,
that $\mathcal{F}_{\mu}^{(k)}$ is coercive and that there is at most
one balanced metrics (modulo scaling).  Then, for any given positive
integer $k$ the following holds: In the large time limit, i.e. when
$m\rightarrow\infty$ the weights $\phi_{m}^{(k)}\rightarrow\phi_{\infty}^{(k)}$
in the $\mathcal{C}^{\infty}-$topology on $X.$ Moreover, in the
relative setting the convergence is uniform with respect to the base
parameter $s.$\end{prop}
\begin{proof}
\emph{a) uniform convergence: }

We equip $FS(\mathcal{H}^{(k)}),$ i.e. the space of all Bergman weights
at level $k,$ with the topology induced by the sup-norm. It is not
hard to see that this is the same topology as the one induced from
the finite dimensional symmetric space $\mathcal{H}^{(k)}=GL(N_{k},\C)/U(N_{k})$
with its usual Riemannian metric, or with respect to the operator
norm on $GL(N_{k},\C).$ Hence, it will be enough to prove the convergence
of Donaldson's iteration in $\mathcal{H}^{(k)}.$ 

Since $\mu_{\phi}$ is assumed normalized Lemma \ref{lem:monotone along bergman general}
shows that $-\log\det H$ is uniformly bounded from below along the
Donaldson iteration in $\mathcal{H}^{(k)}$. Moreover, by assumption
$I_{\mu_{\phi}}\circ FS^{(k)}$ is uniformly bounded from above along
the Donaldson iteration. Hence, just as in the proof of Lemma \ref{lem:compactness of kr flow}
it follows from the coercivity assumption that $I_{\mu_{0}}\circ FS^{(k)}$
is also uniformly bounded from above along the Donaldson iteration.
But then it follows from Lemma \ref{lem:jk is exhaust in bergman}
that the iteration $H_{m}^{(k)}$ stays in a compact subset of $\mathcal{H}^{(k)}.$

Let now $\mathcal{K}:=\overline{\{H_{m}^{(k)}\}}$ be the closure
of the orbit of $T^{(k)}$ which is relatively compact in $\mathcal{H}^{(k)}$
by the previous argument. Denote by $G$ an accumulation point \[
\lim_{j}H_{m_{j}}^{(k)}=G\]
 in $\mathcal{H}^{(k)}.$ By the continuity of $H\mapsto T^{(k)}(H)$
on $\mathcal{H}^{(k)}$ we deduce that \[
\lim_{j}T^{(k)}(H_{m_{j}}^{(k)})=T^{(k)}(G).\]
 In other words, $\mathcal{K}$ is in fact\emph{ compact} and\emph{
invariant under $T^{(k)}.$} Note also that by monotonicity \[
\lim_{j}-\log\det(H_{m_{j}}^{(k)})=-\log\det(G)=\sup_{\mathcal{K}}(-\log\det)\]
Assume now to get a contradiction that $T^{(k)}(G)\neq G.$ By the
strict monotonicity in lemma \ref{lem:monotone along bergman cy}
we have that $\log\det(T^{(k)}G))>\log\det(G)$ contradicting \ref{eq:pf of unif conv of berg}
(since $T^{(k)}(G)\in\mathcal{K}).$ All in all this means that we
have shown that the subsequence $(H_{m_{j}}^{(k)})$of Donaldson iteration
converges to a fixed point, i.e. a balanced metric. By the assumption
on uniqueness up to scaling it follows, again using monotonicity (just
like in the proof of Proposition \ref{pro:conv of curv implies conv of weights})
that all accumulation points coincide, i.e. the iteration converges.

\emph{b) Higher order convergence: }

To simplify the notation we set $k=1$ and write $\phi_{m}^{(k)}=\phi_{m}.$
First note that the $L^{\infty}-$estimate above is uniform over $S,$
as follows by combining the monotonicity of the functionals with the
uniform boundedness of the initial weight $\phi_{0}$. By the uniform
convergence of $\phi_{m}$ it will hence be enough to prove that 

\begin{equation}
\left\Vert \partial_{X}^{\alpha}(h_{0}/h_{m+1})\right\Vert _{L^{\infty}(X)}\leq C_{\alpha}\left\Vert (h_{m}/h_{0})\right\Vert _{L^{\infty}(X)}\label{eq:smoothing prop of bergman it}\end{equation}
where $h_{m}=e^{-\phi_{m}}$ and $\partial_{X}^{\alpha}$ denotes
a real linear differential operator on $X$ of order $\alpha$ (note
that while $h_{m}$ globally corresponds to a metric on $L$ the quotient
$h_{0}/h_{m+1}$ defines a global\emph{ function} on $X).$ Accepting
this estimate for the moment the uniform convergence of $(h_{m})$
hence gives that $\left\Vert \partial_{X}^{\alpha}(h_{0}/h_{m}\right\Vert _{L^{\infty}(X)}$
is uniformly bounded in $m$ and since $h_{m}/h_{0}\rightarrow h_{\infty}/h_{0}$
it then follows that $\left\Vert \partial_{X}^{\alpha}(\phi_{m}-\phi_{0})\right\Vert _{L^{\infty}(X)}$
is also uniformly bounded in $m.$ Hence, standard compactness arguments
show the $\mathcal{C}^{\infty}-$convergence of $(\phi_{m}).$ 

Finally, the estimate \ref{eq:smoothing prop of bergman it} is a
consequence of the following quasi-explicit integral formula for the
Bergman function familiar from the theory of determinantal random
point processes (see \cite{be1} and references therein): $\rho(\phi)(x)=$
\[
=\int_{y\in X^{N-1}}f(x,y)e^{-(\phi-\phi_{0})(x)}e^{-(\phi-\phi_{0})(y)}d\mu_{\phi}(y)^{\otimes N-1}/Z_{\phi},\,\, Z_{\phi}:=\int_{X^{N}}f_{0}e^{-(\phi-\phi_{0})}d\mu{}_{\phi}^{\otimes N}\]
 where $f(x_{1},x_{2},...,x_{N})=|\det_{1\leq i,j\leq N}(f_{i}(x_{i}))_{i,j}|^{2}e^{-\phi_{0}(x_{1})}...e^{-\phi_{0}(x_{N})}$
and $(f_{i})$ is any given orthonormal base with respect to the Hermitian
metric $Hilb^{(1)}(\phi_{0})$ on $H^{0}(X,L)$ (note that $Z_{\phi}$
appears as the normalizing constant). We have used the notation $\phi(x,...x_{m})=\phi(x_{1})+...\phi(x_{m}).$
In particular, \[
(h_{0}/h_{m+1})(x)=\int_{y\in X^{N-1}}f(x,y)e^{-(\phi_{m}-\phi_{0})(y)}d\mu_{\phi}(y)^{\otimes N-1}/Z_{\phi_{m}}\]
and hence differentiating wrt $x$ by applying $\partial_{X}^{\alpha}$
gives \[
|\partial_{X}^{\alpha}(h_{0}/h_{m+1})(x)|=\left|\int(\partial_{X}^{\alpha}f(x,y))e^{-(\phi-\phi_{0})(y)}d\mu_{\phi}(y)^{\otimes N-1}/Z_{\phi}\right|\leq\frac{A_{\alpha}}{Z_{\phi_{m}}}\left\Vert e^{-(\phi_{m}-\phi_{0})}\right\Vert _{L^{\infty}(X)},\]
where $A_{\alpha}$ is a constant independent of $m.$ Since, by the
uniform convergence of $\phi_{m},$ we have that $Z_{\phi_{m,}}>C>0$
for some positive constant $C$ this concludes the proof of the estimate
\ref{eq:smoothing prop of bergman it}.
\end{proof}
The following basic lemma gives a natural criterion for the assumptions
(a part from the monotonicity of $I_{\mu})$ in the previous theorem
to be satisfied. 
\begin{lem}
\label{lem:geod convex}Suppose that $\mathcal{G}$ is a functional
on $\mathcal{H}^{(k)}$ which is geodesically strictly convex with
respect to the symmetric Riemann structure and strictly convex modulo
scaling. Then $\mathcal{G}$ has at most one critical point (modulo
scaling). Moreover, if it has some critical point then $\mathcal{G}$
is coercive.\end{lem}
\begin{proof}
Uniqueness follows immediately from strict convexity and hence we
turn to the proof of coercivity. By a simple compactness argument
it will be clear that, after fixing a reference metric $H_{0}\in\mathcal{H}^{(k)},$
which we take to be a critical point of $\mathcal{G}$, it is enough
to prove coercivity along any fixed geodesic passing through $H_{0}.$
To this end let $H_{t}$ be a geodesic in $\mathcal{H}^{(k)}$ starting
at $H_{0},$ i.e. the orbit of the action of a one-parameter subgroup
of $GL(N_{k}).$ In the notation of Lemma \ref{lem:jk is exhaust in bergman}
this means that $H_{t}=H_{t\lambda}$ for $\lambda\in\C^{N}$ fixed.
By scaling invariance we may assume that the determinant of $H_{t}$
vanishes along the geodesic. Integrating the upper bound in \ref{eq:pf of lemma exhaus on bergman}
over $X$ gives \[
J(H_{t})=0+(I_{\mu_{0}}\circ FS)(H_{t})\leq Ct+D\]
Now, let $f(t)=\mathcal{G}(H_{t}).$ Since by assumption $f$ is convex
and $0$ is a critical point have that $df/dt\geq0$ for all $t.$
Hence, if we fix some number $\epsilon>0,$ then \[
f(t)\geq f(0)+\int_{\epsilon}^{t}(df/ds)ds.\]
But by the assumption on strict convexity the latter integrand is
bounded from below by some $\delta>0.$ All in all this shows that
\[
\mathcal{G}(H_{t})\geq\delta t-A\geq\frac{\delta}{C}J(H_{t})-A'\]
 which finishes the proof. 
\end{proof}

\subsection*{Large $k$ asymptotics}

Next, we will recall the following proposition which is the link between
the Bergman iteration and the Kähler-Ricci flow. It is essentially
due to Bouche and Tian, apart from the uniformity with respect to
$\phi.$ In fact, a complete asymptotic expansion in powers of $k$
holds as was proved by Catlin and Zelditch and the uniformity can
be obtained by tracing through the same arguments (as remarked in
connection to Proposition 6 in \cite{do1}). For references see the
recent survey \cite{z}.
\begin{prop}
\label{pro:bouche tian-1}Assume that the volume form $\mu_{\phi}$
depends smoothly on $\phi.$ Then the following uniform convergence
for the corresponding Bergman function $\rho_{(k)}(\phi)$ holds:
there is an integer $l$ such that \[
\sup_{X}|\rho_{(k)}(\phi)-\frac{(dd^{c}\phi)^{n}/n!}{\mu_{\phi}}|\leq C/k\]
for all weights $\phi$ such that $dd^{c}\phi$ is uniformly bounded
from above in $\mathcal{C}^{l}-$norm with $dd^{c}\phi$ uniformly
bounded from below by some fixed Kähler form. 
\end{prop}

\section{\label{sec:The-Calabi-Yau-setting}The Calabi-Yau setting}

First consider the absolute case where we assume given an ample line
bundle $L\rightarrow X.$ In this section we will the apply the general
setting introduced in the previous setting to the case when the measure
$\mu$ is independent of $\phi.$ We will assume that it is normalized,
i.e. a probability measure. We will mainly be interested in the case
when $X$ is a Calabi-Yau manifold, which induces a canonical probability
measure $\mu$ on $X$ defined by \[
\mu=c_{n}\Omega\wedge\bar{\Omega}\]
 where $\Omega$ is any given holomorphic $n-$form trivializing the
canonical line bundle $K_{X}$ and $c_{n}$ is a normalizing constant.
In the relative Calabi-Yau setting, where each fiber is assumed to
be a Calabi-Yau manifold, this hence yields a canonical smooth family
of measures on the fibers. 

For a fixed reference element $\phi_{0}\in\mathcal{H}_{L}$ we set
\[
I_{\mu}(\phi):=\int_{X}(\phi-\phi_{0})\mu,\]
which is \emph{equivariant} under the usual actions of the additive
group $\R:$ $I_{\mu}(\phi+c)=I_{\mu}(\phi)+c.$ Moreover, be definition
the associated functional $-\mathcal{F}_{\mu}$ is coercive.

\subsection{The relative Kähler-Ricci flow}

The convergence on the level of Kähler forms in the following theorem
is due to Cao (a part from the uniqueness which was first shown by
Calabi). We just observe that, since $\mu$ is normalized, the convergence
of the flow also holds on the level of weights. 
\begin{thm}
\emph{\label{thm:cao}The} \emph{Kähler-Ricci flow on $\mathcal{H}_{L}$
with respect to $\mu$ exists for all times $t\in[0,\infty[$ and
the solution $\phi_{t}$ is smooth on $X\times[0,\infty[.$ Moreover,
$\phi_{t}\rightarrow\phi_{\infty}$ uniformly in the $\mathcal{C}^{\infty}$-
topology on $X$ when} \emph{$t\rightarrow\infty,$ where $\phi_{\infty}$
is the unique (modulo scaling) solution to the inhomogeneous Monge-Ampère
equation }\ref{eq:inhome ms intr}\emph{. More precisely, all the
analytical assumptions in section \ref{sub:The-relative-K=0000E4hler-Ricci general}
are satisfied. In the Calabi-Yau case $\omega_{\infty}$ is Ricci
flat. }\end{thm}
\begin{proof}
As shown by Cao \cite{ca} $\omega_{t}\rightarrow\omega_{\infty}$
in the $\mathcal{C}^{\infty}-$topology. But then it follows from
Proposition \ref{pro:conv of curv implies conv of weights} that \emph{,
$\phi_{t}\rightarrow\phi_{\infty}$ uniformly in the $\mathcal{C}^{\infty}$-
topology on $X.$} The smoothness in the relative case was not stated
explicitly in \cite{ca} but follows from basic maximum principle
arguments.
\end{proof}

\subsubsection{Preliminaries: Kodaira-Spencer classes and Weil-Petersson  geometry }

In this section we will assume that the base $S$ is one-dimensional
and embedded as a domain in $\C.$ Recall that the infinitesimal deformation
of the complex structures on the smooth manifold $\mathcal{X}_{s}$
as $s$ varies is captured by the the Kodaira-Spencer class $\rho(\frac{\partial}{\partial s})\in H^{1,0}(T^{1,0}\mathcal{X}_{s})$
\cite{v}. When the fibers are Calabi-Yau manifolds the {}``size''
of the deformation is measured by the (generalized) Weil-Petersson
form \cite{fs} $\omega_{WP}$ on the base $S.$ It was extensively
studied by and Tian \cite{ti0} and Todorov \cite{to} when the base
$S$ is a moduli space of Calabi-Yau manifolds and $\mathcal{X}$
is the corresponding Kuranishi family. The form $\omega_{WP}$ is
defined by \begin{equation}
\omega_{WP}(\frac{\partial}{\partial s},\frac{\partial}{\partial s}):=\left\Vert A_{CY}\right\Vert _{\omega_{CY}}^{2},\label{eq:def of w-p}\end{equation}
 where $A_{CY}$ denotes the unique representative in the Kodaira-Spencer
class $\rho(\frac{\partial}{\partial s})\in H^{1,0}(T^{1,0}\mathcal{X}_{s})$
which is harmonic with respect to a given Ricci flat metric $\omega_{CY}$
on $\mathcal{X}_{s}$ and the $L^{2}-$norm is computed with respect
to this latter metric. Moreover, as shown in \cite{to} the following
formula holds \begin{equation}
\left\Vert A_{CY}\right\Vert _{\omega_{CY}}^{2}=\frac{\partial^{2}\psi_{\Omega}}{\partial s\partial\bar{s}},\,\,\,\,\psi_{\Omega}(s):=\log i^{n^{2}}\int_{\mathcal{X}_{s}}\Omega\wedge\bar{\Omega},\label{eq:wp as curv}\end{equation}
 where $\Omega$ now denotes any given, nowhere vanishing, global
holomorphic $n-$form on $p^{-1}(U)$ and where $U$ denotes some
neighborhood of a fixed point $s$ in $S.$ More generally, for an
arbitrary smooth base $S$ the $(1,1)-$form $\omega_{WP}$ on $S$
may be defined as the curvature of the line bundle $\mbox{\ensuremath{\pi}}_{*}(K_{\chi/S})$
on $S.$ It is in the latter form that $\omega_{WP}$ will appear
in the proof of Theorem \ref{thm:heat eq for c in c-y} below. In
fact, the formula \ref{eq:def of w-p} may then be deduced from Theorem
\ref{thm:heat eq for c in c-y} (see remark \ref{rem:todorovs formel}).

Next we will explain how, for a fixed base parameter $s,$ a weight
$\phi$ on the line bundle $\mathcal{L}\rightarrow\mathcal{X}\rightarrow S$
induces the following two objects:
\begin{itemize}
\item a $(0,1)-$form $A_{\phi}$ with values in $T^{1,0}\mathcal{X}_{s}$
representing the Kodaira-Spencer class $\rho(\frac{\partial}{\partial s})$
in $H^{1,0}(T^{1,0}\mathcal{X}_{s}).$ 
\item A function $c(\phi)$ on $\mathcal{X}$ measuring the positivity (or
lack of positivity) of $dd^{c}\phi$ on $\mathcal{X}$ in terms of
the positivity of the restrictions of $dd^{c}\phi$ to the fibers
$\mathcal{X}_{s}.$
\end{itemize}
In fact $A_{\phi}$ will only depend on the family, parametrized by
$s,$ of two-forms $\omega_{s}$ obtained as the\emph{ restrictions}
of the curvature form $\omega_{\phi}$ on $\mathcal{X}$ to all fibers
$\mathcal{X}_{s},$ while $c(\phi)$ will depend on the whole form
$\omega_{\phi}.$

\subsubsection*{Trivial fibrations}

Assume that $\pi:\,\mathcal{X}\rightarrow S$ is a holomorphically
trivial fibration, so that $\mathcal{X}$ is embedded in $\C\times X$
and that $\mathcal{L}=\pi^{*}L$ where $L\rightarrow X$ is an ample
line bundle. Given a smooth family of weights $\phi(s,\cdot)$ on
$L\rightarrow X$ with strictly positive curvature form $\omega_{\phi}^{X}:=d_{X}d_{X}^{c}\phi$
(for $s$ fixed) one obtains a smooth vector field $V_{\phi}$ of
type $(1,0)$ as the {}``complex gradient'' of $\partial_{s}\phi:$
\begin{equation}
\delta_{V_{\phi}}\omega_{\phi(s,\cdot)}^{X}=\mbox{\ensuremath{\overline{\partial}}\ensuremath{\ensuremath{_{X}}(\ensuremath{\partial_{s}\phi})}},\label{eq:def of V}\end{equation}
 where $\delta_{V_{\phi}}$ denotes interior multiplication (i.e.
contraction) with $V_{\phi}.$ Now the $(0,1)-$form $A_{\phi}$ with
values in $T^{1,0}X$ (for $s$ fixed) is simply defined by \begin{equation}
A_{\phi}:=-\overline{\partial}_{X}V_{\phi}\label{eq:def of A in trivial case}\end{equation}
Denote by $\omega_{t}^{X}$ the curvature forms on $X$ evolving with
respect to the time parameter $t$ according to the Kähler-Ricci flow
(for $s$ fixed). The Laplacian on $X$ with respect to $\omega_{t}^{X}$
will be denoted by $\Delta_{\omega_{t}^{X}}.$ Given $\phi(s,\cdot)$
we define the following function on $\mathcal{X}:$ \begin{equation}
c(\phi):=\frac{1}{n}(dd^{c}\phi)^{n+1}/(d_{X}d_{X}^{c}\phi)^{n}\wedge ids\wedge d\bar{s}\label{eq:def of c}\end{equation}
Note that, since $\omega_{\phi}^{X}>0$ on $X$ we have that $c(\phi)>0$
at $(s,x)\in\mathcal{X}$ iff $dd^{c}\phi>0$ at $(s,x).$

\subsubsection*{General submersions}

Next we turn to the case of a general holomorphic submersion $\pi:\,\mathcal{X}\rightarrow S.$
Any given point in $\mathcal{X}$ has a neighborhood $\mathcal{U}$
such that the fibration $\pi:\,\mathcal{U}\rightarrow S$ is holomorphically
trivial and the restriction $\mathcal{L}_{\mathcal{U}}$ is isomorphic
to $\pi^{*}L$ over $\mathcal{U}.$ Hence, the vector field $V_{\phi}$
defined above is\emph{ locally} defined, but in general not \emph{globally}
well-defined on \emph{$\mathcal{X}.$} However, the expression \ref{eq:def of A in trivial case}
turns out to still be globally well-defined. For completeness we will
give a proof of this well-known fact \cite{sc,fs}:
\begin{prop}
\label{pro:A is well-def}The $(0,1)-$form $A_{\phi}$ with values
in $T^{1,0}\mathcal{X}_{s},$ locally defined by formula \ref{eq:def of A in trivial case},
is globally well-defined. It represents the Kodaira-Spencer class
in $H^{1,0}(T^{1,0}\mathcal{X}_{s}).$ \end{prop}
\begin{proof}
\emph{Step 1:} the locally defined expression \[
W_{\phi}:=\frac{\partial}{\partial s}-V_{\phi}\]
 defines a global vector field on $\mathcal{X}$ of type $(1,0).$ 

Indeed $W_{\phi}$ may be characterized as the\emph{ horizontal lift
of $\frac{\partial}{\partial s}$} with respect to the $(1,1)-$form
$dd^{c}\phi$ on $\mathcal{X},$ which is non-degenerate along fibers.
To see this first note that \begin{equation}
(i)\, d\pi(W_{\phi})=\frac{\partial}{\partial s},\,\,\,\,(ii)\, dd^{c}\phi(W_{\phi},\ker d\pi)=0\label{eq:def prop of hor vf}\end{equation}
The first point is trivial and the second one follows from a direct
calculation: locally we may decompose \[
dd^{c}\phi=d_{z}d_{z}^{c}\phi+\phi_{s\bar{s}}ds\wedge\bar{ds}+(\overline{\partial}_{z}\phi_{s})\wedge ds+(\partial_{z}\phi_{s})\wedge d\bar{s}.\]
 Hence, for any fixed index $i$ \[
dd^{c}\phi(W_{\phi},\frac{\partial}{\partial\bar{z}_{i}})=-d_{z}d_{z}^{c}\phi(V_{\phi},\frac{\partial}{\partial\bar{z}_{i}})+0+(\frac{\partial}{\partial\bar{z}_{i}}\phi_{s})+0=0,\]
 using the definition \ref{eq:def of V} of $V_{\phi}$ in the last
step. Finally, note that the properties \ref{eq:def prop of hor vf}
determine $W_{\phi}$ uniquely: if $W'$ is another local vector field
satisfying \ref{eq:def prop of hor vf} then clearly $Z:=W_{\phi}-W'$
satisfies \[
(i')\, d\pi(Z)=0,\,\,\,\,(ii')\, dd^{c}\phi(Z,\ker d\pi)=0\]
In particular, $Z$ is tangential to the fibers and $dd^{c}\phi(Z,\bar{Z})=0.$
But since $dd^{c}\phi$ is assumed to be non-degenerate along the
fibers it follows that $Z=0.$

Step 2: $A_{\phi}(s)=(\overline{\partial}W_{\phi})_{\mathcal{X}_{s}}$
and $A_{\phi}(s)$ represents the Kodaira-Spencer class in $H^{1,0}(T^{1,0}\mathcal{X}_{s}).$

The first formula above follows immediately from a local computation
and the second one then follows directly from the definition of the
Kodaira-Spencer class (where $W_{\phi}$ may be taken as\emph{ any
}smooth lift to $T^{1,0}\mathcal{X}$ of the vector field $\frac{\partial}{\partial s}$
\cite{v}).
\end{proof}
As for the function $c(\phi)$ defined by formula \ref{eq:def of c}
it is still well-defined as we have fixed an embedding of $S$ in
$\C.$

\subsubsection{Conservation of positivity along the relative Kähler-Ricci flow}

Next, we will prove the following theorem which is one of the main
results of the present paper. 
\begin{thm}
\label{thm:heat eq for c in c-y}Let $\pi:\mathcal{\, X}\rightarrow S$
be a proper holomorphic submersion with Calabi-Yau fibers and let
$\mathcal{L}$ be a relatively ample line bundle over $\mathcal{X}.$
Assume that the base $S$ is a domain in $\C.$ The following equation
holds along the corresponding relative Kähler-Ricci flow: \begin{equation}
(\frac{\partial}{\partial t}-\Delta_{\omega_{t}^{X}})c(\phi)=|A_{\phi}|_{\omega_{t}^{X}}^{2}-\left\Vert A_{CY}\right\Vert _{\omega_{CY}^{X}}^{2}\label{eq:heat eq for c in c-y}\end{equation}
\end{thm}
\begin{proof}
Since it will be enough to prove the identity at a fixed point $x$
in $X$ in some local holomorphic coordinates and trivializations
we may as well assume that $\omega_{\phi}$ is the Euclidean metric
at the point $x,$ i.e. that the complex Hessian matrix $(\partial^{2}\phi/\partial z_{i}\bar{\partial_{z_{j}}})$
is the identity for $z=0$ (corresponding to the fixed point $x$
in $X).$ Moreover, we may assume that locally the holomorphic $n-$form
$\Omega$ may be expressed as $\Omega=dz_{1}\wedge...\wedge dz_{n}.$Partial
derivatives with respect to $s$ will be indicated by a subscript
$s$ and partial derivatives with respect to $z_{i}$ and $\bar{z}_{j}$
by subscripts $i$ and $\bar{j}$ respectively. If $h=(h_{ij})$ is
an Hermitian matrix we will write $(h^{ij})$ for the matrix $\overline{H}^{-1}.$
The summation convention according to which repeated indices are to
be summed over will be used. Next, we turn to the proof of the theorem
which is based on a direct and completely elementary calculation.

\emph{Step 1:} the following formula holds in the case of a holomorphically
trivial fibration

\[
\frac{\partial}{\partial t}c(\phi)=\phi_{i\bar{i}s\bar{s}}+\phi_{s\bar{i}}\overline{\phi_{s\bar{j}}}\phi_{i\bar{j}k\bar{k}}-\phi_{i\bar{j}s}\overline{\phi_{i\bar{j}s}}-\phi_{s\bar{i}}\overline{\phi_{s\bar{j}}}\phi_{ik\bar{l}}\phi_{j\bar{k}l}-2\Re(\phi_{k\bar{k}s\bar{i}}\overline{\phi_{si}})+2\Re(\phi_{k\bar{l}s}\phi_{k\bar{l}\bar{i}}\overline{\phi_{s\bar{i}}})\]
To see this first recall that \[
c(\phi)=\phi_{s\bar{s}}-\Re(\phi_{s\bar{i}}\overline{\phi_{s\bar{j}}}\phi^{i\bar{j}})\]
and hence (using that $\phi_{i\bar{j}}=\delta_{ij}$ at $z=0,$ so
that $\frac{\partial}{\partial t}\phi^{i\bar{j}}=-\phi_{\bar{j}i}$
at $z=0)$ \begin{equation}
\frac{\partial}{\partial t}c(\phi)=\frac{\partial}{\partial t}\phi_{s\bar{s}}-2\Re(\frac{\partial}{\partial t}\phi_{s\bar{i}})\overline{\phi_{s\bar{i}}})+(\phi_{s\bar{i}}\overline{\phi_{s\bar{j}}})\frac{\partial}{\partial t}\phi_{\bar{j}i}=\label{eq:pf of heat eq cy}\end{equation}
 Using the definition of the relative Kähler-Ricci flow in the Calabi-Yau
case and the simple fact that the linearization of $\psi\mapsto\log\det(\psi_{k\bar{l}})$
at $\psi$ is given by $u\mapsto\Delta_{\omega_{\psi}}u,$ where $\Delta_{\omega_{\psi}}u=\psi^{k\bar{l}}u_{k\bar{l}}$
is the Laplacian with respect to the Kähler metric $\omega_{\psi}$
hence gives 

\[
\frac{\partial}{\partial t}c(\phi)=(\log(\det\phi_{i\bar{j}}))_{s\bar{s}}-2\Re((\log\det(\phi_{k\bar{l}}))_{s\bar{i}}\overline{\phi_{si}})+(\phi_{si}\overline{\phi_{s\bar{j}}})((\log(\det\phi_{k\bar{l}}))_{i\bar{j}}=\]
\[
=(\phi_{i\bar{j}s}\phi^{i\bar{j}})_{\bar{s}}-2\Re(\phi_{k\bar{l}s}\phi^{k\bar{l}})_{\bar{i}}\overline{\phi_{s\bar{i}}}+(\phi_{s\bar{i}}\overline{\phi_{s\bar{j}}})(\phi_{ik\bar{l}}\phi^{k\bar{l}})_{\bar{j}}=\]
\[
=\phi_{i\bar{i}s\bar{s}}-(\phi_{i\bar{i}s\bar{s}}+\phi_{i\bar{j}s}\phi_{j\bar{i}\bar{s}})-2\Re[(\phi_{k\bar{k}s\bar{i}}\overline{\phi_{s\bar{i}}}-\phi_{k\bar{l}s}\phi_{l\bar{k}\bar{i}}\overline{\phi_{s\bar{i}}}]+(\phi_{s\bar{i}}\overline{\phi_{s\bar{j}}})(\phi_{i\bar{j}k\bar{k}}-\phi_{ik\bar{l}}\phi_{j\bar{k}l})\]
(again using $\phi_{i\bar{j}}=\delta_{ij}$ at $z=0),$ finishing
the proof of step 1.

\emph{Step 2: }the following formula holds in the case of a trivial
fibration: \[
c(\phi)_{k\bar{k}}=\phi_{k\bar{k}s\bar{s}}+(\phi_{s\bar{i}}\overline{\phi_{s\bar{j}}})(\phi{}_{k\bar{k}j\bar{i}}-2(\phi_{s\bar{i}}\overline{\phi_{s\bar{j}}})\phi_{\bar{k}\bar{i}m}\phi_{k\bar{m}j}-\phi_{ks\bar{i}}\overline{\phi_{ks\bar{i}}}-\overline{\phi_{\bar{k}s\bar{i}}}\phi_{\bar{k}s\bar{i}}+2\Re(\phi_{ks\bar{i}}\overline{\phi_{s\bar{j}}}\phi_{\bar{k}j\bar{i}})\]

\[
+2\Re(\overline{\phi_{\bar{k}s\bar{i}}}\phi_{s\bar{j}})\phi_{\bar{k}j\bar{i}}-2\Re\phi_{\bar{k}ks\bar{i}}\overline{\phi_{s\bar{i}}}\]

To see this we first differentiate $c(\phi)$ with respect to $z_{k}$
to get \[
c(\phi)_{k}=\phi_{ks\bar{s}}-[(\phi_{s\bar{i}}\overline{\phi_{s\bar{j}}})_{k}\phi^{i\bar{j}}+(\phi_{s\bar{i}}\overline{\phi_{s\bar{j}}})(\phi^{i\bar{j}})_{k}]=\]
\[
=\phi_{ks\bar{s}}-(\phi_{ks\bar{i}}\overline{\phi_{s\bar{j}}}+\overline{\phi_{\bar{k}s\bar{i}}}\phi_{s\bar{j}})\phi^{i\bar{j}}-(\phi_{s\bar{i}}\overline{\phi_{s\bar{j}}})(\phi^{i\bar{j}})_{k}\]
Next, note that if $h$ is a function with values in the space of
Hermitian matrices and $\partial$ a derivation satisfying Leibniz
rule, then \[
\partial(h^{-1})=-h^{-1}(\partial h)h^{-1}.\]
In particular, if $h(0)=I$ then the following holds at $0:$ \[
(\bar{h}^{-1})_{k\bar{k}}=-\bar{h}_{k\bar{k}}+(\bar{h}_{\bar{k}}\bar{h}_{k}+\bar{h}_{k}\bar{h}_{\bar{k}})\]
Applying this to $h=(\phi_{i\bar{j}})$ (when expanding the term $A$
below) gives \[
c(\phi)_{k\bar{k}}=\phi_{k\bar{k}s\bar{s}}-\left([(\phi_{ks\bar{i}}\overline{\phi_{s\bar{i}}})_{\bar{k}}+\overline{(\phi_{\bar{k}s\bar{i}}}\phi_{s\bar{i}})_{\bar{k}}]-(\phi_{ks\bar{i}}\overline{\phi_{s\bar{j}}}+\overline{\phi_{\bar{k}s\bar{i}}}\phi_{sj})\phi_{\bar{k}i\bar{j}}\right)-A=\]
\[
\phi_{k\bar{k}s\bar{s}}-\left([\phi_{\bar{k}ks\bar{i}}\overline{\phi_{s\bar{i}}}+\phi_{ks\bar{i}}\overline{\phi_{ks\bar{i}}}+\overline{\phi_{k\bar{k}s\bar{i}}}\phi_{s\bar{i}}+\overline{\phi_{\bar{k}s\bar{i}}}\phi_{\bar{k}s\bar{i}}]-(\phi_{ks\bar{i}}\overline{\phi_{s\bar{j}}}+\overline{\phi_{\bar{k}s\bar{i}}}\phi_{s\bar{j}})\phi_{\bar{k}i\bar{j}}\right)-A\]
where \[
A:=(\phi_{s\bar{i}}\overline{\phi_{s\bar{j}}})_{\bar{k}}(\phi^{i\bar{j}})_{k}+(\phi_{s\bar{i}}\overline{\phi_{s\bar{j}}})(\phi^{i\bar{j}})_{k\bar{k}}=-(\phi_{s\bar{i}}\overline{\phi_{s\bar{j}}})_{\bar{k}}\phi{}_{kj\bar{i}}+(\phi_{s\bar{i}}\overline{\phi_{s\bar{j}}})(-(\phi{}_{k\bar{k}j\bar{i}}+2\Re(\phi_{\bar{k}\bar{i}m}\phi_{k\bar{m}j}))\]

\[
=-(\phi_{s\bar{i}\bar{k}}\overline{\phi_{s\bar{j}}}+\phi_{s\bar{i}}\overline{\phi_{s\bar{j}k}})\phi{}_{kj\bar{i}}-(\phi_{s\bar{i}}\overline{\phi_{s\bar{j}}})(\phi{}_{k\bar{k}j\bar{i}}+2\Re(\phi_{\bar{k}\bar{i}m}\phi_{k\bar{m}j}))\]
Hence, \[
c(\phi)_{k\bar{k}}=\phi_{k\bar{k}s\bar{s}}-[\phi_{\bar{k}ks\bar{i}}\overline{\phi_{s\bar{i}}}+\phi_{ks\bar{i}}\overline{\phi_{ks\bar{i}}}+\overline{\phi_{k\bar{k}s\bar{i}}}\phi_{s\bar{i}}+\overline{\phi_{\bar{k}s\bar{i}}}\phi_{\bar{k}s\bar{i}}]+(\phi_{ks\bar{i}}\overline{\phi_{s\bar{j}}}+\overline{\phi_{\bar{k}s\bar{i}}}\phi_{s\bar{j}})\phi_{\bar{k}j\bar{i}}+\]
\[
+(\phi_{s\bar{i}\bar{k}}\overline{\phi_{s\bar{i}}}+\phi_{s\bar{i}}\overline{\phi_{s\bar{i}k}})\phi{}_{kj\bar{i}}+(\phi_{s\bar{i}}\overline{\phi_{s\bar{j}}})(\phi{}_{k\bar{k}j\bar{i}}-2\Re(\phi_{s\bar{i}}\overline{\phi_{s\bar{j}}}\phi_{\bar{k}\bar{i}m}\phi_{k\bar{m}j}),\]
 which finishes the proof of Step 2., 

\emph{Step 3:} end of proof of the theorem for a trivial fibration

Subtracting the formulas from the previous steps gives, due to cancellation
of several terms, $\frac{\partial}{\partial t}c(\phi)-c(\phi)_{k\bar{k}}=$
\[
=\phi_{s\bar{m}\bar{k}}\overline{\phi_{s\bar{m}\bar{k}}}+(\phi_{s\bar{i}}\overline{\phi_{s\bar{j}}})\phi_{\bar{k}\bar{i}m}\phi_{k\bar{m}j}-2\Re(\overline{\phi_{\bar{sk}\bar{m}}}\phi_{s\bar{l}}\phi_{\bar{k}l\bar{m}})=\]
\[
=\sum_{m,k}|\phi_{s\bar{m}\bar{k}}-\sum_{l}\phi_{s\bar{l}}\phi_{\bar{k}\bar{m}l}|^{2}\]
Finally, note that \[
\phi_{s\bar{m}\bar{k}}-\sum_{l}\phi_{s\bar{l}}\phi_{\bar{k}\bar{m}l}=(\phi_{s\bar{m}})_{\bar{k}}-(\phi_{\bar{m}l})_{\bar{k}}\sum_{l}\phi_{s\bar{l}}=(\phi_{s\bar{l}}\phi^{m\bar{l}})_{\bar{k}}=(V_{m})_{\bar{k}}\]
(using $\phi_{i\bar{j}}=\delta_{ij}$ at $z=0)$, where $V=(V_{1},...V_{n})$
is the $(0,1)-$vector field \ref{eq:def of V} expressed in local
normal coordinates. This hence finishes the proof of the theorem in
the case of a trivial fibration.

\emph{Step 4:} equation \ref{eq:heat eq for c in c-y} holds for a
general holomorphic submersion:

Computing locally as before the only new contribution comes from the
derivatives on the local function $\psi_{\Omega}(s)$ defined by formula
\ref{eq:wp as curv}, which appear in the definition of the relative
Kähler-Ricci flow \ref{eq:k-r flow of weights in intro} in the Calabi-Yau
case. Indeed, locally this latter flow may be written as \[
\frac{\partial\phi}{\partial t}=\log\det(\phi_{k\bar{l}})-\psi_{\Omega}(s)\]
 and the only new contribution to the previous calculations hence
come from the term $-(\psi_{\Omega}(s))_{s\bar{s}}$ which appears
in the calculation of $(\frac{\partial\phi}{\partial t})_{s\bar{s}}.$
Combining formulae \ref{eq:def of w-p}, \ref{eq:wp as curv} hence
proves that equation \ref{eq:heat eq for c in c-y} holds locally
on $\mathcal{X}.$ Since, all objects appearing the resulting local
formula have been shown to be globally well-defined this hence finishes
the proof of Step 4.
\end{proof}
Now the maximum principle for parabolic equations \cite{p-w} implies
the following
\begin{cor}
\label{cor:conserv of pos along k-r}Let $\mathcal{L}\rightarrow\mathcal{X}\rightarrow S$
be a line bundle over a fibration as in the previous theorem. 
\begin{itemize}
\item If the fibration is holomorphically non-trivial, then the function
$c(t):=\inf_{X}c(\phi)$ is, for a fixed value on $s,$ increasing
along the relative Kähler-Ricci flow and hence the flow preserves
(semi-) positivity of the curvature of $\phi.$
\item For a holomorpically trivial fibration $\mathcal{X}=X\times S$, with
$\mathcal{L}$ the pull-back of an ample line bundle $L\rightarrow X,$
the flow \emph{improves} the positivity of a generic initial weight
in the following sense: if $\phi_{0}$ is a semi-positively curved
weight on $\mathcal{L}$ over $X\times S$ such that $\partial\phi/\partial s$
does not vanish identically on $X\times\{s\}$ for any $s,$ then
$\phi_{t}$ is \emph{strictly }positively curved on $X\times S$ for
$t>0.$
\item In the general case the (semi-) positivity of the curvature of the
weight on $\phi-t\psi_{\Omega}$ on the $\R-$line bundle $\mathcal{L}-tK_{\mathcal{X}/S}$
is preserved under the flow, i.e. \[
dd^{c}\phi_{t}\geq-t\omega_{WP}\]
for all $t$ (and similarly in the strict case)
\end{itemize}
\end{cor}
\begin{proof}
The first and third point follow from the maximum principle exactly
as in in the proof of Cor \ref{cor:pos along flow in kx-s} below.
The second point is proved as follows: if strict positivity does not
hold then one concludes (see the proof of Cor \ref{cor:pos along flow in kx-s}
below) that $-A_{\phi_{0}}=\overline{\partial}_{X}V_{\phi_{0}}$ vanishes
identically on $X$ for some $s_{0}$, i.e. the corresponding vector
field $V_{\phi_{0}}$ defined by \ref{eq:def of V} is holomorphic
on $X$. But, it is a well-known fact that any such holomorphic vector
field $V^{1,0}$ vanishes identically when $X$ is a Calabi-Yau manifold
and hence $\partial\phi/\partial s$ vanishes identically on $X\times\{s_{0}\}$
giving a a contradiction. The vanishing of $V^{1,0}$ may be proved
as follows: by a Bochner-Weitzenbock formula $V^{1,0}$ is covariantly
constant wrt any Ricci flat metric on $X.$ Moreover, the imaginary
part $V_{I}$ satisfies $\omega_{\phi_{0}}(V_{I},\cdot)=df$ for some
real smooth function $f.$ But since $\omega_{\phi_{0}}^{X}>0$ on
$X\times\{s_{0}\}$ the latter equation forces the vanishing of $V_{I}$
at any point where $f$ achieves it maximum and hence $V_{I}\equiv0$
on $X.$ Similarly, the real part $V_{R}$ of $V^{1,0}$ vanishes
identically (by replacing $df$ with $d^{c}f).$ 
\end{proof}
Of course, in the case of a infinitesimally non-trivial fibration
the inequality in the previous corollary is useless for the limit
$\phi_{\infty},$ but its interest lies in the fact that it gives
a lower bound on the (possible) loss of positivity along the relative
Kähler-Ricci flow, which is independent of the initial data.
\begin{rem}
Throughout the paper we assume, for simplicity, that the initial weight
$\phi_{0}$ has relatively positive curvature, when restricted to
the fibers of the $\mathcal{X}.$ But, as in the previous corolllary,
we do allow $\phi_{0}$ to have merely semi-positive curvature over
the total space $\mathcal{X}.$ However, using recent developments
for the Kähler-Ricci flow (see for example \cite{s-t}) the relative
Kähler-Ricci flows are actually well-defined for any weight $\phi_{0}$
which has merely relatively \emph{semi-}positive curvature and $\phi_{t}$
becomes relatively positively curved for any $t>0.$ Using this result
the previous corollary can be seen to be valid for a general semi-positively
curved initial weight $\phi_{0}.$
\end{rem}

\subsubsection{Evolution of the curvature of the top Deligne pairing}

For a general smooth base $S$ (i.e. not necessarily embedded in $\C)$
the weight $\phi$ on $L$ naturally induces a closed $(1,1)-$form
$\Theta_{\phi}(s)$ on $S$ expressed as \[
\Theta_{\phi}:=\pi_{*}((dd^{c}\phi)^{n+1}/(n+1)!)\]
 Equivalently, for any local holomorphic curve $C\subset S$ with
tangent vector $\frac{\partial}{\partial s}\in TS$ 

\[
\Theta_{\phi}(\frac{\partial}{\partial s},\frac{\partial}{\partial\bar{s}}):=\int_{\chi_{s}}c(\phi)\omega_{\phi}^{n}/n!\]
 where $s\in C$ and $\pi$ is the induced map $\pi:\mathcal{\, X}\rightarrow C.$
Geometrically, the form $\Theta_{\phi}$ on $S$ may be described
as the curvature of the Hermitian holomorphic line bundle $(\mathcal{L},\phi)^{n+1}$
over $S$ defined as the top \emph{Deligne pairing} of the Hermitian
holomorphic line bundle $(\mathcal{L},\phi)\rightarrow\mathcal{X}\rightarrow S$
(see \cite{de}; the relevance of Deligne pairings for Kähler geometry
has been emphasized by Phong-Sturm \cite{p-s}). The form $\Theta_{\phi}$
also appears as a multiple of the curvature of the Quillen metric
on the determinant of the direct image of a certain virtual vector
bundle over $\mathcal{X}$ (see \cite{fs} and references therein).

Similarly, one can define a $(1,1)-$form $\omega_{WP_{\phi}}$ on
$S$ depending on $\phi$ by letting \[
\omega_{WP_{\phi}}(\frac{\partial}{\partial s},\frac{\partial}{\partial\bar{s}}):=\int_{\chi_{s}}|A_{\phi}(s)|^{2}\omega_{\phi}^{n}/n!\]
It can be checked that this yields a well-defined $(1,1)-$form on
$\mathcal{X}.$ Anyhow this latter fact is also a consequence of the
following corollary of the previous theorem. 
\begin{cor}
(same assumptions as in the previous theorem). Let $\Theta_{\phi_{t}}$
be the curvature form on $S$ of the top Deligne pairing of $(\mathcal{L},\phi)\rightarrow\mathcal{X}\rightarrow S,$
where $\phi_{t}$ evolves according to the relative Kähler-Ricci flow
in the Calabi-Yau case. Then \[
\frac{\partial}{\partial t}\Theta_{\phi}(s)=-\pi_{*}(R_{\omega_{\phi}^{X}}(dd^{c}\phi)^{n+1}/(n+1)!)+\omega_{WP_{\phi_{t}}}-\omega_{WP},\]
 where $R_{\omega_{\phi}^{X}}$ denotes the fiber-wise scalar curvature
of the metric $\omega_{\phi}.$ \end{cor}
\begin{proof}
We may without loss of generality assume that $S$ is embedded in
$\C_{s}.$ Then \[
\frac{\partial}{\partial t}\int_{\chi_{s}}c(\phi)\omega_{\phi}^{n}/n!=\int_{\chi_{s}}\frac{\partial}{\partial t}c(\phi)\omega_{\phi}^{n}/n!+\int c(\phi)\frac{\omega_{\phi}^{n-1}}{(n-1)!}\wedge dd^{c}\frac{\partial}{\partial t}\phi.\]
 Now by the definition of the Kähler-Ricci flow in the Calabi-Yau
case \[
\frac{\omega_{\phi}^{n-1}}{(n-1)!}\wedge dd^{c}\frac{\partial}{\partial t}\phi=\frac{\omega_{\phi}^{n-1}}{(n-1)!}\wedge(-Ric(\omega_{\phi}^{X}))=:-R_{\omega_{\phi}^{X}}\omega_{\phi}^{n}/(n+1)!\]
 where we have used the definition of the (normalized) scalar curvature
$R_{\omega_{\phi}^{X}}$ of the Kähler metric $\omega_{\phi}^{X}$
in the last step. Finally, integrating the formula in the previous
theorem finishes the proof of the corollary.\end{proof}
\begin{rem}
\label{rem:todorovs formel}. Note that if the initial weight $\phi$
for the Kähler-Ricci flow is taken so that $\omega_{\phi}$ restricts
to a Ricci flat metric on all fibers of $\mathcal{X},$ then $\phi$
is stationary for the Kähler-Ricci flow and hence the previous corollary
(and the proof of the previous theorem) shows that \[
\int_{\chi_{s}}|A_{\phi}(s)|^{2}\omega_{\phi}^{n}/n=\frac{\partial^{2}\psi_{\Omega}}{\partial s\partial\bar{s}},\]
i.e. $\omega_{WP}=d_{s}d_{s}^{c}i^{n^{2}}\log\int_{\mathcal{X}_{s}}\Omega\wedge\bar{\Omega}.$
Since, by Proposition \ref{pro:horis vf give harmon} below $A_{\phi}(s)$
is harmonic on each fiber $\mathcal{X}_{s}$ with respect to the Ricci
flat restriction $\omega_{\phi}$ this implies the equivalence between
the formulae \ref{eq:def of w-p}\ref{eq:wp as curv}. 
\end{rem}

\subsection{\label{sec:Quantization:-The-Bergman L}Quantization: The Bergman
iteration on $\mathcal{H}_{L}$}

In this section we will specialize and develop the general results
in section \ref{sec:Quantization:-The-Bergman general} to the present
setting where we have fixed a family of probability measure $\mu_{s}$
(independent of $\phi)$ on the fibers $\chi_{s}.$

\subsubsection{Convergence and positivity of the Bergman iteration at a fixed level
$k.$}

The following monotonicity properties were shown by Donaldson \cite{do3}
in the present setting. 
\begin{lem}
\label{lem:monotone along bergman cy}The functionals $-I_{\mu}$
and $\mathcal{L}^{(k)}$ are both increasing along the \emph{Bergman
iteration on $\mathcal{H}_{L}$ with respect to $\mu.$ Moreover,
they are }strictly \emph{increasing at $\phi_{m}^{(k)}$ unless $\phi_{m}^{(k)}$
is stationary.}\end{lem}
\begin{proof}
Since $I_{\mu}$ is affine and in particular concave on the affine
space of all smooth weights the lemma follows immediately from Lemma
\ref{lem:monotone along bergman general}. 
\end{proof}
We can now prove the convergence of the Bergman iteration at a fixed
level $k$ in the present setting.
\begin{thm}
\label{thm:conv of bergman iter at level k }Let $L\rightarrow X$
be an ample line bundle and $\mu$ a fixed volume form on $X$ giving
unit volume to $X.$ Assume given a smooth initial weight $\phi_{0}.$
For any given positive integer $k$ the following holds: In the large
time limit, i.e. when $m\rightarrow\infty,$ the weights $\phi_{m}^{(k)}\rightarrow\phi_{\infty}^{(k)}$
in the $\mathcal{C}^{\infty}-$topology on $X.$ Moreover, in the
relative setting the convergence is locally uniform with respect to
the base parameter $s.$\end{thm}
\begin{proof}
By the previous lemma $-I_{\mu}$ is increasing and by definition
$-\mathcal{F}_{\mu}^{(k)}$ is coercive. Moreover, as shown in \cite{bbgz}
balanced weights are unique modulo scaling and hence all the convergence
criteria in Proposition \ref{pro:conv of bergman iter at level k general}
are satisfied. 
\end{proof}

\subsubsection{\label{sub:Conservation-of-positivity}Conservation of positivity }

Recall that, given a relatively ample line bundle $\mathcal{L}$ over
a fibration $\pi:\,\mathcal{X}\rightarrow S$ as above the corresponding
direct image bundle $\pi_{*}(\mathcal{L}+K_{\mathcal{X}/S})\rightarrow S$
is the vector bundle such that the fiber over $s$ is naturally identified
with the space $H^{0}(X,L+K_{X})$ of all holomorphic $n-$forms $f$
on $X:=\mathcal{X}_{s}$ with values in $L:=\mathcal{L}_{X}$ (as
is well-known this is indeed a vector bundle, which is shown using
vanishing theorems). Moreover, any given weight $\phi$ on $\mathcal{L}$
induces an Hermitian metric on $\pi_{*}(k\mathcal{L}+K_{\mathcal{X}/S})$
whose fiber-wise restriction will be denoted by $Hilb_{L+K_{X}}(\phi):$
\[
Hilb_{L+K_{X}}(\phi)(f,f):=i^{n^{2}}\int_{X}f\wedge\overline{f}e^{-\phi}\]
The point is that there is no need to specify an integration measure
$\mu$ thanks to the twist by the relative canonical line bundle $K_{\mathcal{X}/S}.$
We will have great use for the following recent results of Berndtsson.
\begin{thm}
\label{thm:(Berndtsson).-Let-}(Berndtsson ) Let $\pi:\mathcal{\, X}\rightarrow S$
be a proper holomorphic submersion and let $\mathcal{L}$ be a relatively
ample line bundle over $\mathcal{X}$ equipped with a smooth weight
with semi-positive curvature. Then
\begin{itemize}
\item \cite{bern0} the curvature of the Hermitian vector bundle over $S$
defined as the direct image bundle $\pi_{*}(\mathcal{L}+K_{\mathcal{X}/S})$
is semi-positive in the sense of Nakano (and in particular in the
sense of Griffiths). 
\item (see Theorem 1.2 in \cite{bern2} and the subsequent discussion) The
vector bundle $\pi_{*}(\mathcal{L}+K_{\mathcal{X}/S})$ has\emph{
strictly} positive curvature in the sense of Griffiths if \emph{either}
the curvature form of $\phi$ is strictly positive over all of $\mathcal{\, X}$\emph{
or} strictly positive along the fibers of $\pi:\mathcal{\, X}\rightarrow S$
and the fibration is infinitesimally non-trivial (i.e. the Kodaira-Spencer
 classes are non-trivial for all $s\in S)$.
\end{itemize}
\end{thm}
We will only use the following simple consequence of Berndtsson's
theorem (compare \cite{bern0,b-p}):
\begin{cor}
\label{cor:bernd}Under the assumptions in the first point of the
previous theorem we have\begin{equation}
dd^{c}(FS^{(k)}\circ Hilb_{k\mathcal{L}+K_{\mathcal{X}/S}})(\phi)\geq0\label{eq:pos from bernd thm}\end{equation}
and the inequality is strict under the assumptions in the second point
of the theorem.\end{cor}
\begin{proof}
We will denote the line bundle $k\mathcal{L}+K_{\mathcal{X}/S}$ over
$\mathcal{X}$ by $\mathcal{F}$ and the vector bundle $\pi_{*}(\mathcal{F})$
over $S$ by $E$ (and its dual by $E^{*}).$ First note that the
weight on $\mathcal{F}$ that we are interested in may be written
as \begin{equation}
(FS\circ Hilb_{\mathcal{F}})(s,x_{s})=\log\sup_{f_{s}\in E_{s}}\frac{|f_{s}(x_{s})|^{2}}{\left\Vert f(x_{s})\right\Vert ^{2}}=\log|\Lambda_{(s,x_{s})}|^{2}\label{eq:pf of cor bernd}\end{equation}
 where $\Lambda_{(s,x_{s})}$ is the element in $E_{s}^{*}\otimes$
$\mathcal{F}_{s}$ defined by \[
(\Lambda_{(s,x_{s})}f_{s}):=f_{s}(x_{s})\]
Let now $t\mapsto(s_{t},x_{s_{t}})$ be a local holomorphic curve
in $\mathcal{X}$ with $t\in\Delta$ (the unit-disc). Trivializing
$\mathcal{F}$ in a neighborhood of the previous curve we may pull
back $\Lambda_{(s,x_{s})}$ to a holomorphic section $\Lambda_{t}$
of $E^{*}$ over the unit-disc and identify the weight defined by
\ref{eq:pf of cor bernd} with a function $\log|\Lambda_{t}|^{2}$
on $\Delta.$ We have to prove that this latter function is (strictly)
psh. But this follows from the following well-known fact: a vector
bundle $E\rightarrow\Delta$ is (strictly) positive in the sense of
Griffiths iff $\log(\left\Vert \Lambda_{t}\right\Vert ^{2})$ is (strictly)
subharmonic on $\Delta$ where $\Lambda$ is any non-trivial holomorphic
section of the dual vector bundle $E^{*}.$ For example, to get the
required (strict) subharmonicity one just notes that, after a standard
computation, \[
\frac{\partial^{2}\log(\left\Vert \Lambda_{t}\right\Vert ^{2})}{\partial t\partial\bar{t}}_{t=0}\geq-\frac{\Theta_{E^{*}}(\Lambda_{0},\Lambda_{0})}{\left\Vert \Lambda_{0}\right\Vert ^{2}},\]
 where $\Theta_{E^{*}}$ at $t$ is the Hermitian endomorphism of
$E_{t}^{*}$ representing the curvature of $E.$ By the previous theorem
$\Theta_{E}$ is (strictly) positive which is equivalent to $\Theta_{E*}$
being (strictly) negative and the corollary hence follows from the
previous inequality.
\end{proof}
We next obtain  a {}``quantized'' version of Corollary \ref{cor:conserv of pos along k-r}.
\begin{cor}
\label{cor: conserv of pos. along bergman it c-y}Let $\pi:\mathcal{\, X}\rightarrow S$
be a proper holomorphic submersion with Calabi-Yau fibers and let
$\mathcal{L}$ be a relatively ample line bundle over $\mathcal{X}.$ 
\begin{itemize}
\item When $\pi$ is holomorphically trivial the relative Bergman iteration
preserves semi- positivity of the curvature of $\phi.$
\item In the case of a general submersion with Calabi-Yau fibers , \[
dd^{c}\phi_{(m)}^{(k)}\geq-\frac{m}{k}\omega_{WP}\]
for all $m.$
\end{itemize}
\end{cor}
\begin{proof}
For simplicity first consider the case of a trivial fibration. Fix
a holomorphic $n-$form $\Omega$ on $X:=\mathcal{X}_{0}$ trivializing
$K_{X}.$ Under the assumption that $\mathcal{X}\rightarrow S$ is
holomorphically trivial $\Omega$ extends to a holomorphic $n-$form
on all of $\mathcal{X}$ such that $\psi_{\Omega}:=\log\int_{\mathcal{X}_{s}}i^{n^{2}}\Omega\wedge\bar{\Omega}$
is independent of $s.$ In this notation \[
Hilb^{(k)}((\phi(s,\cdot))(f,f):=\int_{\mathcal{X}_{s}}|f|^{2}e^{-(k\phi(s,\cdot)-\psi_{\Omega}(s))}i^{n^{2}}\Omega\wedge\bar{\Omega}\]
Now consider the fiber-wise isomorphism \[
j:\, H^{0}(X,kL)\rightarrow H^{0}(X,kL+K_{X}),\,\,\, j(f)=f\otimes\Omega,\]
which clearly satisfies $Hilb^{(k)}((\phi(s,\cdot)=e^{\psi_{\Omega}}j^{*}Hilb_{L+K_{X}}(\phi(s,\cdot)).$
This means that, up to a multiplicative constant independent of $s,$
the map $j$ is an \emph{isometry }when $H^{0}(X_{s},kL+K_{X_{s}})=H^{0}(\mathcal{X},k\mathcal{L}+K_{\mathcal{X}/S})_{\mathcal{X}_{s}}$
is equipped with its natural Hermitian product. In particular, by
\ref{eq:pos from bernd thm}, \[
dd^{c}\phi\geq0\implies dd^{c}(FS^{(k)}\circ Hilb^{(k)})(\phi)\geq0\]
Iterating hence proves the first point in the statement of the corollary.
Finally, for a general submersion the same argument gives, but now
taking into account the fact that $\psi_{\Omega}$ depends on $s$
that \[
dd^{c}\phi\geq0\implies dd^{c}(FS^{(k)}\circ Hilb^{(k)})(\phi)\geq-dd^{c}\psi_{\Omega}(s)/k:=-\omega_{FS}(s)/k\]
using formula \ref{eq:wp as curv} in the last equality. Replacing
$\phi$ with $FS^{(k)}\circ Hilb^{(k)}-\psi_{\Omega}(s)$ and iterating
hence finishes the proof of the corollary.
\end{proof}

\subsubsection{Convergence towards the Kähler-Ricci flow}

The following very simple proposition will turn out to be very useful:
\begin{prop}
\label{pro:bergman iter decr sup cy}The following monotonicity holds
for the Bergman iteration at level $k$ (with respect to $\mu).$
Assume that $\phi_{m}^{(k)}\leq\psi_{m}^{(k)}.$ Then $\phi_{m+1}^{(k)}\leq\psi_{m+1}^{(k)}.$
In particular, the Bergman iteration decreases the distance in $\mathcal{H}_{L}$
defined with respect to the sup-norm: $d(\phi,\psi):=\sup_{X}|\phi-\psi|.$ \end{prop}
\begin{proof}
By definition we have \[
\phi_{m+1}^{(k)}=\phi_{m}^{(k)}+\frac{1}{k}\log\rho^{(k)}(k\phi_{m}^{(k)})=\frac{1}{N_{k}}\sum_{i}|f_{i}|^{2}.\]
By a well-known identity for Bergman kernels \[
\sum_{i=1}|f_{i}|^{2}(x)=\sup_{f\in H^{0}(X,kL)}(|f(x)|^{2}/\int_{X}|f|^{2}e^{-k\phi_{m}}d\mu).\]
But this latter expression is clearly monotone in $\phi_{m}$ proving
the first statement of the proposition. As for the last statement
just let $C:=\sup_{X}|\phi_{m}^{(k)}-\psi_{m}^{(k)}|$ so that \[
\phi_{m}^{(k)}\leq\psi_{m}^{(k)}+C,\,\,\,\,\,\psi_{m}^{(k)}\leq\phi_{m}^{(k)}+C.\]
 Applying the first statement of the proposition hence finishes the
proof.\end{proof}
\begin{rem}
The previous proposition can be seen as a {}``quantum'' analog of
the corresponding result for the Kähler-Ricci flow \ref{eq:k-r flow of weights in intro},
which follows directly from the maximum principle for the Monge-Ampère
operator and its parabolic analogue. 
\end{rem}
Now we can prove the following Theorem which is one of the main results
in this paper. 
\begin{thm}
\label{thm:conv of bergman iter to ricci cy}Let $L\rightarrow X$
be an ample line bundle and $\mu$ a volume form on $X$ giving unit
volume to $X.$ Fix a smooth weight $\phi_{0}$ on $L$, whose curvature
form is fiber-wise strictly positive and consider the corresponding
Bergman iteration $\phi_{m}^{(k)}$ at level $k$ and discrete time
$m,$ as well as the Kähler Ricci flow $\phi_{t}$ - both defined
with respect to $\mu.$ Then there is a constant $C$ such that \[
\sup_{X}|\phi_{m}^{(k)}-\phi_{m/k}|\leq Cm/k^{2}\]
 In particular, if $m_{k}$ is a sequence such that $m_{k}/k\rightarrow t,$
then \[
\phi_{m_{k}}^{(k)}\rightarrow\phi_{t}\]
 uniformly on $X.$ Moreover, in the relative setting $C$ is locally
bounded in the base parameter $s$ if $\mu$ depends smoothly on $s.$\end{thm}
\begin{proof}
Write $\psi_{k,m}=\phi_{m/k}$ and $F^{(k)}(\psi)=\frac{1}{k}\log\rho^{(k)}(\psi).$

\emph{Step 1: }The following holds for all $(k,m)$ \[
\psi_{k,m+1}-\psi_{k,m}=F^{(k)}(\psi_{k,m})+O(1/k^{2}),\]
 where the error term is uniform in $(k,m)$ (and we will in the following
take that as a definition of $O(1/k)$ etc). 

To prove this we write the lhs as \[
\frac{1}{k}(\frac{\phi_{m/k+1/k}-\phi_{m/k}}{1/k})=\frac{1}{k}(\frac{\partial\phi_{t}}{\partial t}_{t=m/k}+O(1/k))\]
 using that $|\frac{\partial^{2}\phi_{t}}{\partial^{2}t}|\leq C$
on $X\times[0,T]$ by Theorem \ref{thm:cao}. More precisely, by the
mean value theorem the error term $O(1/k)$ may be written as $\frac{1}{k}\frac{\partial^{2}\phi_{t}}{\partial^{2}t}(\zeta)/2$
for some $\zeta\in[0,1/k].$

Since $\phi_{t}$ evolves according to the Kähler-Ricci flow this
means that \[
\psi_{k,m+1}-\psi_{k,m}=\frac{1}{k}\log(\frac{(dd^{c}\phi_{m/k})^{n}/n!}{\mu})+O(1/k^{2}).\]
But by Proposition \ref{pro:bouche tian-1} we have that \[
F^{(k)}(\phi_{m/k})=\frac{1}{k}\log(\frac{(dd^{c}\phi_{m/k})^{n}/n!}{\mu})+O(1/k^{2}),\]
 where the error term is uniformly bounded in $(m,k)$ for $m/k\leq T$
by Theorem \ref{thm:cao}. In fact, as is well-known the uniform estimates
\ref{eq:stab estim} on the {}``space-derivatives'' of $\phi_{t}$
in Theorem \ref{thm:cao} also hold for all time-derivatives $d^{r}\phi_{t}/d^{r}t$
(and in particular for $r=1$ and $r=2$ used above). This is well-known
and shown by differentiating the flow equation with respect to time
and applying the maximum principle repeatedly. Hence, $T$ may be
taken to be equal to infinity, which finishes the proof of Step 1.

\emph{Step 2: }Given step 1 and the fact that the Bergman iteration
decreases the sup norm, we have \begin{equation}
\sup_{X}|\phi_{m}^{(k)}-\psi_{k,m}|\leq Cm/k^{2}\label{eq:induction hyp}\end{equation}
We will prove this by induction over $m$ (for $k$ fix) the statement
being trivially true for $m=0.$ By Step 1 there is a uniform constant
$C$ such that \[
\sup_{X}|\psi_{k,m+1}-(\psi_{k,m}+F^{(k)}(\psi_{k,m})|\leq C(1/k^{2}).\]
for all $(m,k).$ Now we fix the integer $k$ and assume as an induction
hypothesis that \ref{eq:induction hyp} holds for $m$ with $C$ the
constant in the previous inequality. By Proposition \ref{pro:bergman iter decr sup cy},
\[
\sup_{X}|(\psi_{k,m}+F^{(k)}(\psi_{k,m}))-(\phi_{m}^{(k)}+F^{(k)}(\phi_{m}^{(k)})|\leq\sup_{X}|(\psi_{k,m}-\phi_{m}^{(k)}|\leq Cm/k^{2}\]
with the same constant $C$ as above, using the induction hypothesis
in the last step. Combining this estimate with the previous inequality
gives \[
\sup_{X}|(\psi_{k,m+1}-\phi_{m+1}^{(k)}|\leq Cm/k^{2}+C/k^{2}\]
proving the induction step and hence Step 2.
\end{proof}
Of course, it seems natural to expect that $\mathcal{C}^{\infty}-$convergence
holds but we leave this problem for the future.

Combining the previous corollary with Theorem \ref{thm:conv of bergman iter to ricci cy}
and the variational principle in \cite{bbgz} (the $\mathcal{C}^{\infty}-$convergence
rather uses \cite{ke,wa} ) now gives the following 
\begin{cor}
\label{cor:conv of balanced etc cy}The conservation of semi-positivity
of the curvature of $\phi_{t}$ in Corollary \ref{cor:conserv of pos along k-r}
holds. Moreover, for a fixed initial data $\phi_{0}=\phi_{0}^{(k)}\in\mathcal{H}_{L}$
the following convergence results holds for the Bergman iteration
$\phi_{m}^{(k)}:$
\begin{itemize}
\item For any sequence $m_{k}$ such that $m_{k}/k\rightarrow\infty$ the
convergence $\phi_{m_{k}}^{(k)}\rightarrow\phi_{\infty}$ holds in
the $L^{1}-$ topology on $X.$ Moreover, if it is also assumed that
$m_{k}/k^{2}\rightarrow0$ then the convergence holds in the $C^{0}-$topology. 
\item The balanced weights $\phi_{\infty}^{(k)}:=\lim_{m\rightarrow\infty}\phi_{m}^{(k)}$
at level $k$ converge, when $k\rightarrow\infty,$ in the $\mathcal{C}^{\infty}-$topology,
to the weight $\phi_{\infty}$ which is the large time limit of the
corresponding Kähler-Ricci flow (and in particular a solution to the
corresponding inhomogeneous Monge-Ampère equation).
\end{itemize}
In the relative case the convergence holds fiberwise locally uniformly
with respect to the base parameter $s.$\end{cor}
\begin{proof}
The first statement follows immediately by combining Theorem \ref{thm:conv of bergman iter to ricci cy}
and the previous corollary, since semi-positivity is preserved under
uniform limits of weights. Hence, we turn to the proof of the first
point. It is based on the following inequalities: \begin{equation}
\limsup_{k\rightarrow\infty}I_{\mu}(\phi_{m_{k}}^{(k)})\leq I_{\mu}(\phi_{\infty}),\,\,\,\,\,\liminf_{k\rightarrow\infty}\mathcal{E}(\phi_{m_{k}}^{(k)})\geq\mathcal{E}(\phi_{\infty}).\label{eq:inequal pf second point}\end{equation}
To prove these inequalities take a sequence $m_{k}'$ such that $m_{k}'/k\rightarrow t$
and $m_{k}'\leq m_{k}.$ By monotonicity (Lemma \ref{lem:monotone along bergman cy})
\[
I_{\mu}(\phi_{m{}_{k}}^{(k)})\leq I_{\mu}(\phi_{m'{}_{k}}^{(k)})\]
Hence, letting $k\rightarrow\infty$ and using that $\phi_{m'{}_{k}}^{(k)}\rightarrow\phi_{t}$
uniformly (by Theorem \ref{thm:conv of bergman iter to ricci cy})
gives \[
\limsup_{k\rightarrow\infty}I_{\mu}(\phi_{m_{k}}^{(k)})\leq I_{\mu}(\phi_{t})\]
Finally, letting $t\rightarrow\infty$ and using Theorem \ref{thm:cao}
proves the first inequality in \ref{eq:inequal pf second point}.
As for the second inequality in \ref{eq:inequal pf second point},
it is similarly proved by noting that, by monotonicity, \[
\mathcal{L}^{(k)}(\phi_{m'_{k}}^{(k)})\leq\mathcal{L}^{(k)}(\phi_{m{}_{k}}^{(k)}).\]
To proceed we will use that $\psi_{k}\rightarrow\psi$ uniformly in
$\mathcal{H}_{L}$ implies that \[
\mathcal{L}^{(k)}(\psi_{k})\rightarrow\mathcal{E}(\psi)\]
To see this recall that this is well-known when $\psi_{k}=\psi$ for
all $k$ (as follows for example from Proposition \ref{pro:bouche tian-1},
saying that the convergence holds for the differentials $d\mathcal{L}^{(k)}$
and $d\mathcal{E};$ for more general convergence results see \cite{b-b}).
But then the general case follows easily from the fact that $\mathcal{L}^{(k)}$
is monotone in the argument $\psi$ and scaling equivariant. Hence,
letting $k\rightarrow\infty$ gives, since $\psi_{k}:=\phi_{m'{}_{k}}^{(k)}\rightarrow\phi_{t}$
uniformly, that \[
\mathcal{E}(\phi_{t})\leq\liminf_{k\rightarrow\infty}\mathcal{L}^{(k)}(\phi_{m{}_{k}}^{(k)}).\]
Finally, the proof of the second inequality in \ref{eq:inequal pf second point}
is finished by using that (as shown in \cite{bbgz}) for any sequence
$(\psi_{k})$ in $\mathcal{H}_{L}$ \[
\limsup_{k\rightarrow\infty}\mathcal{L}^{(k)}(\psi_{k})\leq\liminf_{k\rightarrow\infty}\mathcal{E}(\psi_{k})\]
Now, adding up the two inequalities in \ref{eq:inequal pf second point}
gives \[
\liminf_{k\rightarrow\infty}\mathcal{F}(\phi_{m_{k}}^{(k)})\geq\mathcal{F}(\phi_{\infty}).\]
 But then it follows from the variational results in \cite{bbgz}
that \begin{equation}
\lim_{k\rightarrow\infty}\mathcal{F}(\phi_{m_{k}}^{(k)})=\mathcal{F}(\phi_{\infty})\label{eq:f function along appr}\end{equation}
 and \begin{equation}
dd^{c}\phi_{m_{k}}^{(k)}\rightarrow dd^{c}\phi_{\infty}\label{eq:conv of curv of appr}\end{equation}
in the weak topology of currents. Next, note that by the inequalities
\ref{eq:inequal pf second point} the sequence $\phi_{m_{k}}^{(k)}$
is contained in a compact subset of \emph{$\mathcal{H}_{L}$} equipped
with the $L^{1}-$topology (compare the proof of Lemma \ref{lem:compactness of kr flow})
and hence we may assume (perhaps after passing to a subsequence) that
$\phi_{m_{k}}^{(k)}\rightarrow\psi$ in the $L^{1}-$topology. But
then the convergence in \ref{eq:conv of curv of appr} forces $\psi=\phi_{\infty}+C$
for some constant $C.$ Finally, it will hence be enough to prove
that $C=0.$ To this end note that combining \ref{eq:f function along appr}
and the inequalities \ref{eq:inequal pf second point} shows that
the latter inequalities are in fact equalities. In particular, \[
\lim_{_{k\rightarrow\infty}}\mathcal{E}(\phi_{m_{k}}^{(k)})=\mathcal{E}(\phi_{\infty})\]
 By the scaling\emph{ }equivariance of $\mathcal{E}$ it hence follows
that $C=0,$ which finishes the proof of the first point. If one assumes
that $m_{k}/k^{2}\rightarrow0$ then it follows immediately from combining
Theorem \ref{thm:cao} and Theorem \ref{thm:conv of bergman iter to ricci cy}
that the convergence holds uniformly on $X$, i.e. in the $C^{0}-$topology. 

To prove the second point in the statement of the corollary note that
replacing $\phi_{m_{k}}^{(k)}$ by $\phi_{\infty}^{(k)}$ in the previous
argument gives, just as before, that $\phi_{\infty}^{(k)}\rightarrow\phi_{\infty}$
in the $L^{1}-$topology. Moreover, since it was shown in \cite{ke,wa}
that the convergence of the corresponding curvature forms holds in
the $\mathcal{C}^{\infty}-$topology this proves the second point. 
\end{proof}

\section{\label{sec:The-(anti-)-canonical}The (anti-) canonical setting }

In this section we will consider another particular case of the general
setting in section \ref{sec:The-general-setting} arising when the
line bundle $L:=rK_{X}$ is ample, where $r=1$ or $r=-1$ (for any
fiber $X$ of the fibration). Hence, $X$ is necessarily of general
type in the former {}``positive'' case and a Fano manifold in the
latter {}``negative'' setting. We will also refer to these two different
settings as the $\pm K_{X}-$settings. 

By the very definition of the canonical line bundle any weight $\phi$
on $\pm K_{X}$ determines a canonical scale invariant probability
measure $\mu_{\pm}(\phi)$ on $X,$ where \[
\mu_{\pm}(\phi):=e^{\pm\phi}/\int_{X}e^{\pm\phi}\]
 (with a slight abuse of notation), so that $\mu_{\pm}(\phi+c)=\mu_{\pm}(\phi).$
Equivalently, $\mu_{\pm}(\phi)$ may be identified with the one-form
o\emph{n} \emph{$\mathcal{H}_{\pm K_{X}}$ }obtained as the differential
of the following functional $I_{\pm}(\phi)$ on \emph{$\mathcal{H}_{\pm K_{X}}:$}\[
I_{\pm}(\phi):=\pm\log\int_{X}e^{\pm\phi},\,\,\,\mu_{\pm}(\phi)=dI_{\pm}.\]
A characteristic feature of the $\pm K_{X}-$setting is that the anti-derivative
$I_{\pm}$ is canonically defined (i.e. not only up to scaling). As
a consequence there is a canonical normalization condition for weights
that will occasionally be used below, namely the condition that $I_{\pm}(\phi)=0.$ 

We will also have use, as before, for the equivariant functional \[
\mathcal{F}_{\pm}:=\mathcal{E}-I_{\pm},\]
 where $\mathcal{E}$ is the functional defined in section \ref{sub:The-measure-mu and functionals}
(with respect to a fixed reference weight in $\pm K_{X})$ %
\footnote{Note that $\mathcal{F}_{\pm}:$ is minus the functional introduced
by Ding-Tian \cite{ti}.%
} Note that the critical points of $\mathcal{F}_{\pm}$ on \emph{$\mathcal{H}_{\pm K_{X}}$}
are the\emph{ Kähler-Einstein weights} $\phi,$ i.e. the weights such
that $\omega_{\phi}$ is a Kähler-Einstein metric on $X$ (compare
Theorem \ref{thm:perelman etc} below).

It will also be important to consider a\emph{ non-normalized} variant
of $\mu_{\pm}(\phi)$ defined by 

\[
\mu'_{\pm}(\phi):=e^{\pm\phi}\]
 (which is the differential of the \emph{non}-equivariant functional
$\phi\mapsto\int e^{\pm\phi}).$ In the sequel we will refer to the
two different settings defined by $\mu_{\pm}(\phi)$ and $\mu'_{\pm}(\phi)$
as the\emph{ normalized $\pm K_{X}-$setting} and the \emph{non-normalized
$\pm K_{X}-$setting, respectively. }It should be pointed out that
it is the latter one which usually appears in the literature on the
Kähler-Ricci flow (see for example \cite{ca,t-z,p-s-s}).

\subsection{The relative Kähler-Ricci flow }

According to the general construction in section \ref{sec:The-general-setting}
each particular setting introduced above comes with an associated
relative Kähler-Ricci flow. For future reference we will write out
the fiber-wise flow in the non-normalized $\pm K_{X}-$setting in
local holomorphic coordinates:\begin{equation}
\frac{\partial\phi}{\partial t}=\log\det(\frac{1}{\pi}\frac{\partial^{2}\phi}{\partial z_{i}\partial\bar{z_{j}}})/n!-(\pm\phi)\label{eq:k-r flow on weig in non-normal local}\end{equation}
 The normalized and the non-normalized setting induce the same evolution
of the fiber-wise curvature forms $\omega_{t}:$ \begin{equation}
\frac{\partial\omega_{t}}{\partial t}=-\mbox{Ric}\omega_{t}-\pm\omega_{t},\label{eq:k-r flow for k metrics in kx}\end{equation}
in $c_{1}(\pm K_{X})$ %
\footnote{In the literature this latter flow of Kähler forms is sometimes referred
to as the normalized Kähler-Ricci flow as opposed to Hamilton's original
flow, but our use of the term normalized is different and only applies
on the level of weights on $L.$%
}

In particular, if $\omega_{t}$ converges to $\omega_{\infty}$ in
the large time limit, then $\omega_{\infty}$ is necessarily a Kähler-Einstein
metric, which is of \emph{negative} scalar curvature in the $K_{X}-$setting
and \emph{positive }scalar curvature in the $-K_{X}-$setting. 

The main virtue of the Kähler-Ricci flow in the normalized setting
as compared with the non-normalized one is that the first one is convergent
precisely when the flow of curvature forms $\omega_{t}$ is. On the
other hand, as will be seen later the flow in the non-normalized setting
(and its quantized version) has better monotonicity and positivity
properties.
\begin{thm}
\label{thm:perelman etc}The Kähler-Ricci flow in the $\pm K_{X}-$settings
always exists and is smooth on $X\times[0,\infty[.$ More precisely,
all the analytical assumptions in section \ref{sub:The-relative-K=0000E4hler-Ricci general}
are satisfied. In the normalized $K_{X}-$ setting it converges to
a Kähler-Einstein metric of negative scalar curvature. In the $-K_{X}-$setting
the flow converges to a Kähler-Einstein metric of positive scalar
curvature under the assumptions that $H^{0}(TX)=0$ and $X$ a priori
admits a Kähler-Einstein metric\emph{. }Furthermore, in the relative
case the convergence is locally uniform with respect to the base parameter
$s.$\emph{ }
\end{thm}
A part from the uniqueness statement, the first part of the previous
theorem is due to Cao \cite{ca}. The convergence on the level of
Kähler metrics in the Fano case, i.e. when $-K_{X}$ is ample, was
proved by Perelman (unpublished) and Tian-Zhu \cite{t-z}. The convergence
on the level of weights then follows directly from Proposition \ref{pro:conv of curv implies conv of weights}
and the known coercivity of the functionals $-\mathcal{F}_{\pm};$
the coercivity of $-\mathcal{F}_{+}$ follows immediately from Jensen's
inequality, while the coercivity of $-\mathcal{F}_{-}$ was shown
in \cite{p-s-s-w}, confirming a conjecture of Tian. The uniqueness
in the difficult case of $-K_{X}$ is due to Bando-Mabuchi (for a
comparatively simple proof see \cite{bbgz}). 
\begin{rem}
\label{rem:stab}The first key analytical ingredient in the proof
of the convergence of the flow of Kähler metric $\omega_{t}$ in the
case (i.e the $-K_{X}$- setting) is an estimate of Perelman saying
that the Ricci potential $h_{t}$ of $\omega_{t}$, when suitably
normalized, is aways bounded along the Kähler-Ricci flow for $\omega_{t}$
(see \cite{t-z,p-s-s}). In fact, in the present notation $h_{t}$
coincides (modulo signs) with the time derivative of $\phi_{t}$ evolving
according to the\emph{ normalized} Kähler-Ricci flow in the $-K_{X}-$setting.
The second key ingredient is the fact that the existence of a Kähler-Einstein
metric implies that $-\mathcal{F}_{+}$ is proper (and conversely
\cite{ti,p-s-s-w}). As is well-known there are, in general, obstructions
to existence of Kähler-Einstein metrics in the $-K_{X}-$setting.
According to a conjecture of Yau the existence of a Kähler-Einstein
metric should be equivalent to a suitable notion of algebraic stability
(in the sense of Geometric Invariant Theory). From this point of view
the properness (or coercivity) assumption on the functional $-\mathcal{F}_{+}$
can be considered as an \emph{analytic} stability \cite{ti}. 
\end{rem}
It will be convenient to make the following
\begin{defn*}
A weight $\phi_{KE}$ on $\pm K_{X}$ will be called a\emph{ normalized
Kähler-Einstein weight }if $I_{\pm}(\phi_{KE})=0,$ or equivalently
if $e^{\pm\phi_{KE}}=\omega_{KE}^{n}/n!.$
\end{defn*}
Hence, there is precisely one normalized Kähler-Einstein weight on
$+K_{X}$ when it is ample. The following simple corollary of the
previous theorem and the subsequent remark illustrates the difference
between the normalized and non-normalized settings.
\begin{cor}
\label{cor:conv of non-normal flow}In the $+K_{X}-$setting the non-normalized
flow \ref{eq:k-r flow on weig in non-normal local} always converges
to the normalized Kähler-Einstein weight. \end{cor}
\begin{proof}
Write $\phi_{t}'$ for the evolution under the Kähler-Ricci flow in
the non-normalized $K_{X}-$setting so that \[
\phi_{t}'=\phi_{t}+C_{t},\]
 where $C_{t}$ is a constant for each $t.$ Since $\phi\mapsto(dd^{c}\phi)^{n}$
is invariant under scalings, comparing the two flow equations gives
\begin{equation}
\frac{\partial C_{t}}{\partial t}=-C_{t}-I_{+}(\phi{}_{t}).\label{eq:flow of cst in norm setting}\end{equation}
 Let $D_{t}:=C_{t}-I(\phi_{t}).$ Then we get \[
\frac{\partial D_{t}}{\partial t}=-D_{t}+\epsilon_{t}.\]
 where $\epsilon_{t}:=\frac{\partial I_{+}(\phi_{t})}{\partial t}.$
In the $+K_{X}-$setting Theorem \ref{thm:perelman etc} implies that
$\epsilon_{t}\rightarrow0$ and hence it follows for elementary reasons
that $D_{t}\rightarrow0.$ Indeed, assume for a contradiction that
$D_{t}$ does not converge to $0.$ Then $\frac{\partial\log|D_{t}|}{\partial t}\rightarrow-1$
i.e. $|D_{t}|\leq C_{\delta}e^{-t(1-\delta)}\rightarrow0$ for $0<\delta<<1,$
giving a contradiction. Finally, in the non-normalized $-K_{X}-$setting
it was shown in \cite{p-s-s} (building on \cite{ct}) that there
is a constant $c_{0}$ such that $\phi_{t}'$ converges. But then
it follows immediately from combining the scaling invariance of $\phi\mapsto(dd^{c}\phi)^{n}$
and the scaling equivariance of $\mu_{-}\text{'}$ that the flow diverges
exponentially for any other choice of constant $c_{0}.$
\end{proof}
As for the non-normalized $-K_{X}-$setting we make the following 
\begin{rem}
I\label{rem:div of flow}n the non-normalized $-K_{X}-$setting (under
the assumptions in the previous theorem) it was shown in \cite{p-s-s}
(building on \cite{ct}) that the flow converges when the initial
weight $\phi_{0}$ is replaced by $\phi_{0}+c_{0}$ for a unique constant
$c_{0}.$ The argument in the proof of the previous corollary then
gives that for a \emph{generic} initial weight the flow is divergent.
\end{rem}

\subsection{Weil-Petersson geometry }

As before we may in the following assume that the base $S$ is embedded
in $\C.$ In the relative $\pm K_{X}-$setting the (generalized) Weil-Petersson
 form  $\omega_{WP}$ on $S$ was introduced in \cite{ko} (see also
\cite{fs} for generalizations): \begin{equation}
\omega_{WP}(\frac{\partial}{\partial s},\frac{\partial}{\partial\bar{s}}):=\left\Vert A_{KE}\right\Vert _{\omega_{KE}}^{2},\label{eq:general def of wp as harmon}\end{equation}
where $A_{KE}$ denotes the unique representative in the Kodaira-Spencer
class $\rho(\frac{\partial}{\partial s})\in H^{1,0}(T^{1,0}\mathcal{X}_{s})$
which is harmonic with respect to the Kähler-Einstein metric on $\mathcal{X}_{s}$
and the $L^{2}-$norm is computed with respect to this latter metric.
In fact, as showed in \cite{fs} (Proposition 4.12) $A_{KE}=-\bar{\partial}V_{\omega_{s}},$
where $V_{\omega_{s}}$ is the local vector field defined by formula
\ref{eq:def of V}. This is a consequence of the following proposition
proved in \cite{fs}. 
\begin{prop}
\label{pro:horis vf give harmon}Let $\pi:\mathcal{\, X}\rightarrow S$
be a proper holomorphic submersion and $\omega_{s}$ a smooth family
of $2-$forms on the fibers $\mathcal{X}_{s}$ such that $\omega_{s}$
is Kähler-Einstein on $\mathcal{X}_{s}.$ Then $A_{\omega_{s}}$ is
the unique element in $H^{0,1}(T\mathcal{X}_{s})$ which is harmonic
with respect to $\omega_{s}.$
\end{prop}
Note that {}``harmonic'' lifts of vector fields were previously
used by Siu\cite{si0} in the context of Weil-Petersson geometry.
\begin{rem}
When the relative dimension is one the space $H^{0,1}(T\mathcal{X}_{s})$
is isomorphic to $H^{1,0}((T\mathcal{X}_{s})^{*})=H^{0}(2K_{\chi_{s}})$
under Serre duality. Hence, the Weil-Petersson  form  as defined in
terms of harmonic representatives then coincides with the metric on
$\mathcal{X}$ introduced by Weil in the case when $\mathcal{X}$
is the universal family over Teichmuller space. As conjectured by
Weil and subsequently proved by Ahlfors this latter $(1,1)-$form
is \emph{closed} and hence Kähler. In the higher dimensional case,
it was observed in \cite{fs} that the Kähler property of $\omega_{WP}$
as defined by \ref{eq:general def of wp as harmon} follows immediately
from Corollary \ref{corpush forward formel f=0000F6r w-p i kx} below
\end{rem}
By an application of the implicit function theorem (in appropriate
Banach spaces) the smoothness of the family $\omega_{s}$ (and of
the associated normalized weight) in the previous proposition is automatic
in the $+K_{X}-$case case, as well as in the $-K_{X}$ case if there
are no non-trivial holomorphic vector fields tangential to the fibers
of the fibration (see Theorem 6.3 in \cite{fs}). 

Now we can prove the following variant of Theorem \ref{thm:heat eq for c in c-y}.
\begin{thm}
\label{thm: heat eq in kx-setting}Let $\pi:\mathcal{\, X}\rightarrow S$
be a proper holomorphic submersion. Assume that $\pm K_{\mathcal{X}/S}$
is relatively ample and that $\phi_{t}$ evolves according to the
Kähler-Ricci flow in the non-normalized setting. Then \begin{equation}
\frac{\partial c(\phi_{t})}{\partial t}=\Delta_{\omega_{t}^{X}}c(\phi_{t})-\pm c(\phi_{t})+|A_{\phi_{t}}|_{\omega_{t}^{X}}^{2}.\label{eq:parbol eq in thm kx}\end{equation}
In particular, if $\phi_{KE}$ is a fiber-wise normalized Kähler-Einstein
weight, then \[
\Delta_{\omega_{KE}}c(\phi_{KE})-\pm c(\phi_{KE})+|A_{\omega_{KE}}|_{\omega_{KE}^{X}}^{2}=0.\]
\end{thm}
\begin{proof}
To simplify the notation we will only consider the $+K_{X}-$setting,
but the proof in the $-K_{X}$ setting is essentially the same. We
will just indicate the simple modifications of the proof of Theorem
\ref{thm:heat eq for c in c-y} which arise in the present setting. 

Let us first consider the modifications to the calculation of the
$t-$derivative of $c(\phi)$ that arise from the additional term
$-\phi$ appearing in the calculation of the time derivative $\phi_{t},$
since now \[
\frac{\partial}{\partial t}\phi_{t}=\log\det(\phi_{k\bar{l}})-\phi\]
 in local coordinates. To this end we assume to simplify the notation
that $X$ is one-dimensional (but the general argument is essentially
the same).First recall that according to formula \ref{eq:pf of heat eq cy}:

\[
\frac{\partial}{\partial t}c(\phi)=\frac{\partial}{\partial t}\phi_{s\bar{s}}-[(\phi_{s\bar{z}}\overline{\phi_{s\bar{z}}})_{t}\phi_{z\bar{z}}^{-1}-(\phi_{s\bar{z}}\overline{\phi_{s\bar{z}}})\phi_{z\bar{z}}^{-2}\frac{\partial}{\partial t}\phi_{z\bar{z}}].\]
 Hence, the additional contribution referred to above is of the form
\[
B:=(-\phi)_{s\bar{s}}-2\Re(-\phi_{s\bar{z}}\overline{\phi_{s\bar{z}}})+\phi_{s\bar{z}}\overline{\phi_{s\bar{z}}}\phi_{z\bar{z}}^{-2}(-\phi_{z\bar{z}})=\]
\[
=(-\phi_{s\bar{s}})+|\phi_{s\bar{z}}|^{2}\phi_{z\bar{z}}^{-1}-\phi_{s\bar{z}}\overline{\phi_{s\bar{z}}}\phi_{z\bar{z}}^{-1}=-c(\phi).\]
 Hence, the local calculations in the Calabi-Yau case give that \[
\frac{\partial}{\partial t}c(\phi)=\Delta_{\omega_{t}}c(\phi)-c(\phi)+|A_{\phi}|_{\omega_{t}^{X}}^{2}.\]
 Finally, since a normalized Kähler-Einstein weight is stationary
for the non-normalized Kähler-Ricci flow this finishes the proof of
the theorem. 
\end{proof}
The last fiber-wise elliptic equation in the previous corollary (in
the $K_{X}-$setting) was first obtained very recently by Schumacher
\cite{sc}, who used the maximum principle to deduce the following
interesting 
\begin{cor}
\label{cor:pos of ke }Let $\pi:\,\mathcal{X}\rightarrow S$ be a
fibration as in the previous theorem and assume that $K_{\mathcal{X}/S}$
is relatively ample. Then the canonical fiber-wise Kähler-Einstein
weight $\phi_{KE}$ on $K_{\mathcal{X}/S}$ is smooth with semi-positive
curvature form on $\mathcal{X}.$ Moreover, if the Kodaira-Spencer
classes of the fibration are non-trivial for all $s,$ then the curvature
form of $\phi_{KE}$ is\emph{ strictly} positive on $\mathcal{X}.$
\end{cor}
The first part of the corollary was also shown by Tsuji \cite{ts,ts2}
using his iteration. Similarly, by a simple application of the parabolic
maximum principle we deduce the following corollary from the parabolic
equation in the previous theorem.
\begin{cor}
\label{cor:pos along flow in kx-s}(same assumptions as in the previous
theorem). Let $\phi{}_{t}$ evolve according to the Kähler-Ricci flow
in the non-normalized $\pm K_{X}-$setting. If the initial weight
has (semi-) positive curvature form on $\mathcal{X}$ then so has
$\phi{}_{t}$ for all $t.$ More precisely, $(dd^{c}\phi{}_{t})_{x_{z}}>0$
(in all $n+1$ directions) at any point $x_{s}$ in the fiber $\mathcal{X}_{s}$
unless $(dd^{c}\phi_{0})^{n+1}$ and $A_{\phi_{0}}$ vanish identically
on $\mathcal{X}_{s}.$ \end{cor}
\begin{proof}
As usual we may assume that $S$ is embedded in $\C.$ Let us start
with the \emph{semi}-positive case where the conclusion follows from
the \emph{weak} maximum principle. Indeed, assume to get a contradiction
that $c(\phi_{t})\geq0$ on $\mathcal{X}$ for $t=0$ but that there
is $(t,s,x)$ such that at $(t,s,x)$ we have $c(\phi_{t})(s,x)<0.$
By optimizing over $(x,t)$ we may also assume that $\frac{\partial(e^{at}c(\phi_{t}))}{\partial t}\leq0,$
$\Delta_{\omega_{t}^{X}}c(\phi_{t})\geq0.$ Then equation \ref{eq:parbol eq in thm kx}
gives \[
0\geq e^{at}(ac(\phi_{t})+\frac{\partial c(\phi_{t})}{\partial t})=e^{at}(\Delta_{\omega_{t}^{X}}c(\phi_{t})-(a\pm1)c(\phi_{t})+|A_{\phi_{t}}|_{\omega_{t}^{X}}^{2})\]
But if $a$ is chosen so that $a\pm1>0,$ then the rhs above is strictly
positive, giving a contradiction. To handle the remaining cases we
invoke the following well-known \emph{strong} maximum principle for
the heat operator (which by standard argument can be reduced to the
corresponding local statement in \cite{p-w}): let $h_{t}\geq0$ satisfy
\[
\frac{\partial h_{t}}{\partial t}\geq\Delta_{g_{t}}h_{t}\,\,\,\mbox{on\,}[0,T]\times X\]
for any smooth family $g_{t}$ of Riemannian metrics. Then either
$h_{t}>0$ for all $t>0$ or $h_{0}\equiv0.$ In our case we set $h_{t}=e^{at}c(\phi_{t})$
with $a=-\pm1$ and conclude that if it is not the case that $c(\phi_{t})>0$
for all $t>0$ then $c(\phi_{0})\equiv0$ and hence \[
\frac{\partial}{\partial t}c(\phi_{t})_{t=0}=|A_{\phi_{0}}|_{\omega_{0}^{X}}^{2}\]
If we now assume, to get a contradiction, that the rhs above is strictly
positive at $x_{0}$ then it follows that there is an $\epsilon>0$
such that $c(\phi_{t})(x_{0})>0$ for $t\in]0,\epsilon[$ i.e. for
such $t$ it is not the case that $c(\phi_{t})\equiv0$ on $X.$ Hence,
as explained above $c(\phi_{t})>0$ on all of $]0,\infty[\times X,$
which yields the desired contradiction. 
\end{proof}
In particular, the previous corollary says that if the fibration $\mathcal{X}$
is infinitesimally non-trivial then the non-normalized Kähler-Ricci
flows instantly makes any semi-positively curved initial weight \emph{strictly}
positive. 

Next we note that integrating the last formula in the previous theorem
immediately gives the following corollary first shown by Fujiki-Schumacher
\cite{fs} (Theorem 7.9).
\begin{cor}
\label{corpush forward formel f=0000F6r w-p i kx}(same assumptions
as in the previous theorem). Let $\phi_{KE}$ be the weight of a smooth
metric on $\pm K_{\mathcal{X}/S}$ which restricts to a normalized
Kähler-Einstein weight on each fiber, then \textup{\[
\pi_{*}((dd^{c}\phi_{KE})^{n+1}/(n+1)!)=\pm\omega_{WP}\]
 on $S,$ where $\pi_{*}$ denotes the fiber integral. In particular,
if $S$ is effectively parametrized (i.e. all Kodaira-Spencer classes
are non-trivial) then $\pm\pi_{*}((dd^{c}\phi_{-})^{n+1}$ and hence
the Weil-Petersson metric $\omega_{WP}$ is a Kähler form on the base
$S.$}\end{cor}
\begin{rem}
\label{rem:never conv}It follows immediately from the previous corollary
that, in the $-K_{X}-$setting, the normalized Kähler-Einstein weight
$\phi_{KE}$\emph{ never }has semi-positive curvature on all of $\mathcal{X}$
if the family is effectively parametrized. Combining this fact with
Corollary \ref{cor:pos along flow in kx-s} shows that the relative
Kähler-Ricci flow in the non-normalized $K_{X}$- setting never converges
in the $L^{1}(\mathcal{X})-$topology for an initial weight $\phi_{0}$
with semi-positive curvature form on an effectively parametrized fibration
$\mathcal{X}.$
\end{rem}

\subsection{\label{sub:Quantization:-the-Bergman KX}Quantization: the Bergman
iteration}

The \emph{(normalized) Bergman iteration} in the $\pm K_{X}-$setting
on \emph{$\mathcal{H}_{\pm K_{X}}.$} is defined precisely as in section
\ref{sec:Quantization:-The-Bergman L}, but using the probability
measure $\mu_{\pm}(\phi)$ in the definition of $Hilb^{(k)}(\phi,\mu_{\pm}(\phi)).$
Similarly, the non-normalized\emph{ Bergman iteration} is defined
in terms of the measure $\mu_{\pm}'.$ The virtue of the non-normalized
setting is that the corresponding Hilbert norms correspond to the
{}``adjoint'' norms appearing in Berndtsson's Theorem \ref{thm:(Berndtsson).-Let-}:

\begin{equation}
Hilb^{(k)}(\phi,\mu_{\mu'_{\pm}})(f,f):=i^{n^{2}}\int_{X}f\wedge\bar{f}e^{-(k\pm1)\phi}:=Hilb_{(k-1)L+K_{X}}(\phi)\label{eq:non-normal hilb as adjoint}\end{equation}
for $L=\pm K_{X}.$ Moreover, they are clearly \emph{decreasing} in
$\phi$ (for $k\geq1)$ and hence the analogue of Proposition \ref{pro:bergman iter decr sup cy}
of the corresponding Bergman iteration holds:
\begin{prop}
\label{pro:bergman it decr sup in kx}Consider the Bergman interation
$\phi_{m}^{(k)}$ in the non-normalized $\pm K_{X}-$setting and assume
that $\phi_{m}^{(k)}\leq\psi{}_{m}^{(k)}.$ Then $\phi_{m+1}^{(k)}\leq\psi{}_{m+1}^{(k)}.$
Moreover, if $d(\phi,\psi)$ denotes the sup norm of $\phi-\psi$
then \[
d(\psi_{m+1},\phi_{m+1})\leq d(\psi_{m+1},\phi_{m+1})(1\pm\frac{1}{k})\]
In particular, the Bergman interation decreases the distance $d(\phi,\psi)$
in the non-normalized $K_{X}-$setting.\end{prop}
\begin{proof}
Given the discussion preceeding the proposition we just have to prove
the claimed property of the distance $d.$ But this follows directly
from the monotonicity in the first part combined with the fact that
$\log\rho^{(k)}(\phi_{m}+c)/k=\log\rho^{(k)}(\phi_{m})/k-\pm\frac{c}{k},$
which in turn follows from $\mu'_{\pm}(\phi+c):=e^{\pm(\phi+c)}=\mu'_{\pm}(\phi)e^{\pm c}$ 
\end{proof}
On the other hand, the following monotonicity of functionals holds
in the \emph{normalized} setting:
\begin{lem}
\label{lem:mon of f}The functionals \emph{$-I_{\mu_{\pm}}$ and }$\mathcal{L}^{(k)}$
are increasing along the \emph{normalized Bergman iteration on $\mathcal{H}_{\pm K_{X}}.$
Moreover, they are }strictly \emph{increasing at $\phi_{m}^{(k)}$
unless $\phi_{m}^{(k)}$ is stationary (when $k>1$ in the case of
$I_{\mu_{+}})$}\end{lem}
\begin{proof}
By the general Lemma \ref{lem:monotone along bergman general} $\mathcal{L}^{(k)}$
is increasing and $I_{-}$ is decreasing under the iteration, since
$\phi\mapsto I_{-_{-}}(\phi)$ is concave with respect to the affine
structure by Jensen's inequality. To show that $I_{+}^{(k)}$ is increasing
in the $K_{X}-$setting just observe that, \[
I_{+}(\phi_{m+1}^{(k)})-I_{+}(\phi_{m}^{(k)}):=\log\frac{\int e^{(\phi_{m+1}^{(k)}-\phi_{m}^{(k)})}e^{\phi_{m}^{(k)}}}{\int e^{\phi_{m}^{(k)}}}=\log\int(\rho^{(k)})^{1/k}\mu(\phi_{m}^{(k)})\leq\]
\[
\leq\log((\int\rho^{(k)}(\phi_{m}^{(k)})\mu)^{1/k})=0,\]
 using Jensen's inequality applied to the concave function $t\mapsto t^{1/k},$
which is strictly concave for $k>1.$ 
\end{proof}

\subsubsection{Convergence of the Bergman iteration at a fixed level $k$}
\begin{thm}
\label{thm:conv of bergman it for kx} The Bergman iteration $\phi_{m}^{(k)}$
at level $k$ converges, when the discrete time $m\rightarrow\infty,$
to a balanced weight $\phi_{\infty}^{(k)}$ in the following settings:
\begin{itemize}
\item The normalized $K_{X}-$setting
\item The normalized $-K_{X}-$setting if it is a priori assumed that there
exists some balanced metric at level $k$ and $H^{0}(TX)=0.$ 
\item The normalized $-K_{X}-$setting for $k$ sufficiently large under
the a assumption that $X$ admits a Kähler-Einstein metric and $H^{0}(TX)=0.$ 
\item The non-normalized $+K_{X}-$setting, where the limiting balanced
weight is the unique normalized one. 
\end{itemize}
\end{thm}
\begin{proof}
\emph{Proof of the first point: }By the previous lemma $-I_{\mu}$
is increasing and as shown in \cite{bbgz} $-\mathcal{F}_{\mu}^{(k)}$
is coercive (as follows immediately from Jensen's inequality). Moreover,
as shown in \cite{bbgz} balanced weights are unique modulo scaling
and hence all the convergence criteria in Proposition \ref{pro:conv of bergman iter at level k general}
are hence satisfied. 

\emph{Proof of the second point:} By the previous lemma $-I_{\mu}$
is increasing and as shown in \cite{bbgz} it follows immediately
from Berndtsson's theorem \ref{thm:(Berndtsson).-Let-} applied to
$L=-K_{X}$ that $-\mathcal{F}_{\mu}^{(k)}$ is strictly convex modulo
scaling. Hence, the convergence follows by combining Proposition \ref{pro:conv of bergman iter at level k general}
and Lemma \ref{lem:geod convex}.

\emph{Proof of the third point: }The fact that $-\mathcal{F}_{\mu}^{(k)}$
is coercive was shown in \cite{bbgz} (using the corresponding coercivity
of $-\mathcal{F}_{\mu}$ on $\mathcal{H}_{L}$ ). Given this coercivity
the convergence follows as in the previous point.

\emph{Proof of the fourth point:} let $(\phi')_{m}^{(k)}=\phi{}_{m}^{(k)}+C_{m}^{(k)}$
denote the non-normalized Bergman iteration in the $K_{X}-$setting.
By the definition of the Bergman iteration (compare equation \ref{eq:difference eq}
below): \[
(C_{m+1}^{(k)}-C_{m}^{(k)})=-C_{m}^{(k)}/k-I(\phi{}_{m}^{(k)})/k\]
 where by the first point above $I(\phi{}_{m}^{(k)})\rightarrow I_{\infty},$
when $m\rightarrow\infty.$ Set $D_{m}:=C_{m}^{(k)}+I(\phi{}_{m}^{(k)}).$
Then \[
D_{m+1}=(1-1/k)D_{m}+\mbox{\ensuremath{\epsilon}}_{m},\]
 where $\epsilon_{m}=(I(\phi{}_{m+1}^{(k)})-I(\phi{}_{m}^{(k)}))\rightarrow0$
as $m\rightarrow\infty.$ But then it follows for elementary reasons
that $D_{m}\rightarrow0,$ i.e. $C_{m}^{(k)}\rightarrow-I_{\infty}$
showing that $(\phi')_{m}^{(k)}$ indeed converges and $I_{+}((\phi')_{m}^{(k)})\rightarrow0,$
proving the second point. For completeness we finally show that $D_{m}\rightarrow0.$
Assume for a contradiction that this is not the case. But then $D_{m+1}/D_{m}\rightarrow(1-1/k)$
and hence $D_{m}\leq C_{\delta}(1-1/k)+\delta)^{m}\rightarrow0$ for
$\delta$ sufficiently small, giving a contradiction.
\end{proof}
The convergence in the fourth point above also follows immediately
from the contracting property of the corresponding iteration (compare
the proof of Theorem \ref{thm:uniform conv of balanced} below). We
also note the following direct consequence of Berndtsson's Theorem
\ref{thm:(Berndtsson).-Let-}, using formula \ref{eq:non-normal hilb as adjoint}
in the non-normalized setting. 
\begin{cor}
The Bergman iteration in the non-normalized $\pm K_{X}-$setting preserves
the (semi-)positivity of the curvature of the initial weight. Moreover,
if the fibration $\mathcal{X}$ is assumed infinitesimally non-trivial
then any initial weight on $\pm K_{\mathcal{X}/S}$ on which is semi-positively
curved and strictly positively curved along the fibers of $\mathcal{X}$
becomes strictly positively curved under the iteration.
\end{cor}
Combining the previous corollary and Theorem \ref{thm:conv of bergman it for kx}
now gives the following:
\begin{cor}
\label{cor:pos of semi-bal when kx amp}Let $\pi:\mathcal{\, X}\rightarrow S$
be a proper holomorphic submersion with $K_{\mathcal{X}/S}$ relatively
ample. Let $\phi^{(k)}$ be the weight on $K_{\mathcal{X}/S}$ obtained
by requiring that its restriction to any fiber is the unique normalized
balanced weight at level $k,$ i.e $\int_{\mathcal{X}_{s}}e^{\phi^{(k)}}=1.$
Then $\phi^{(k)}$ is smooth with semi-positive curvature form. Moreover,
if the fibration $\mathcal{X}$ is assumed infinitesimally non-trivial
then $\phi^{(k)}$ is strictly positively curved. \end{cor}
\begin{proof}
Since positivity  and smoothness are local notions it is enough to
prove the corollary when when $S$ is embedded in $\C.$

\emph{Smoothness:} By definition $\phi^{(k)}=FS^{(k)}(H^{(k})$ where
$H^{(k)}$ is an element in the finite-dimensional smooth manifold
$\mathcal{H}^{(k)}$ uniquely determined by $G^{(k)}(H^{(k)},s)=0,$
\cite{bbgz} where $G^{(k)}$ is the smooth map defined by \[
G^{(k)}(H^{(k)},s):=(T^{(k)}-I,I_{+}\circ FS^{(k)})\in\mathcal{H}^{(k)}\times\R.\]
Moreover, as shown in \cite{bbgz} the linearization of $T^{(k)}-I$
is invertible modulo scaling (since it represents the differential
of a functional on $H^{(k)}$ which is strictly convex modulo scaling).
Hence, the claimed smoothness follows from the implicit function theorem.

\emph{Positivity:} Since $K_{\mathcal{X}/S}$ is assumed relatively
ample it admits a smooth weight $\phi_{0},$ which has fiber-wise
positive curvature form. After adding a sufficiently large multiple
of the pull-back from the base of $|s|^{2}$ we may assume that $\phi_{0}$
has positive curvature over $\mathcal{X}.$ By the last point of the
previous theorem the Bergman iteration $\phi_{m}^{(k)}$ in the non-normalized
$K_{X}-$setting with initial weight $\phi_{0}$ yields a sequence
of weights on $K_{\mathcal{X}/S}$ converging, when $m\rightarrow\infty,$
uniformly to the unique normalized balanced weight $\phi^{(k)}$ at
level $k.$ As a consequence $dd^{c}\phi^{(k)}\geq0$ on $\mathcal{X}.$
Moreover, if the fibration $\mathcal{X}$ is assumed infinitesimally
non-trivial the previous corollary shows that applying the Bergman
iteration to $\phi^{(k)}$ yields a strictly positively curved metric.
But since $\phi^{(k)}$ is fixed under the iteration this finishes
the proof of the corollary.\end{proof}
\begin{cor}
Let $\pi:\mathcal{\, X}\rightarrow S$ be the universal curve of the
Teichmuller space of complex curves of a genus $g\geq2.$ Fix a positive
integer $k$ (for $g=2$ we assume that $k\geq2).$ Under the natural
isomorphism \[
(T^{1,0}S)^{*}=\pi_{*}(2K_{\mathcal{X}/S})\]
the fiber-wise normalized balanced weight $\phi^{(k)}$ on $K_{\mathcal{X}/S}$
at level $k$ (appearing in the previous corollary) induces an Hermitian
metric $\omega^{(k)}$ on $S$ with a curvature which is dually Nakano
positive.. Moreover, when $k\rightarrow\infty$ the metric $\omega^{(k)}$
converges towards the Weil-Petersson metric $\omega_{WP}$ point-wise
on $S.$\end{cor}
\begin{proof}
As is classical the assumptions on $k$ ensure that $K_{\mathcal{X}/S}$
is very ample. By the previous corollary $\phi^{(k)}$ is a smooth
weight on $K_{\mathcal{X}/S}\rightarrow\mathcal{X}$ with strictly
positive curvature and hence the $L^{2}-$metric on the direct image
bundle $\pi_{*}(\mathcal{L}+K_{\mathcal{X}/S})$ (with $\mathcal{L}=K_{\mathcal{X}/S})$
induced by $\phi^{(k)}$ has, according to the first point in Theorem
\ref{thm:(Berndtsson).-Let-}, a curvature which is positive in the
sense of Nakano. Since, $T^{1,0}S_{|s}=H^{1}(T^{1,0}\mathcal{X}_{s})\cong H^{0}(2K_{\mathcal{X}_{s}})^{*}$
this proves the first statement. To prove the point-wise convergence
on $S$ of $\omega^{(k)}$ towards $\omega_{WP}$ it is enough to
prove that \[
e^{-\phi^{(k)}}\rightarrow e^{-\phi_{KE}}\]
in $L_{loc}^{1}(X)$ for $X=\mathcal{X}_{s}$ (since, by definition,
it implies the point-wise convergence of the corresponding Hermitian
metrics on $\pi_{*}(\mathcal{L}+K_{\mathcal{X}/S})).$ But this convergence
follows from the $L^{1}$ convergence of $\phi^{(k)}$ towards $\phi_{KE}$
(Theorem \ref{thm:conv of bergman it for kx}) combined with the fact
that $J(\phi^{(k)})$ is uniformly bounded, as shown in \cite{bbgz}
(see Lemma 6.4 therein). Alternatively, it follows immediately from
the uniform convergence in Theorem \ref{thm:uniform conv of balanced}
below.
\end{proof}
The convergence in the previous corollary should be compared with
the approximation results for the Weil-Petterson metric for moduli
spaces of higher dimensional manifolds recently obtained by Keller-Lukic
\cite{k-l}. The approximating \emph{Kähler} metrics $\omega'_{k}$
in \cite{k-l} are related to different balanced metrics, namely those
defined wrt Donaldson'a original setting in \cite{do1} (where $\mu(\phi)=MA(\phi)).$

\subsubsection{Convergence towards the Kähler-Ricci flow}
\begin{thm}
\label{thm:conv of bergman to ricci}The following convergence results
hold in all settings introduced in the beginning of section \ref{sec:The-(anti-)-canonical}
(i.e. in the (non-) normalized $\pm K_{X}-$settings). Fix a smooth
and strictly psh weight initial weight $\phi_{0}$ on $\pm K_{X}$
and consider the corresponding Bergman iteration $\phi_{m}^{(k)}$
at level $k$ and discrete time $m,$ as well as the corresponding
Kähler Ricci flow $\phi_{t}.$ Then there is a constant $A$ such
that \begin{equation}
\sup_{X}|\phi{}_{m}^{(k)}-\phi{}_{m/k}|\leq Am/k^{2}\label{eq:statement thm conv bergman kx}\end{equation}
 uniformly in $(m,k)$ satisfying $m/k\leq T$ (in the $K_{X}-$setting
$A$ is independent of $T).$ In particular, if $m_{k}$ is a sequence
such that $m_{k}/k\rightarrow t,$ then $\phi_{m_{k}}^{(k)}\rightarrow\phi(t)$
uniformly on $X$ and \[
dd^{c}\phi_{m_{k}}^{(k)}\rightarrow\omega_{t}\]
 on $X$ in the sense of currents, where $\omega_{t}$ evolves according
to the corresponding Kähler-Ricci flow \ref{eq:k-r flow for k metrics in kx}.
The corresponding result also holds for the corresponding non-normalized
flows and in the relative setting, where the convergence is locally
uniform with respect to the base parameter $s.$ \end{thm}
\begin{proof}
In the case of the non-normalized $K_{X}-$setting (denoted by primed
objects) the proof of Theorem \ref{thm:conv of bergman iter to ricci cy}
carries over essentially verbatim, thanks to the last statement in
Proposition \ref{pro:bergman it decr sup in kx} and Corollary \ref{cor:conv of non-normal flow}
which gives the uniformity wrt $T\in[0,\infty].$To handle the non-normalized
$-K_{X}-$setting we need to modify the previous argument slightly.
More precisely, we will prove that \begin{equation}
\sup_{X}|\phi{}_{m}^{(k)}-\phi{}_{m/k}|\leq A(1+\frac{1}{k})^{m}m/k^{2}\label{eq:statement thm conv bergman kx}\end{equation}
Accepting this for the moment the claimed convergence when $m_{k}/k\rightarrow t$
follows using that \[
(1+\frac{1}{k})^{m}=((1+\frac{1}{k})^{k})^{m/k}\leq e^{m/k}\leq e^{T},\]
 when $m/k\leq T.$ To prove \ref{eq:statement thm conv bergman kx}
first observe that Step 1 in the proof of Theorem \ref{thm:conv of bergman iter to ricci cy}
still applies for $(m,k)$ such that $m/k\leq T$ (using Proposition
\ref{pro:bouche tian-1} applied to the non-normalized $-K_{X}-$
setting). In other words, there is a constant $A$ (depending on $T)$
such that \[
\sup_{X}|\psi_{k,m+1}-(\psi_{k,m}+F^{(k)}(\psi_{k,m})|\leq A(1/k^{2})\]
 for all $(m,k)$ such that $m/k\leq T.$ Now we fix the integer $k$
and assume as an induction hypothesis that \ref{eq:statement thm conv bergman kx}
holds for $m$ with $A$ the constant in the previous inequality.
By Proposition \ref{pro:bergman iter decr sup cy}, \[
\sup_{X}|(\psi_{k,m}+F^{(k)}(\psi_{k,m}))-(\phi_{m}^{(k)}+F^{(k)}(\phi_{m}^{(k)})|\leq\sup_{X}|(\psi_{k,m}-\phi_{m}^{(k)}|(1+\frac{1}{k})\leq\]
\[
\leq(A(1+\frac{1}{k})^{m}m/k^{2})(1+\frac{1}{k})\]
with the same constant $A$ as above, using the induction hypothesis
in the last step. Combining this estimate with the previous inequality
gives \[
\sup_{X}|(\psi_{k,m+1}-\phi_{m+1}^{(k)}|\leq A(1+\frac{1}{k})^{m+1}m/k^{2}+A/k^{2}.\]
But using that $1\leq(1+\frac{1}{k})^{m+1}$ in the last term above
proves the induction step and hence finishes the proof of the estimate
\ref{thm:conv of bergman iter to ricci cy}. 

To treat the Kähler-Ricci flows $\phi_{t}$ in the normalized settings
we write \[
\phi_{t}'=\phi_{t}+C_{t},\]
 where $C_{t}$ is a constant for each $t.$ Then \begin{equation}
\frac{\partial C_{t}}{\partial t}=-(I_{\pm}(\phi'_{t})\label{eq:diff eq in pf bergman it kx}\end{equation}
 Indeed, by the definition of the flow $\phi'_{t}$ and $\phi_{t}$
we have \[
\begin{array}{rcl}
\frac{\partial\phi'_{t}}{\partial t} & = & \log(MA(\phi'_{t})-\pm\phi'_{t}\\
\frac{\partial\phi{}_{t}}{\partial t} & = & \log(MA(\phi{}_{t})-\pm\phi_{t}+\pm(I_{\pm}(\phi{}_{t})\end{array}\]
By scale invariance we may as well replace $\phi_{t}$ with $\phi_{t}'$
in the rhs of the second equation above and hence subtracting the
second equation from the first one proves \ref{eq:diff eq in pf bergman it kx}. 

Similarly, writing \[
(\phi')_{m}^{(k)}=\phi{}_{m}^{(k)}+C_{m}^{(k)}\]
we obtain the the following difference equation, using that $\phi\mapsto\rho^{(k)}(\phi),$
defined with respect to $\mu_{\pm}$ is scale invariant: \begin{equation}
C_{m+1}^{(k)}-C_{m}^{(k)}=-\frac{1}{k}I_{\pm}((\phi')_{m}^{(k)})\label{eq:difference eq}\end{equation}
Now, as explained above the estimate \ref{eq:statement thm conv bergman kx}
, holds for the {}``primed'' objects and hence by the scaling equivariance
of $I_{\pm}$ \begin{equation}
|I_{\pm}(\phi'_{m/k})-I_{\pm}((\phi')_{m}^{(k)})|\leq Am/k^{2}\label{eq:conv of scaled int}\end{equation}
 A simple version of the argument given in the proof of Theorem \ref{thm:conv of bergman to ricci}
now shows, by comparing the differential equation \ref{eq:diff eq in pf bergman it kx}
with the difference equation \ref{eq:difference eq} and using \ref{eq:conv of scaled int}
, that \[
|C_{m}^{(k)}-C_{m/k}|\leq Bm/k^{2}\]
 for a uniform constant $B.$ All in all this hence finishes the proof
of the theorem.
\end{proof}
We also have the following analogue of Cor \ref{cor:conv of balanced etc cy}:
\begin{cor}
For a fixed initial data $\phi_{0}=\phi_{0}^{(k)}\in\mathcal{H}_{\pm K_{X}}$
the following convergence results hold for the Bergman iteration $\phi_{m}^{(k)}$
in the normalized $\pm K_{X}-$setting (in the $-K_{X}-$setting it
is assumed that $H^{0}(TX)=0$ and $X$ a priori admits a Kähler-Einstein
metric):
\begin{itemize}
\item For any sequence $m_{k}$ such that $m_{k}/k\rightarrow\infty$ the
convergence $\phi_{m_{k}}^{(k)}\rightarrow\phi_{\infty}$ holds in
the $L^{1}-$ topology on $X.$ 
\item The balanced weights $\phi_{\infty}^{(k)}:=\lim_{m\rightarrow\infty}\phi_{m}^{(k)}$
at level $k$ converge, when $k\rightarrow\infty,$ in the $\mathcal{C}^{\infty}-$topology,
to the weight $\phi_{\infty}$ which is the large time limit of the
corresponding Kähler-Ricci flow. 
\end{itemize}
Moreover, the convergence in the second point also holds in the non-normalized
$K_{X}-$setting, where the limit $\phi_{\infty}$ coincides with
the canonical Kähler-Einstein weight $\phi_{KE}.$ In the relative
case all convergence results hold fiberwise locally uniformly with
respect to the base parameter $s.$\end{cor}
\begin{proof}
The proof of the first two points proceeds exactly as in the previous
setting (again using the variational characterization in \cite{bbgz}).
As for the claimed convergence in the non-normalized setting it is
obtained by noting that the large $m$ limit $(\phi')_{m}^{(k)}$
in the non-normalized setting is the unique balanced weight such that
$I_{\pm}((\phi')_{\infty}^{(k)})=0.$ In other words, $(\phi'){}_{\infty}^{(k)}=\phi_{\infty}^{(k)}-I_{\pm}(\phi_{\infty}^{(k)}),$
where $\phi_{\infty}^{(k)}$ is the large $m$ limit of the iteration
in the normalized setting. But by the second point above this means
that $(\phi'){}_{\infty}^{(k)}\rightarrow\phi_{\infty}-I_{\pm}(\phi_{\infty})$
in $L^{1}$ (also using the continuity with respect to the $L^{1}-$topology
of the functional $I_{\pm}$ on compacts; compare \cite{bbgz}). By
uniqueness, this means that the limit must be $\phi_{KE}.$ 
\end{proof}

\subsection{Uniform convergence of the balanced weights in the $K_{X}-$setting }

Next we point out that in the $K_{X}-$setting the convergence of
the balanced weights is actually\emph{ uniform} (the proof is independent
of the variational one given in \cite{bbgz}). The proof simply uses
that $\phi^{(k)}$ is close to $\phi_{t_{k}}$ where $\phi_{t}$ is
the corresponding Kähler-Ricci flow and $t_{k}$ is a suitable sequence
tending to infinity. 
\begin{thm}
\label{thm:uniform conv of balanced}Let $\phi^{(k)}$ be the the
balanced weight at level $k$ on the canonical line bundle $K_{X}$
(in the non-normalized setting). When $k\rightarrow\infty,$ the weights
$\phi^{(k)}$ converge uniformly towards the normalized Kähler-Einstein
weight $\phi_{KE}.$\end{thm}
\begin{proof}
Fix a smooth and positively curved weight $\phi_{0}$ on $K_{X}$
and denote by $\phi_{m}^{(k)}$ the Bergman iteration at level $k$
with initial data $\phi_{0}^{(k)}=\phi_{0}.$ By Proposition the map
whose iterations define the Bergman iterations is a contraction mapping
with contacting constant $q=(1-\frac{1}{k})<1$ and hence it follows
from the Banach fix point theorem that \[
\left\Vert \phi^{(k)}-\phi_{m}^{(k)}\right\Vert _{L^{\infty}}\leq\frac{q^{m}}{(1-q)}\left\Vert \phi_{1}^{(k)}-\phi_{0}\right\Vert _{L^{\infty}}\]
By definition we have $\phi_{1}^{(k)}-\phi_{0}=\frac{1}{k}\log\rho(k\phi)$
which, according to Prop \ref{pro:bouche tian-1}, is uniformly bounded
by a constant times $\frac{1}{k}\log k$ and hence \[
\left\Vert \phi^{(k)}-\phi_{m}^{(k)}\right\Vert _{L^{\infty}}\leq C(1-\frac{1}{k})^{k})^{m/k}\log k\]
Next we take the sequence $m=m_{k}:=[k^{3/2}]$ where $[c]$ denotes
the smallest integer which is larger than $c.$ Then $t_{k}:=m_{k}/k=k^{1/2}\rightarrow\infty$
as $k\rightarrow\infty$ and since $(1-\frac{1}{k})^{k}\rightarrow e^{-1}<1$
we conclude that \[
\left\Vert \phi_{m}^{(k)}-\phi_{0}\right\Vert _{L^{\infty}}\rightarrow0\]
 as $k\rightarrow\infty.$ If now $\phi_{t}$ denotes the Kähler-Ricci
flow in the non-normalized $K_{X}-$setting we have, according to
Theorem \ref{thm:conv of bergman to ricci}, that \[
\left\Vert \phi_{m}^{(k)}-\phi_{\frac{m_{k}}{k}}\right\Vert _{L^{\infty}}\rightarrow0\]
using that $m_{k}/k^{2}\rightarrow\infty.$ Finally, since $\phi_{t_{k}}\rightarrow\phi_{KE}$
uniformly as $t_{k}\rightarrow\infty$ this proves the theorem. Of
course, the last convergence is not really needed for the proof as
we may as well start with $\phi_{0}=\phi_{KE}$ which is trivially
fixed under the Kähler-Ricci flow.
\end{proof}
It should be pointed out that the uniform convergence in the previous
theorem has been previously obtained by Berndtsson (who also related
it to Tsuji's iteration \cite{ts}), using a different approach -
see the announcement in \cite{bern3}. But hopefully the relation
to the convergence of the Kähler-Ricci flow above may shed some new
light on the convergence.

\subsection{\label{sub:fam gen type}Families of varieties of general type and
comparison with the NS-metric}

The quantized setting concerning the case when $K_{X}$ is ample admits
a straight forward generalization to the case when $K_{X}$ is merely
$\Q-$effective \cite{la}. For simplicity we will only discuss the
case when $K_{X}$ is big, i.e. $X$ is a non-singular variety of
general type. Moreover, we will no longer assume that the map $\pi$
is a submersion. More precisely, we are given a surjective quasi-projective
morphism $\pi:\,\mathcal{X}\rightarrow S$ between non-singular varieties
such that the generic fiber is a variety of general type. We denote
by $S^{0}$ the maximal Zariski open subset of $S$ such that $\pi$
restricted to $\mathcal{X}^{0}:=\pi^{-1}(S^{0})$ is a submersion,
i.e. a smooth morphism (and hence the fibers of $S^{0}$ are non-singular
varieties of general type).

Let us first consider the general absolute case, where we are given
a line bundle $L\rightarrow X$ and an integer $k$ such that $kL$
is effective, i.e. $H^{0}(X,kL)\neq\{0\}.$ The main new feature in
this more general setting is that any Bergman weight $\psi_{k}$ at
level $k,$ i.e. $\psi_{k}\in FS^{(k)}(\mathcal{H}^{(k)}),$ will
usually have singularities, i.e. it defines a singular metric on $L$
with positive curvature form. More precisely, the weight $k\psi_{k}$
on $kL$ is singular precisely along the base locus $Bs(kL)$ of $kL,$
i.e. the intersection of the zero sets of all elements in $H^{0}(kL).$
Anyway, the \emph{difference} of any two Bergman metrics is clearly
bounded. Moreover, when $L=K_{X}$ the measure $\mu_{\psi_{k}}:=e^{\psi_{k}}$
has a smooth density which vanishes precisely along $Bs(kL).$ As
a consequence, we may fix such a reference (singular) weight $\phi_{0}:=\psi_{k}$
and the reference measure $\mu_{0}:=e^{\psi_{k}}.$ Then Lemma \ref{lem:jk is exhaust in bergman}
still applies (as explained in the remark following the lemma). As
a consequence the proof of the convergence of the Bergman iteration
to a balanced weight at level $k$ in the non-normalized $K_{X}-$setting
(Theorem \ref{thm:conv of bergman it for kx}) is still valid as long
as $kK_{X}$ is effective. Combining this latter convergence with
the generalizations \cite{b-p,b-p2} of Berndtsson's Theorem \ref{thm:(Berndtsson).-Let-}
and the invariance of plurigenera \cite{si} then gives the following
generalization of Corollary \ref{cor:pos of semi-bal when kx amp}:
\begin{thm}
\label{thm:pos of semi-bal on v general type}Let $\pi:\,\mathcal{X}\rightarrow S$
be a surjective quasi-projective morphism such that the generic fiber
is a variety of general type. Then, for $k$ sufficiently large there
is a unique singular weight $\phi^{(k)}$ on the relative canonical
line bundle $K_{\mathcal{X}/S}\rightarrow\mathcal{X}$ with positive
curvature current, such that the restriction of $\phi^{(k)}$ to any
fiber over $S^{0}$ is a normalized and balanced weight at level $k.$
Moreover, the weight $\phi^{(k)}$ is smooth on the Zariski open set
defined as the complement in $\mathcal{X}$ of $\bigcup_{s\in S^{0}}Bs(kK_{\mathcal{X}_{s}})\cup\pi^{-1}(S-S^{0}).$ \end{thm}
\begin{proof}
Let us first prove the positivity statement. As before we may assume
that $S$ is a domain in $\C.$ First we consider the behavior over
the set $S^{0},$ i.e. where the fibration is a submersion. Fix $s_{0}\in S_{0}$
and write $X=\mathcal{X}_{s_{0}}.$ Let $(f_{i})$ be a bases in $H^{0}(X,kK_{X}).$
By the invariance of plurigenera \cite{si} $s_{0}$ has a neighborhood
$U\subset S_{0}$ with holomorphic sections $F_{i}$ of $kK_{\mathcal{X}/S}\rightarrow U$
such that $F_{i}$ restricts to $f_{i}$ on $X.$ After perhaps shrinking
$U$ we may hence assume that the restrictions of $F_{i}$ to any
fiber give a base in $H^{0}(\mathcal{X}_{s},kK_{\mathcal{X}_{s}}).$
Let now $\phi_{0}:=\frac{1}{k}\log(\frac{1}{N_{k}}\sum|F_{i}|^{2})$
so that $\phi_{0}$ is a singular weight on $K_{\mathcal{X}/S}$ over
$U$ with positive curvature and such that $\phi_{0}$ restricts to
a Bergman weight at level $k$ on each fiber. In particular, \begin{equation}
\int_{\mathcal{X}_{s}}|f|^{2}e^{-(k-1)\phi_{0}}<\infty\label{eq:integr cond}\end{equation}
 for any $f\in H^{0}(\mathcal{X}_{s},kK_{\mathcal{X}_{s}}).$ Decomposing,
as before, $kK_{X}=(k-1)L+K_{X}$ with $L=K_{X},$ but now using Theorem
3.5 in \cite{b-p} shows that the curvature current of the weight
$\phi_{1}^{(k)}:=FS^{(k)}\circ Hilb^{(k)}(\phi_{0})$ on $K_{\mathcal{X}/S}$
is positive over $U.$ Since, by definition, $\phi_{1}^{(k)}$ is
still fiber-wise a Bergman weight at level $k$ we may iterate the
same argument and conclude that $\phi_{m}^{(k)}$ has a positive curvature
current for any $m.$ Now, as explained in the discussion before the
statement of the theorem, \[
m\rightarrow\infty\implies\sup_{\mathcal{X}_{s}}|\phi_{m}^{(k)}-\phi^{(k)}|\rightarrow0,\]
 locally uniformly with respect to $s,$ where $\phi^{(k)}$ is the
unique normalized fiber-wise balanced weight at level $k.$ In particular,
it follows that $\phi^{(k)}$ has a curvature current which is positive
over $S^{0}.$ 

Finally, to prove the claimed extension property of $\phi^{(k)}$
over $S-S_{0}$ first note that, writing $X=\mathcal{X}_{s}$ for
a fixed fiber, \begin{equation}
\phi^{(k)}\leq\phi_{NS}^{(k)}:=\log(\sup_{f\in H^{0}(X,kK_{X})}(|f|^{2/k}/\int_{X}(f\wedge\bar{f})^{1/k}))\label{eq:bound by ns}\end{equation}
 where $k\phi_{NS}^{(k)}$ is the weight of the Narasimhan-Simha (NS-)
metric on $kK_{\mathcal{X}/S}.$\cite{ns,ka,ts2,b-p2}. Accepting
this for the moment we can use the result in \cite{b-p2} saying that
$\phi_{NS}^{(k)}$ is locally bounded from above, with a constant
which does not blow-up as $s$ converges to a point in $S-S^{0}$
(this is proved by an $L^{2/k}$ variant of the local Ohsawa-Takegoshi
$L^{2}-$extension theorem). By the inequality \ref{eq:bound by ns}
it hence follows that $\phi^{(k)}$ is also locally bounded from above
by the same constant and then the claimed extension property follows
from basic pluripotential theory.

Finally, to prove the inequality \ref{eq:bound by ns} fix a point
$x\in X.$ By the extremal definition of Bergman kernels there are
sections $f_{i}$ (depending on $x)$ such that $\phi^{(k)}(x)=\frac{1}{k}\log(\frac{1}{N_{k}}|f_{1}|^{2}(x))$
and $\phi^{(k)}=\frac{1}{k}\log(\frac{1}{N_{k}}\sum_{i}|f_{i}|^{2})$
on $X.$ Since, $\int_{X}e^{\phi^{(k)}}=1$ it hence follows that
$\int_{X}(\frac{1}{N_{k}}f_{1}\wedge\bar{f_{1}})^{1/k})\leq1$ which
finishes the proof of the inequality \ref{eq:bound by ns}, since
$f_{1}/(N_{k})^{1/2}$ is a candidate for the sup defining $\phi_{NS}^{(k)}.$

As for the last smoothness statement in the theorem it is proved exactly
is in Corollary \ref{cor:pos of semi-bal when kx amp}, using that
$\pi_{*}(kK_{\mathcal{X}/S})$ is a locally trivial vector bundle
over $S^{0}.$ Indeed, it follows as before that the fiber-wise normalized
balanced metrics $H_{s}^{(k)},$ which by the local freeness may be
identified with a family in $GL(N_{k}),$ form a \emph{smooth} family.
Applying the Fubini-Study map to get $\phi^{(k)}$ then introduces
the singular locus described in the statement of the theorem. \end{proof}
\begin{rem}
\label{rem:inv of plurig}If one does not invoke the invariance of
plurigenera in the proof of the previous theorem then the same argument
gives the slightly weaker statement where $S^{0}$ is replaced by
the intersection of $S^{0}$ with a Zariski open set where $\pi_{*}(kK_{\mathcal{X}/S})$
is a locally trivial vector bundle. If one could then prove that the
extension of $\phi^{(k)}$ is such that the integrability condition
\ref{eq:integr cond} holds over all of $S$, then the invariance
of plurigenera would follow from a well-known version of the Ohsawa-Takegoshi
extension theorem. It would be interesting to see if this approach
is fruitful in the non-projective Kähler case where the invariance
of plurigenera is still open. When $\phi^{(k)}$ is replaced by the
weight of the $NS-$metric $\phi_{NS}^{(k)}$ (see formula \ref{eq:conv of curv of appr})
this approach was recently used by Tsuji \cite{ts2} to give a new
proof of the invariance of plurigenera (in the projective case). 
\end{rem}
It should also be pointed out that (singular) Kähler-Einstein metrics
and Kähler-Ricci flows have been studied recently for $K_{X}$ big.
For example, using the deep finite generation of the canonical ring
there is a unique Kähler-Einstein weight with minimal singularities
which satisfies the Monge-Ampère equation \[
(dd^{c}\phi_{KE})^{n}/n!=e^{\phi_{KE}}\]
 on a Zariski open set in $X$ \cite{egz,begz}. It seems likely that
the positivity result in Corollary \ref{cor:pos of ke } can be extended
to families of such singular weights $\phi_{KE}.$ But there are several
regularity issues which need to be dealt with. Moreover, it also seems
likely that the canonical balanced weights $\phi^{(k)}$ converge
to $\phi_{KE},$ when $K_{X}$ is big, but this would require a generalization
of the convergence results in \cite{bbgz} (which only concern ample
line bundles). This latter conjectural convergence should be compared
with the convergence of the weight of the $NS-$metrics proved in
\cite{b-d}, saying that $\phi_{NS}^{(k)}$ converges in $L^{1}$
(and uniformly on compacts of an Zariski open set) to \[
\phi_{can}:=\sup\{\psi:\,\int_{X}e^{\psi}=1\},\]
 where the sup is taken over all singular weights $\psi$ on $K_{X}$
with positive curvature current. In particular, $\phi_{KE}\leq\phi_{can}$,
which is consistent with the inequality \ref{eq:bound by ns}.

\subsection{\label{sub:Comparison-with-the}Comparison with the constant scalar
curvature and other settings}

Given an ample line bundle $L\rightarrow X$ the absolute setting
when $\mu(\phi):=(dd^{c}\phi)^{n}/n!$ was studied in depth by Donaldson
in \cite{do1,do2}. Of course, in this setting the Kähler-Ricci flow
is trivial, but the corresponding quantized setting and the study
of its large $k$ limit is highly non-trivial. In fact, it it was
shown by Donaldson in \cite{do1} that, if it is a priori assumed
that $c_{1}(L)$ contains a Kähler metric $\omega$ with constant
scalar curvature and if $H^{0}(TX)=\{0\},$ then the curvature forms
of any sequence of balanced weights converge in the $\mathcal{C}^{\infty}-$topology
to $\omega.$ Moreover, Donaldson showed that such balanced weights
do exist for $k$ sufficiently large. As earlier shown by Zhang this
latter fact is equivalent to the polarized variety $(X,kL)$ being
stable in the sense of Chow-Mumford (with respect to a certain action
of the group $SL(N_{k})).$ An explicit proof of the convergence of
the Bergman iteration in this setting was given by Sano \cite{sa}
(see also \cite{do2}). 

Note that in this setting the functional $I_{\mu}$ is precisely the
functional $\mathcal{E}$ (compare the beginning of section \ref{sec:The-general-setting}).
Since, $\mathcal{E}$ is well-known to be concave on $\mathcal{H}_{L}$
with respect to the affine structure and $\mathcal{E}\circ FS$ is
geodesically convex on $\mathcal{H}^{(k)}$ the convergence of the
corresponding Bergman iteration is also a consequence of Proposition
\ref{pro:conv of bergman iter at level k general}. 

It should also be pointed out that the role of the Kähler-Ricci flow
of Kähler metrics in this setting is played by the Calabi flow. Indeed,
as shown by Fine \cite{fi}, the balancing flow, which is a continuous
version of Donaldson's iteration, converges, at the level of Kähler
metrics, in the large $k$ limit to the Calabi flow. More precisely,
the balancing flow $H_{t}^{(k)}$ is simply the scaled gradient flow
on the symmetric space $\mathcal{H}^{(k)}$ of the functional $\mathcal{F_{\mu}}^{(k)}$
in this setting and the convergence holds for the curvature forms
of the weights $FS^{(k)}(H_{t})$ in $\mathcal{H}_{L}.$ 
\begin{rem}
Another, less studied, setting of geometric relevance (see \cite{bern1,be2})
appears when we let \[
\mu(\phi):=\frac{1}{N_{l}}\sum_{i=1}^{N_{l}}f_{i}\wedge\bar{f_{i}}e^{-l\phi}\]
 for a fixed integer $l$ where $f_{i}$ is an orthonormal base for
$H^{0}(lL+K_{X})$ equipped with the Hermitian metric induced by $\phi.$
When $L=-K_{X}$ and $l=1$ this is precisely the normalized $-K_{X}-$setting.
In the general case $I_{\mu}(\phi)$ is essentially the induced metric
on the top exterior power of the Hilbert space $H^{0}(lL+K_{X}).$
Moreover, as soon as the corresponding functional $\mathcal{F}_{\mu}^{(k)}$
has a critical point and $H^{0}(TX)=\{0\}$ the assumptions for convergence
in Proposition \ref{pro:conv of bergman iter at level k general}
are satisfied (see \cite{bern1}).\end{rem}


\begin{thebibliography}{56}
\bibitem{be1}Berman, R.J. : Determinantal point processes and fermions
on complex manifolds: Bulk universality. arXiv:0811.3341 

\bibitem{b-b}Berman, R.J.: Boucksom, S: Growth of balls of holomorphic
sections and energy at equilibrium. Invent. Math. 181 (2010), no.
2, 337-394. 

\bibitem{bbgz}Berman, R.J.: Boucksom, S; Guedj, V; Zeriahi, A: A
variational approach to complex Monge-Ampère equations. arXiv:0907.4490

\bibitem{b-d}Berman, R.J; Demailly, J-P: Regularity of plurisubharmonic
upper envelopes in big cohomology classes. Arkiv för Matematik (to
appear) arXiv:0905.1246. 

\bibitem{bern0}Berndtsson, B: Curvature of vector bundles associated
to holomorphic fibrations. Ann. of Math. (2) 169 (2009), no. 2, 531--560. 

\bibitem{bern1}Berndtsson, B: Positivity of direct image bundles
and convexity on the space of Kähler metrics. J. Differential Geom.
81 (2009), no. 3, 457--482.

\bibitem{b-p}Berndtsson, B; Paun, M: Bergman kernels and the pseudoeffectivity
of relative canonical bundles. Duke Math. J. Volume 145, Number 2
(2008), 341-378.

\bibitem{b-p2}Berndtsson, B; Paun, M: A Bergman kernel proof of the
Kawamata subadjunction theorem. arXiv:0804.3884

\bibitem{bern2}Berndtsson, B: Strict and non strict positivity of
direct image bundles. To appear in Math Zeitschrif. arXiv:1002.4797.

\bibitem{bern3}Berndtsson, B; Abstract in {}``Mathematisches Forschungsinstitut
Oberwolfach: Workshop on Multiplier Ideal Sheaves in Algebraic and
Complex Geometry. Report No. 21/2009''

\bibitem{begz}Boucksom, S; Eyssidieux, P; Guedj, V; Zeriahi, A: Monge-Ampère
equations in big cohomology classes. Acta Math. (to appear). arXiv:0812.3674

\bibitem{b-l-y}Bourguignon, J-P.; Li,P; Yau. S.T: Upper bounds for
the first eigenvalue of algebraic submanifolds Comment. Math. Helvetici
69 (1994) 199-207

\bibitem{ca}Cao, H.D: Deformation of Kähler metrics to Kähler-Einstein
metrics on compact Kähler manifolds. Invent. Math. 81 (1985), no.
2, 359--372.

\bibitem{ch}Chen, X.: The space of Kähler metrics, J. Differential
Geom. 56 (2000), no. 2, 189-234.

\bibitem{ct}Chen, X; Tian, G: Ricci flow on Kaehler-Einstein surfaces.
Invent. Math. Vol. 147, Nr. 3 (2002).

\bibitem{de}Deligne, P. Le déterminant de la cohomologie. (French)
{[}The determinant of the cohomology{]} Current trends in arithmetical
algebraic geometry (Arcata, Calif., 1985), 93--177, Contemp. Math.,
67, Amer. Math. Soc., Providence, RI, 1987. 

\bibitem{do1}Donaldson, S. K. Scalar curvature and projective embeddings.
I. J. Differential Geom. 59 (2001), no. 3, 479--522. 

\bibitem{do2}Donaldson, S. K. Scalar curvature and projective embeddings.
II. Q. J. Math. 56 (2005), no. 3, 345--356.

\bibitem{do3}Donaldson, S. K. Some numerical results in complex differential
geometry. Pure Appl. Math. Q. 5 (2009), no. 2, Special Issue: In honor
of Friedrich Herzebruch. Part 1, 571--618.

\bibitem{egz}Eyssidieux, P Guedj, V Zeriahi, A: Singular K\textasciidieresis{}ahler-Einstein
metrics. arXiv:math/0603431. To appear in J. of A.M.S.

\bibitem{g-s-v-y}Greene, B; Shapere,A; Vafa, C; Yau. S.T: Stringy
cosmic strings and noncompact Calabi-Yau manifolds. Nuclear Physics
B, 337(1):1\textendash{} 36, 1990.

\bibitem{fi}Fine, J: Calabi flow and projective embeddings. arXiv:0811.0155 

\bibitem{fs}Fujiki, A; Schumacher, G: The moduli space of extremal
compact Kähler manifolds and generalized Weil-Petersson metrics. Publ.
Res. Inst. Math. Sci. 26 (1990), no. 1, 101--183. 

\bibitem{g-w}Gross, M; Wilson, P. M. H.: Large complex structure
limits of \$K3\$ surfaces. J. Differential Geom. 55 (2000), no. 3,
475--546.

\bibitem{g-z}Guedj,V; Zeriahi, A: Intrinsic capacities on compact
Kähler manifolds. J. Geom. Anal. 15 (2005), no. 4, 607--639.

\bibitem{ha}Hamilton, R.S. Three-manifolds with positive Ricci curvature,
J. Differential Geom. 17 (1982), no. 2, 255\textendash{}306.

\bibitem{ka}Kawamata, Y: Kodaira dimension of Algebraic fiber spaces
over curves, Invent. Math. 66 (1982), pp. 57-71.

\bibitem{ke}Keller, J: Ricci iterations on Kahler classes. Journal
of the Institute of Mathematics of Jussieu (2009), 8 : 743-768 Cambridge
University Press arXiv:0709.1490 

\bibitem{k-l}Keller, J; Lukic, S: Numerical Weil-Petersson metrics
on moduli spaces of Calabi-Yau manifolds. arXiv:0907.1387

\bibitem{ko}Koiso, N., Einstein metrics and complex structure, Invent,
math., 73 (1983), 71-106.

\bibitem{la}Lazarsfeld, R: Positivity in algebraic geometry. I. Classical
setting: line bundles and linear series. II. Positivity for vector
bundles, and multiplier ideals. A Series of Modern Surveys in Mathematics,
48. and 49. Springer-Verlag, Berlin, 2004. 

\bibitem{lsy}Liu, K; Xiaofeng Sun, X; Yau, S-T: Good Geometry on
the Curve Moduli. Publ. RIMS, Kyoto Univ. 42 (2008), 699\textendash{}724. 

\bibitem{ma}Mabuchi, T: K-energy maps integrating Futaki invariants.
Tohoku Math. J. (2) 38 (1986), no. 4, 575\textendash{}593.

\bibitem{ns}Narasimhan, M.S; Simha, R.R.: Manifolds with ample canonical
class. Inventiones Math. 5 (1968) 120\textendash{}128.

\bibitem{p-s}Phong, D. H.; Sturm, J: Scalar curvature, moment maps,
and the Deligne pairing. Amer. J. Math. 126 (2004), no. 3, 693--712.

\bibitem{p-s-s-w}Phong, D.H; Song, J; Sturm, J; Weinkove, B: The
Moser-Trudinger inequality on K\textasciidieresis{}ahler-Einstein
manifolds. Amer. J. Math. 130 (2008), no. 4, 1067\textendash{}1085.

\bibitem{p-s-s}Phong, D.H., Sesum, N; Sturm, J: Multiplier ideal
sheaves and the K\textasciidieresis{}ahler-Ricci flow. arXiv: math.DG/0611794,
to appear in Commun. Anal. Geom. (2007)

\bibitem{p-w}Protter, M.H.; Weinberger, H.F.: Maximum-principles
in differential equations. Englewood Cliffs, N.J.: Prentice-Hall,
Inc., (1967).

\bibitem{ru}Rubinstein, Y.A: Some discretizations of geometric evolution
equations and the Ricci iteration on the space of Kahler metrics.
Adv. Math. 218 (2008), 1526-1565. 

\bibitem{sa}Sano, Y: Numerical algorithm for finding balanced metrics.
Osaka J. Math. 43 (2006), no. 3, 

\bibitem{sc}Schumacher, G: Positivity of relative canonical bundles
for families of canonically polarized manifolds. arXiv:0808.3259 

\bibitem{si0}Siu, Y. T., Curvature of the Weil-Petersson metric in
the moduli space of compact Kahler- Einstein manifolds of negative
first chern class, Aspects of Math. 9: Friedr. Vieweg \& Sohn, Braunschweig/Wiesbaden
1986 pp. 261-298.

\bibitem{si}Siu, Y-T: Invariance of plurigenera. Invent. Math. 134
(1998), no. 3, 661--673. 

\bibitem{s-t}Song, J;Tian, G: Canonical measures and Kahler-Ricci
flow. arXiv:0802.2570

\bibitem{s-w}Song, J; Weinkove, B: Constructions of Kahler-Einstein
metrics with negative scalar curvature. arXiv:0704.1005 

\bibitem{ti0}Tain, G: Smoothness of the Universal Deformation Space
of Calabi-Yau Manifolds and its Petersson-Weil Metric. In Mathematical
aspects of string theory (San Diego, Calif., 1986), 629\textendash{}646.
Adv. Ser. Math. Phys. 1, World Sci. Publishing, Singapore, 1987.

\bibitem{ti}Tian, Gang Canonical metrics in Kähler geometry. Notes
taken by Meike Akveld. Lectures in Mathematics ETH Zürich. Birkhäuser
Verlag, Basel, 2000. vi+101 pp.

\bibitem{t-z}Tian, G; Zhu, X: Convergence of Kähler-Ricci flow. J.
Amer. Math. Soc. 20 (2007), no. 3,

\bibitem{to}Todorov, A. N: The Weil-Petersson geometry of the moduli
space of SU($\geq3)$ (Calabi-Yau) manifolds. Comm. Math. Phys. 126
(1989), 325\textendash{}346.

\bibitem{ts}Tsuji, H: Dynamical construction of Kähler-Einstein metrics.
arXiv:math/0606626

\bibitem{ts2}Tsuji, H.: Canonical singular hermitian metrics on relative
canonical bundles; arXiv:0704.0566.

\bibitem{v}Voisin, C: Hodge theory and complex algebraic geometry.
I. Translated from the French by Leila Schneps. Reprint of the 2002
English edition. Cambridge Studies in Advanced Mathematics, 76. Cambridge
University Press, Cambridge, 2007. x+322 pp.

\bibitem{wa}Wang, X: Canonical metrics on stable vector bundles.
Comm. Anal. Geom. 13 (2005), no. 2, 253--285.

\bibitem{yau}Yau, S. T.: On the Ricci curvature of a compact Kähler
manifold and the complex Monge-Ampère equation. I. Comm. Pure Appl.
Math. 31 (1978), no. 3, 339--411.

\bibitem{yau2}Yau, S. T.: Non-linear analysis in geometry. Enseignement
Math. (33) (1986) 109-158

\bibitem{z}Zelditch, S: Book review of {}``Holomorphic Morse inequalities
and Bergman kernels Journal'' (by Xiaonan Ma and George Marinescu)
in Bull. Amer. Math. Soc. 46 (2009), 349-361. 
\end{thebibliography}
\end{document}